\documentclass[11pt, leqno]{amsart}

\usepackage[normalem]{ulem}
\usepackage[titletoc]{appendix}
\usepackage{lineno}

\usepackage{graphicx, psfrag, epstopdf}

\usepackage{booktabs}   
\usepackage{hyperref}
\usepackage{comment}
\usepackage[capitalize,nameinlink]{cleveref}

\usepackage{dsfont}
\usepackage{amsmath}
\usepackage[fleqn]{mathtools}
\usepackage{amssymb}
\usepackage{graphicx}
\usepackage{tikz}

\usepackage{amssymb}

\usepackage{amsmath}

\usepackage{amsfonts}

\usepackage[utf8]{inputenc}
\usepackage[T1]{fontenc}
\usepackage{mathpazo}
\usepackage{color}
\usepackage{yfonts}

\usepackage{paralist}

\usepackage{stmaryrd}
\usepackage{enumitem}

\usepackage{autonum} %

\def\eps{\varepsilon}
\def\diam{\mathrm{diam}}

\newcommand{\cR}{\mathcal R}

\newcommand{\norm}[1]{\| #1\|}

\def\cI{          \mathcal I}

\def\cJ{          \mathcal J}

\def\cB{          \mathcal B}
\def\cC{          \mathcal C}
\def\cD{          \mathcal D}
\def\cF{          \mathcal F}
\def\cG{          \mathcal G}

\definecolor{orange}{rgb}{1,.5,0}
 
\definecolor{olive}{rgb}{0.5,0.5,0}

\def\br{        \mathbf{r}}

\let\cal\mathcal

\def \R{{\mathbb R}}

\def \Z{{\mathbb Z}}
\def \N{{\mathbb N}}

\newcommand{\prf}{{\begin{proof}}}
	\newcommand{\epf}{{\end{proof}}}

\newcommand{\bw}{{\mathbf w}}

\newcommand{\cS}{{\mathcal S}}

\newcommand{\cP}{{\mathcal P}}

\newcommand{\cK}{{\mathcal K}}

\newcommand{\ary}{\begin{eqnarray}}
	\newcommand{\eary}{\end{eqnarray}}

\newcommand{\aryst}{\begin{eqnarray*}}
	\newcommand{\earyst}{\end{eqnarray*}}

\newcommand{\enmt}{\begin{enumerate}}
	\newcommand{\eenmt}{\end{enumerate}}

\definecolor{adjcol}{rgb}{0.55,0.00,0.55}

\providecommand{\Tube}{\mathcal{N}}
\providecommand{\bw}{\mathbf w}
\providecommand{\cW}{\mathcal W}
\providecommand{\cB}{\mathcal B}
\providecommand{\cC}{\mathcal C}

\usepackage{mathptmx}
\usepackage{bm}

\theoremstyle{plain}
\newtheorem{thm}{\bf Theorem}[section]

\newtheorem{lemma}[thm]{\bf Lemma}
\newtheorem{prop}[thm]{\bf Proposition}

\newtheorem{cor}[thm]{\bf Corollary}

\theoremstyle{definition}
\newtheorem{defi}[thm]{Definition}

\theoremstyle{remark}
\newtheorem{rema}[thm]{Remark}

\theoremstyle{definition}

\def\bee{\begin{equation}}
	\def\eee{\end{equation}}

\newif\ifA

\newlist{properties}{enumerate}{1}
\setlist[properties]{label=\textup{(P\arabic*)}, ref=\textup{(P\arabic*)}}

\def\Dil{\mathrm{S}}

\def\supp{{\rm supp}}

\def\cJ{\mathcal{J}}

\def\dist{ {\rm dist} }

\def\bd{{\mathbf{d}}}

\newcommand{\GL}{{\rm GL}}

\newcommand{\Gr}{{\rm Gr}}

\newcommand{\bE}{{\bf E}}

\newcommand{\bi}{{\bf i}}

\newcommand{\bj}{{\bf j}}

\newcommand{\bp}{{\bf p}}

\newcommand{\cM}{\mathcal{M}}

\newcommand{\dimH}{\dim_{\mathrm H}}
\newcommand{\dimB}{\dim_{\mathrm B}}
\newcommand{\dimaff}{\dim_{\mathrm{aff}}}

\definecolor{blue}{rgb}{0,0,1}
\definecolor{red}{rgb}{1,0,0}
\definecolor{green}{rgb}{0,.7,0}
\definecolor{brown}{rgb}{0.6,0.3,0.1}

\definecolor{revcolor}{rgb}{0.56,0,1}   %
\definecolor{delcolor}{gray}{0.5}       %

\DeclareFontFamily{U}{solomos}{}
\DeclareErrorFont{U}{solomos}{m}{n}{10}
\DeclareFontShape{U}{solomos}{m}{n}{
	<-> s*[1.1]  gsolomos8r
}{}

\begin{document}
	\bibliographystyle{alpha}
	
	\title[Dimension of affine IFS]{Dimension of affine iterated function systems on $\R^d$}

	\author{Jialun Li}
	\address{Jialun Li,  School of Mathematical Sciences, Fudan University, No. 220 Handan Road, Shanghai 200433, China}
	\email{  jialunli@fudan.edu.cn}

	\author{Wenyu Pan}
	\address{Wenyu Pan, Department of Mathematics, University of Toronto, 40 St George St, Toronto ON,	M5S 2E4, Canada}
	\email{wenyup.pan@utoronto.ca}
	
	\author{Zhiyuan Zhang}
	
	\address{Zhiyuan Zhang,
		Department of Mathematics, 
		Imperial College London, 
		London SW7 2AZ, 
		United Kingdom 
	}
	\email{zhiyuan.zhang@imperial.ac.uk} 		
	\begin{abstract}
		Let $\{f_i(x)=M_ix+a_i\}_{i\in I}$ be an affine iterated function system on $\R^d$ with attractor $\Lambda$, and let $\mu$ be an associated self-affine measure. Assume that the $M_i$ generate a Zariski dense subgroup of $\GL_d(\R)$. Under the strong open set condition, we show that $\dimH\Lambda$ equals the affinity dimension and that $\dimH\mu$ equals the Lyapunov dimension. 
		
		The above results are consequences of the following projection theorem, which requires no separation condition. For every $1\leq n\leq d$ and for Furstenberg measure a.e. $(d-n)$-dimensional subspace $V$, 
		\[\dimH (P^V)_*\mu=\min\{n,\,\dimH\mu\}.\]
		This projection theorem in fact holds under hypotheses weaker than Zariski density: simplicity of the Lyapunov spectrum and a mild transversality condition (the Oseledets non-concentration property). The source of dimension increment in our proof is the positive dimension of the fiberwise conditional measures,
		exploited through Wu's one-dimensional entropy increase theorem. A key ingredient in reducing the
		higher-dimensional problem to this setting is a new structural result: a slab-filling property of the measure.
	\end{abstract}
	
	\maketitle

	\tableofcontents

	\section{Introduction}
	An affine iterated function system (IFS) on \(\mathbb R^d\) is a finite family of affine maps \(\Phi=\{f_i(x)=M_i x+a_i\}_{i \in I} \), with invertible contracting matrices \(M_i \in \GL_{d}(\R)\) and translations \(a_i \in \R^{d} \).  
	Hutchinson \cite{Hut} proved that there is a unique nonempty compact set \(\Lambda\) satisfying \(\Lambda=\bigcup_i f_i(\Lambda)\). This set is called the attractor of the system or a self-affine set. A central problem in fractal geometry is to determine the Hausdorff dimension of \(\Lambda\).  The study of this question has revealed deep connections between geometric measure theory and dynamical systems.

	When the linear parts are conformal, the attractor $\Lambda$  is called a self-similar set. The classical results of Moran and Hutchinson in \cite{Moran1946,Hut} provide a simple formula for the Hausdorff dimension of self-similar sets under the open set condition.  Overlaps, however, introduce subtler arithmetic phenomena, already seen in Bernoulli convolutions \cite{Erdos1939, Solomyak1995}. See  \cite{hochman_self-similar_2014, hochman_self-similar_2015} for Hochman's recent breakthrough on self-similar measures.

	For a general affine IFS, contraction rates may vary with direction. 
	The Bedford-McMullen carpets \cite{Bedford1984,McMullen1984} show that the Hausdorff dimension can be strictly smaller than the box dimension, even under the open set condition. Their Hausdorff dimension can also depend on the translation vectors, even when the linear parts are fixed. In these examples, the linear parts preserve the coordinate directions. This motivates the irreducibility assumptions used to rule out such structures in more general dimension formulas.

	\subsection{Dimension of self-affine attractors}
	
	Falconer \cite{falconer_hausdorff_1988} introduced the \emph{affinity dimension}, defined in terms of the singular values of matrix products, and proved that it gives an upper bound for the Hausdorff dimension. He also showed that if each linear part has operator norm less than \(\frac{1}{3}\), then this bound is attained for Lebesgue-almost every choice of translations. 
	The main question is therefore to determine when a specific system attains this bound. 
	For partial results in this direction, see \cite{LalleyGatzouras1992,hueter_falconers_1995,Solomyak1998,KaenmakiShmerkin2009,BK}.
	
	A decisive advance came from B\'ar\'any, Hochman, and Rapaport in \cite{barany_hausdorff_2017}. Under a separation condition called the \emph{strong open set condition} (SOSC; see \cref{def SSCSOSC}), and assuming that the normalized matrices \(M_i/\sqrt{|\det M_i|}\) generate a noncompact and totally irreducible group, they proved that for $d = 2$, 
	\[
	\dimH \Lambda=\dimB \Lambda=\dimaff\Phi.
	\]
	Here, $\dimH,\dimB$ and $\dimaff$ denote the Hausdorff, box, and affinity dimensions, respectively. The case $d=3$ was settled by Rapaport in \cite{Rapaport24}.
	These developments suggest the following folklore conjecture:
	\begin{quote}
		\emph{
			Under suitable irreducibility conditions on the linear parts and a separation condition, the Hausdorff dimension of a self-affine set should equal its affinity dimension.
		}
	\end{quote}

	Our first main result confirms this conjecture for every $d$. 
	
	\begin{thm}\label{thm affinity}
		Let $\Phi$ be an affine IFS with   Zariski dense linear parts, and assume that $\Phi$ satisfies SOSC. Then the Hausdorff dimension of its attractor $\Lambda$ is equal to its affinity dimension, that is,
		\[\dim_{\mathrm{H}} \Lambda=\dim_{\mathrm{aff}}\Phi.\]
	\end{thm}

	\subsection{Dimension of self-affine measures}
	One way to obtain lower bounds for the Hausdorff dimension of the attractor is to study the dimension of measures supported on it. We follow this approach by studying self-affine measures. Given weights \(p_i>0\) with \(\sum_i p_i=1\), the associated self-affine  measure is the unique Borel probability measure $\mu$ satisfying
	\[
	\mu=\sum_{i \in I} p_i(f_i)_*\mu,
	\]
	where \((f_i)_*\mu\) denotes the  pushforward of $\mu$ under $f_i$. 
	
	Henceforth, by an affine IFS, we always mean a tuple
	\aryst
	\cF = \big( \{f_i\}_{i \in I},\, (p_i)_{i \in I} \big),
	\earyst
	consisting of a family of affine contractions $\{f_i\}_{i \in I}$ and a positive probability vector $\bp = (p_i)_{i \in I}$.
	
	For self-affine measures, a central conjecture asserts the following:

	\begin{quote}
		\emph{Dimension is determined by entropy and Lyapunov exponents.}
	\end{quote}
	Here, the Lyapunov exponents are the numbers
	$0>\chi_1\geq \cdots\geq \chi_d$
	describing the exponential growth rates of the singular values of random products of the linear parts driven by $\bp$.  
	Two notions of entropy are relevant to this conjecture.
	The first one is the Shannon entropy of the probability vector, $H(\bp)= - \sum_i p_i\log p_i$. 
	The second is the Furstenberg entropy $H(\cF)$, which we define in \eqref{eq entropy F} after some preparation. It is known that $H(\cF) \leq \sum_{i}|\chi_i|$ always holds, and that $H(\cF) = H(\bp)$ under SOSC (\cref{def SSCSOSC}).
	
	The Lyapunov dimension of $\mu$ is
	\[ \dim_{\rm L}\mu=  \min\left\{m+\frac{H(\bp)+\sum_{1\leq j\leq m}\chi_j}{ |\chi_{m+1}|},d\right\},  \]
	where $m\leq d-1$ is the largest integer satisfying $\sum_{1\leq j\leq m}|\chi_j|\leq H(\bp)$.

	In \cite{barany_hausdorff_2017} and \cite{Rapaport24}, the authors confirmed that, under SOSC, the dimension of $\mu$ equals its Lyapunov dimension when $d=2,3$. Our second main result proves this for all $d$. This result follows as a special case of the general formula below, which expresses the Hausdorff dimension in terms of the Furstenberg entropy and the Lyapunov exponents. This formula holds without any separation conditions.
	
	\begin{thm}\label{main thm}
		Let $\cF$ be an affine IFS on $\R^d$ with  Zariski dense linear parts, and let $\mu$ be its self-affine measure. Then 
		\[ 
		\dim_{\mathrm{H}}\mu =     m+\frac{H(\cF)+\sum_{1\leq j\leq m}\chi_j}{ |\chi_{m+1}|},  \]
		where $m\leq d-1$ is the largest integer  satisfying  $\sum_{1\leq j\leq m}|\chi_j|\leq H(\cF)$.
		
		In particular, if in addition $\cF$ satisfies SOSC, then 
		\aryst
		\dim_{\mathrm{H}}\mu = \dim_{\mathrm{L}}\mu.
		\earyst
	\end{thm}
	
	Combining \Cref{main thm} with the variational principle of Morris and Sert \cite[Theorem 1.4]{morris_variational_2023}, which relates the dimension of a self-affine set to the dimensions of its self-affine measures, yields \Cref{thm affinity}.

	\subsection{Projection theorems without separation conditions}
	
	A key tool is the Ledrappier–Young formula for self-affine measures, an analog of the classical formula from smooth dynamics. It was established by B\'ar\'any and K\"aenm\"aki in \cite{BK} and by Feng in \cite{fengDimensionInvariantMeasures2023}, and played an important role in both  \cite{barany_hausdorff_2017, Rapaport24}.

	The formula reduces the computation of $\dimH \mu$ to a projection theorem, and proving such a theorem is the main task. For $d=2$, B\'ar\'any, Hochman and Rapaport \cite[Theorem~1.3]{barany_hausdorff_2017} proved that, under the {\it exponential separation condition} (ESC) and condition \eqref{cond star} on the linear parts,
	\[
	\dim_{\mathrm{H}}(P^V\mu)=\min\{1,\dim_{\mathrm{H}}\mu\}
	\quad\text{for Furstenberg-almost every line }V.
	\]
	Here, \(P^V\) denotes the orthogonal projection onto $V^{\perp}$, and the Furstenberg measure describes the distribution of directions arising from the transposed random matrix products. ESC requires that the affine maps corresponding to distinct words of length \(n\) be at least \(c^n\) apart, for some fixed $c \in (0, 1)$ and every $n$.  Rapaport \cite{Rapaport24} subsequently proved the corresponding projection theorem for $d=3$, again under ESC and condition \eqref{cond star}.  
	
	Since ESC requires separation only in the affine group, it allows the sets $f_i(\Lambda)$ to overlap. Verifying ESC for a {\it deterministic} system with overlaps can nevertheless be difficult. 
	To our knowledge, the only systematic method relies on the algebraicity of the matrix entries; see \cite[Lemma 6.1]{HochmanSolomyak2017}.
	
	Our main projection theorem holds in any dimension and without any separation conditions.
	\begin{thm} \label{main thm2}
		Let $\cF$ be as in \cref{main thm}. Then, for every $1 \leq m \leq d$ and   $\nu_m^{+}$-a.e. $V \in \Gr(d, d-m+1)$, \footnote{ Here, $\nu_{m}^{+}$ is the distribution of the most contracting $(d-m+1)$-dimensional subspace; see \cref{sec preliminary}.}
		\aryst
		\dimH (P^{V})_* \mu = \min\{m - 1, \dimH \mu \}. 
		\earyst
	\end{thm}
	It is worth mentioning that, for $d=2$, under ESC and condition \eqref{cond star}, Hochman and Rapaport proved in \cite{hochman_hausdorff_2021} that $\dim_{\rm H}\mu = \min\{2, \dim_{\rm L}\mu\}$.
	By Theorem \ref{main thm}, this is equivalent to $H(\cF)=\min\{H(\bp),|\chi_1|+|\chi_2|\}$. Our method describes the dimension of $\mu$ and its projections in terms of $H(\cF)$ and the Lyapunov exponents, but it says nothing directly about $H(\cF)$ itself. Hence, for the moment, this identity is out of reach for our method.
	
	For the dimensions of projections of self-affine sets and measures with generic translation vectors, see \cite{feng_dimensions_projected_2023,feng_dimensions_orthogonal_2025,morris_projections_2026}. 
	
	\subsection{Main technical result}
	\label{subsec:main tech result}

	To state our main technical result, we need to introduce several additional objects.
	
	Let $(\Sigma, \sigma)$ be the two-sided full shift over the alphabet $I$; that is,
	$\Sigma = I^{\Z}$, and $\sigma \colon \Sigma \to \Sigma$ is the shift map
	$\sigma\big( (i_n)_{n \in \Z} \big) = (i_{n+1})_{n \in \Z}$.
	Let $\nu_0 \in \cP(I)$ be the probability measure given by $\nu_0(\{i\}) = p_i$ for
	$i \in I$, and let $\nu = \nu_0^{\otimes \Z} \in \cP(\Sigma)$ be the associated
	Bernoulli measure. Then $\nu$ is $\sigma$-invariant and ergodic.

	Given $\bi = (i_k)_{k \in \Z} \in \Sigma$, we write $\bi_{+} = (i_k)_{k \geq 0}$ and
	$\bi_{-} = (i_k)_{k < 0}$ for its forward and backward halves, and we set
	$\Sigma_{+} = I^{\Z_{\geq 0}}$ and $\Sigma_{-} = I^{\Z_{< 0}}$, so that
	$\bi_{\pm} \in \Sigma_{\pm}$. We write $\nu_{\pm}$ for the projection of $\nu$ to $\Sigma_{\pm}$. For each integer $n \geq 1$, we set
	\aryst
	M_{\bi_{+}}^{n} = M_{i_{n-1}} \cdots M_{i_0}, \qquad
	M_{\bi_{-}}^{-n} = M_{i_{-n}}^{-1} \cdots M_{i_{-1}}^{-1}.
	\earyst
	Given a matrix $A\in \GL_d(\mathbb{R})$, we denote by $\alpha_1(A) \geq \alpha_2(A) \geq \cdots \alpha_d(A)$ the singular values of $A$.
	
	\begin{thm}[Oseledets]\label{thm Oseledets}
		Let $\cF = \big( \{ f_i \}_{i \in I}, (p_i)_{i \in I} \big)$ be an affine IFS on $\R^d$, and let $\nu$ be its Bernoulli measure. Then there exist an
		integer $k \in \{1, \dots, d\}$, integers $d_1, \dots, d_k \geq 1$ with
		$d_1 + \cdots + d_k = d$, and real numbers $0 > \chi_{1} > \cdots > \chi_{k}$, such that
		for $\nu$-almost every $\bi = (i_n)_{n \in \Z} \in \Sigma$, the following hold.
		\enmt
		\item For every $m \in \{1, \dots, d\}$,
		\aryst
		\lim_{n \to \infty} \frac{1}{n} \log \alpha_m \big( M_{\bi_{+}}^{n} \big) = \chi_j
		\quad \text{whenever } d_1 + \cdots + d_{j-1} < m \leq d_1 + \cdots + d_j.
		\earyst
		
		\item There is a flag
		\aryst
		\{0\} = V_{k+1}(\bi_{+}) \subsetneq V_{k}(\bi_{+}) \subsetneq \cdots
		\subsetneq V_{1}(\bi_{+}) = \R^d ,
		\earyst
		depending measurably on $\bi$ and only through $\bi_{+}$, with $\dim V_{j}(\bi_{+}) = d_{j} + \cdots + d_{k}$ for
		$j \in \{1, \dots, k\}$, such that for every $j \in \{1, \dots, k\}$,
		\begin{align}
			M_{i_0} V_{j}(\bi_{+}) &= V_{j}\big( (\sigma \bi)_{+} \big),\\
			\lim_{n \to \infty} \frac{1}{n} \log \norm{ M_{\bi_{+}}^{n} v } &= \chi_{j},
			\qquad v \in V_{j}(\bi_{+}) \setminus V_{j+1}(\bi_{+}) .
		\end{align}
		
		\item There is a flag
		\aryst
		\{0\} = U^{0}(\bi_{-}) \subsetneq U^{1}(\bi_{-}) \subsetneq \cdots
		\subsetneq U^{k}(\bi_{-}) = \R^d ,
		\earyst
		depending measurably on $\bi$ and only through $\bi_{-}$, with $\dim U^{j}(\bi_{-}) = d_{1} + \cdots + d_{j}$ for
		$j \in \{1, \dots, k\}$, such that for every $j \in \{1, \dots, k\}$,
		\begin{align}
			M_{i_{-1}}^{-1} U^{j}(\bi_{-})& = U^{j}\big( (\sigma^{-1} \bi)_{-} \big),\\
			\lim_{n \to \infty} \frac{1}{n} \log \norm{ M_{\bi_{-}}^{-n} v } &= -\chi_{j},
			\qquad v \in U^{j}(\bi_{-}) \setminus U^{j-1}(\bi_{-}) .
		\end{align}
		
		\item Set $E_{j}(\bi) = V_{j}(\bi_{+}) \cap U^{j}(\bi_{-})$ for $j \in \{1, \dots, k\}$.
		Then $\dim E_{j}(\bi) = d_{j}$, and
		\aryst
		V_{j}(\bi_{+}) =  \bigoplus_{l \geq j} E_{l}(\bi) \quad \text{and} \quad
		U^{j}(\bi_{-})  = \bigoplus_{l \leq j} E_{l}(\bi) .
		\earyst
		Moreover $M_{i_0} E_{j}(\bi) = E_{j}(\sigma \bi)$, and for $v \in \R^d \setminus \{0\}$
		and $j \in \{1, \dots, k\}$, one has $v \in E_{j}(\bi)$ if and only if
		\begin{equation}
			\lim_{n \to \infty} \frac{1}{n} \log \norm{ M_{\bi_{+}}^{n} v } = \chi_{j}
			\quad \text{and} \quad
			\lim_{n \to \infty} \frac{1}{n} \log \norm{ M_{\bi_{-}}^{-n} v } =  -\chi_{j} .
		\end{equation}
		
		\item Given a non-empty proper subset $J \subset \{1, \dots, k\}$, write
		$E_{J}(\bi)$ for the sum of the $E_{j}(\bi)$ with $j \in J$. Then
		\aryst
		\lim_{n \to \pm \infty} \frac{1}{n}
		\log \sin \measuredangle \big( E_{J}(\sigma^{n} \bi), \, E_{J^{c}}(\sigma^{n} \bi) \big)= 0.
		\earyst
		\eenmt
	\end{thm}
	Parts (1)–(4) in the
	form stated here are \cite[Thm.~3.4.11]{Arnold}, and part (5) is \cite[Addendum~5]{Bochi}.

	The numbers $0 > \chi_1 > \cdots > \chi_k$ given by \cref{thm Oseledets} are called
	the {\it Lyapunov exponents} of $\cF$, and $\{ \chi_1, \dots, \chi_k \}$ the
	{\it Lyapunov spectrum} of $\cF$. We say that the spectrum is {\it simple} if $k = d$. The spectrum is simple whenever $\cF$ has Zariski dense linear parts
	(see \cite[Thm.~10.9]{benoist_random_2016}).

	Both \cref{main thm} and \cref{main thm2} follow quickly from a more technical statement, \cref{main technical}, which we state below after some preparation.

	We first define the separation conditions used in this paper.
	\begin{defi} \label{def SSCSOSC}
		Let $\cF$ be an affine IFS. We say that $\cF$ satisfies the {\it strong separation condition} (SSC) if the sets $f_i(\Lambda),i\in I$, are pairwise disjoint.
		
		We say that $\cF$ satisfies the {\it strong open set condition} (SOSC) if there exists a bounded open set $O$ such that $O\cap\Lambda\neq \emptyset$, $f_i(O)\subset O$ for every $i\in I$, and the sets $  f_i(O),i\in I$ are pairwise disjoint.
	\end{defi}

	\begin{defi} \label{def ONC}
		Let $\cF$ be an affine IFS on $\R^d$, and let $d_0 \in \{1, \cdots, d\}$.  We say that $\cF$ has the {\it Oseledets non-concentration property at level $d_0$} (ONC at level $d_0$) if, for every $m \in \{1, \cdots, d_0 - 1\}$, every $q \in \{m+1, \cdots, d_0\}$,  
		and $(\nu_-\otimes \nu_{+} \otimes \nu_{+})$-a.e. $(\bi_{-}, \bi_{+}, \bj_{+})$, we have 
		\aryst
		\measuredangle( V_{m+1}(\bj_{+}), U^{m-1}(\bi_{-}) \oplus E_{q}(\bi_-, \bi_{+}) ) > 0.
		\earyst
	\end{defi}
	For $q = m$, $U^{m-1}(\bi_{-}) \oplus E_{q}(\bi_-, \bi_{+}) = U^{m}(\bi_{-})$, and hence the above transversality is automatic.
	For $d=d_0=2$ (so that $m=1$ and $q=2$), the condition reduces to $\measuredangle( V_{2}(\bj_{+}), V_{2}(\bi_{+}) )  > 0$,  which is a form of irreducibility (or non-integrability). 
	
	The following is the main technical result of this paper.

	\begin{thm} \label{main technical}
		Let $\cF$ be an affine IFS on $\R^d$, and let $\mu$ be its self-affine measure.
		Assume that 
		\enmt
		\item $\cF$ satisfies SSC,
		\item the Lyapunov spectrum of $\cF$ is simple,
		\item $\cF$ has the ONC at level $d_0$  for some $d_0 \in \{1, \cdots, d\}$. 
		\eenmt
		Then for every $1 \leq m \leq d_0 + 1$, and  $\nu_m^{+}$-a.e.  $V \in \Gr(d, d-m+1)$, we have 
		\aryst
		\dimH (P^{V})_* \mu = \min\{m - 1, D_{d_0+1}\},
		\earyst
		where $D_{d_0+1}$ denotes the common value of $\dimH {P^{V_{d_0+1}}\mu}$ over $\nu_{d_0+1}^+$-a.e. $V_{d_0+1}\in \Gr(d,d-d_0)$ (by \cref{thm Exactdim}).
	\end{thm}
	The SSC assumption in
	\Cref{main technical} can in fact be removed.  The lifting argument in \cref{subsec ly}, which we use to deduce \cref{main thm} and \cref{main thm2} from \Cref{main technical}, applies here as well because the ONC property passes to the lifted IFS. This yields the following statement, the strongest form of our results.
	\begin{thm}
		\label{main cor}
		Let $\cF$ be an affine IFS on $\R^d$, and let $\mu$ be its self-affine measure. 
		Assume that 
		\enmt
		\item the Lyapunov spectrum of $\cF$ is simple,
		\item $\cF$ has the ONC at level $d_0$  for some $d_0 \in \{1, \cdots, d\}$. 
		\eenmt
		Then for every $1 \leq m \leq d_0 + 1$, and $\nu_m^{+}$-a.e.  $V \in \Gr(d, d-m+1)$, we have 
		\aryst
		\dimH (P^{V})_* \mu = \min\{m - 1, D_{d_0+1}\},
		\earyst
		where $D_{d_0+1}$ denotes the common value of $\dimH (P^{V_{d_0+1}})_*\mu$ over $\nu_{d_0+1}^+$-a.e. $V_{d_0+1}\in \Gr(d,d-d_0)$.
	\end{thm}

	The irreducibility condition used in \cite{barany_hausdorff_2017} and
	\cite{Rapaport24} can be summarized as \emph{proximality and strong
		irreducibility in all exterior powers}:
	\begin{equation}\tag{$\star$}\label{cond star}
		\parbox{0.8\textwidth}{For every $1 \leq k < d$, the induced action of the
			semigroup generated by $\{M_i\}_{i \in I}$ on $\wedge^k \R^d$ is proximal
			and strongly irreducible.}
	\end{equation}
	Under condition \eqref{cond star}, hypotheses (1) and (2) of \cref{main cor} hold automatically\footnote{Simple spectrum comes from proximality. Strong irreducibility implies that for each fixed $m$-dimensional subspace $W\subset \R^d$, $V_{m+1}(\bj_+)\cap W=\{0\}$ for $\nu_+$-almost every $\bj_+$.}. Since \cref{main cor} assumes no separation condition, it contains the results of \cite{barany_hausdorff_2017} and \cite{Rapaport24}. ONC also covers cases that condition \eqref{cond star} excludes, such as the triangular case which is treated separately in \cite{barany_hausdorff_2017}. ONC is thus a more natural transversality condition: it is compatible with the lifting argument that removes the separation hypothesis, and it retains exactly the non-integrability that the projection argument uses.

	\subsection{Idea of the proof of \Cref{main technical}}
	Let $\cF$ be as in \Cref{main technical}. By the Ledrappier-Young formula for self-affine measures (see \Cref{thm Exactdim}), for every $k\in \{1,\cdots, d+1\}$, there is a number $D_k$ such that $(P^{V_k(\bi_+)})_*\mu$ is exact dimensional with dimension $D_k$ for $\nu_+$-a.e. $\bi_+$. In particular, $D_{d+1}=\dim \mu$, and we have $D_k\leq D_{k+1}$ and $D_k\leq k-1$. In this notation, the conclusion of \Cref{main technical} is equivalent to
	\begin{equation}
		\label{eqn:LY reduction}
		D_{m+1}=\min\{m,\,D_{d_0+1}\}\qquad \text{for}\,\,1\leq m\leq d_0-1.
	\end{equation}

	Like many previous works on this topic, our proof is also based on a certain dimension increment argument. If \eqref{eqn:LY reduction} fails, then there is some $m \in \{1, \cdots, d_0-1\}$ with $$D_{m+1} < \min\{m,D_{d_0+1}\},$$ and we derive a contradiction by a bootstrapping argument, which increases the value of $D_{m+1}$. 
	
	The first key ingredient in our approach is a different source of dimension increment. 
	B\'ar\'any, Hochman and Rapaport \cite{barany_hausdorff_2017} express the projection of $\mu$ as a convolution of $\mu$ with a measure supported on the projected iterated affine maps. In their argument, dimension growth comes from the conditional entropy of this measure with respect to two partitions at polynomially comparable small scales. Under ESC, this entropy is positive when the dimension formula fails. In our approach, the source is the dimension of the {\it fiberwise conditional measure } $\mu_x^{V_{m+1}(\bi_{+})}$ on the fibers $x+V_{m+1}(\bi_+)$. Indeed, the gap between $D_{m+1}$ and $D_{d_0+1}$ ensures these measures have positive dimension: $\dimH \mu_x^{V_{m+1}(\bi_{+})} = D_{d_0+1}  - D_{m+1} > 0$ for typical $x\in \Lambda_+$ and $\bi_+\in \Sigma_+$.
	This new viewpoint is inspired by the recent breakthrough of Wu \cite{wuProjectionTheoremsCountably2025} on the exact overlaps conjecture, and was first implemented in the study of non-linear planar IFS in the recent work by the first two authors of this paper with Tong and Xu \cite{LPTX}. 
	
	The proof architecture is as follows, though the actual implementation deviates from it due to technical details.
	
	{\bf I}. 
	Let $C_0$ be a large constant depending only on $\cF$. We consider dimensions near the scale $2^{-C_0n}$ for large $n$. For a typical $\bi_+$,
	we find a typical box $Z$ of length $2^{-C_0 n}$ along $V_{m+1}$ and $2^{-(C_0+1)n}$ along $V_{m+1}^{\perp}$, so that $(P^{V_{m+1}(\bi_{+})})_* \mu^{Z}$ reflects the expected dimension $D_{m+1}$ at a scale not much smaller than $2^{-C_0n}$. We will always consider projections of component measures (i.e., the restriction of a measure to a small region, such as $\mu^Z$). We will use the exact dimensionality of Feng \cite{fengDimensionInvariantMeasures2023} for both the fiberwise measure and the projection measure.
	
	{\bf II}. We look for $\bj_{+}$ such that $V_{m+1}(\bj_{+})$ is close to, but distinct from, $V_{m+1}(\bi_{+})$, so that the projection of $\mu^{Z}$ along $V_{m+1}(\bj_{+})$ detects the dimension of $\mu^{V_{m+1}(\bi_{+})} $. The precise requirement on this projection measure is the {\it spreading property } of \cite{wuProjectionTheoremsCountably2025}. Both SSC and ONC enter at this step. Under SSC, the fiberwise measures are invariant under the backward dynamics, so spreading can be obtained by ergodic arguments; ONC provides the transversality needed to pass from the $\bi_+$-fiber to its $\bj_+$-projection. We would like to apply Wu's entropy increase theorem \cite{wuProjectionTheoremsCountably2025} (see \cref{thm Wu})  to conclude that $(P^{V_{m+1}(\bj_{+})})_* \mu^{Z}$ has dimension larger than $D_{m+1}$, which would give a contradiction.

	An immediate difficulty with this outline is that Wu's entropy increase theorem does not apply directly in dimensions greater than $1$: it is a sumset estimate on the line, bounding $|\bigcup_{a\in A}(a+B_a)|$ from below (see \cref{thm Wu}), and it requires only cardinality bounds on the sets $B_a$.  In higher dimensions, we also need information on the structure of these sets, since translates of subsets may concentrate in a low-dimensional subspace.

	To resolve this issue, we need the second key ingredient in our proof: a new structural result on the component measure in a tiny scale window.  For every $m$ with $D_m = m - 1$, 
	we show that near a typical $x\in \Lambda$, provided that $\rho^{1 + } < \delta < \rho$, the component measure in a slab of size $\rho$ along $U^{m-1}(x)$ and size $\delta$ along the orthogonal direction almost reaches the full dimension at scale $\delta$: one needs $(\rho/\delta)^{m-1-}$ many $\delta$-balls to cover a definite mass. We call this {\it slab-filling property}.
	We can find sufficiently many such slabs such that the dimension $D_{m+1}$ can be computed from the number of slabs. Moreover, we also have the H\"older continuity of $U^{m-1}(x)$ for $x$ in a large measure set. Hence, in a small ball of size $\rho$, we can assume all the slabs are parallel. This allows us to quotient out the direction $U^{m-1}(x)$, reducing it to the dimension-increment picture of the planar case. 

	A further difficulty arises in verifying the spreading property of fiber measures in higher dimensions, which is treated in \cref{sec proofofprop fiberwisedim}. The measure for which we need spreading is $(P^{V_{m+1}(\bj_+) + U^{m-1}(x)})_* \mu_{x,\bi_+}$, where $\mu_{x,\bi_+}$ is the restriction of $\mu_x^{V_{m+1}(\bi_{+})}$ to a small cube.
	In \cite{wuProjectionTheoremsCountably2025},  the spreading property follows from the properties of the CP process and Birkhoff's ergodic theorem.
	In the planar case \cite{LPTX}, this follows from the same property of the fiberwise measure itself, as $P^{V_{2}(\bj_{+})}$ embeds the support of $\mu_{x,\bi_+}$ into $V_{2}(\bj_{+})^{\perp}$. In higher dimensions, the situation is more delicate. The projection collapses the support of the fiberwise measure, so we must keep careful track of the angles between various Lyapunov subspaces associated with $\bi_{+}$ and $\bj_{+}$, and of their effect on projected measures. 
	
	Moreover, the boxes must be typical for two {\it nearby} $\bi_{+}$ and $\bj_{+}$ at the same time. 
	We therefore work with a non-trivial self-coupling measure $\zeta_n$ of $\nu_{+}$ on $\Sigma_+^2$, which allows us to find pairs $(\bi_+,\bj_+)$ for which all the relevant angles are close to the values predicted by the Lyapunov exponents. Another complication is that $\mu^{V_{m+1}(\bi_{+})}$ may be supported in the smaller subspace $V_{m_0+1}(\bi_{+})$ for some $m_0 > m$. This requires modifying the choice of $Z$.

	\subsection{Notation}
	
	We denote $\N = \{0, 1, 2, \cdots\}$.
	
	Given real numbers $a \in \R$ and $b > 0$, we denote by $a \pm b$ an unspecified number in $[a-b, a+b]$.
	
	\noindent{$\bullet$ \bf Vector spaces.}           
	Given a linear space $V$ equipped with a norm induced by an inner product.  Given a $\eta > 0$, we denote by $B^{V}(\eta)$ the open ball of radius $\eta$ in $V$ centered at the origin. With a slight abuse of notation, we will use $0$ to denote the origin in various normed linear spaces.
	
	Given a subspace $U \subset V$, we denote by $P_U : V \to U$ the orthogonal projection to $U$, and denote  $P^U = P_{U^{\perp}}$. We write $P^U\mu$ for the pushforward $(P^U)_{*}\mu$, and similarly for other operators.

	\noindent{$\bullet$ \bf Measures.}
	
	Given a complete separable metric space $X$, we denote by $\cM(X)$, resp. $\cP(X)$, the space of Borel measures, resp. Borel probability measures, on $X$.
	
	Let $\zeta$ be a probability measure on $X$, and let $Y \subset X$ be a Borel subset such that $\zeta(Y) > 0$. We  denote
	\aryst
	\zeta^Y = ( \zeta|_{Y} )  / \zeta(Y).
	\earyst

	\noindent{$\bullet$ \bf Implicit constants.} 
	
	Given $A\in \R$, $B>0$, and a parameter $q$ (potentially multidimensional), we write $A = O_{q}(B)$ or $A \lesssim_{q} B$  to mean that there is a positive constant $C$ depending  on $M$ and $q$ only such that $|A| \leq C B$.
	We also say that a statement involving $A,B$  is valid under the condition $A\ll_{q} B$ if it holds provided  $A\leq \eps B$ where $\eps>0$ is a small enough constant depending on $M$ and $q$ only. 
	When the above $C,\eps$ only depend on $M$, we do not write any subscript. We insist that the notations $O(\cdot)$, $\lesssim$, $\ll$.
	
	Given a parameter $q$, we write $C=C(q)$ to indicate that a certain constant $C$ depends on $q$ only.  
	
	\subsection*{Acknowledgement}
	Part of this work was carried out while W.P. and Z.Z. were visiting Fudan University in July 2026, and they thank the university for its hospitality. J.L. also thanks the Institut Mittag-Leffler and the program organizers for their hospitality during a stay in which part of this paper was written. W.P. is supported by NSERC Discovery Grant RGPIN-2025-06706.
	
	\subsection*{AI disclosure}
	
	The authors used LLM tools to search the literature, check grammar, and improve the writing of this manuscript. The authors drafted parts of \Cref{app miniscule} with the help of LLM tools, working from the authors' ideas and outlines, and similar results have been established in \cite{LPTX}. LLM tools were also used to assist with the linear algebra computation in \Cref{app linear algebra}. All mathematical ideas and results are due to the authors, who take full responsibility for the content.

	\section{Preliminaries} \label{sec preliminary}
	
	\subsection{The setup}
	Throughout this section, $\cF = ( \{ f_i(x) = M_i x + a_i \}_{i \in I}, (p_i)_{i \in I} )$ is an affine IFS on $\R^d$ with  a simple Lyapunov spectrum.
	We do {\em not} assume that $\cF$ has SSC; the additional structure available under SSC is collected in \cref{subsec SSC}, and the ONC at level $d_0$ is assumed only where it is explicitly stated.
	Since the statements of \cref{main thm}, \cref{main thm2} and \cref{main technical} are homothety invariant, we may assume without loss of generality that $f_i (B^{\R^d}(1/2)) \Subset B^{\R^d}(1/2)$ for every $i \in I$, and consequently $\Lambda \Subset B^{\R^d}(1/2)$.
	
	\subsection{Grassmannians, flags and slabs} \label{subsec grassmannianflagsslabs}
	
	Equip $\R^d$ with a standard euclidean norm. For non-zero $u, v \in \R^d$, let 
	\aryst
	\measuredangle(u, v) = \arccos \frac{\langle u, v \rangle}{\|u\| \, \|v\|} \in [0, \pi / 2].
	\earyst
	For a non-zero vector $u$ and a non-zero subspace $W \subset \R^d$, let
	\aryst
	\measuredangle(u, W) = \min_{w \in W \setminus \{0\}} \measuredangle(u, w) \in [0, \pi/2].
	\earyst
	For non-zero subspaces $V, W \subset \R^d$, the {\em minimal angle} between $V$ and $W$ is
	\aryst
	\measuredangle(V, W) = \min_{v \in V \setminus \{0\}} \measuredangle(v, W)
	= \min_{\substack{v \in V, \ w \in W \\ \|v\| = \|w\| = 1}} \measuredangle(v, w).
	\earyst
	We also set
	\aryst
	\bd(V, W) = \begin{cases}
		\max_{v \in V \setminus \{0\}} \sin \measuredangle(v, W) & \mbox{if } \dim V \leq \dim W, \\
		\max_{w \in W \setminus \{0\}} \sin \measuredangle(w, V) & \mbox{if } \dim W \leq \dim V,
	\end{cases}
	\earyst
	which is equivalent to the Hausdorff distance as subsets in the projective space $\mathbb{P}(\R^d)$.
	
	For a linear subspace $F \subset \R^d$, $x \in \R^d$ and $r \in (0, \infty]$, we write
	\aryst
	B^{F}(r) = \{ v \in F : \|v\| < r \}, \qquad B^{F}(x, r) = x + B^{F}(r),
	\earyst
	so that $B^{F}(\infty) = F$. For $k \in \{1, \cdots, d-1\}$, $V \in \Gr(d, k)$ and $r > 0$, we write 
	\aryst
	B(V, r) = \{ V' \in \Gr(d, k) : \bd(V, V') < r \}.
	\earyst
	We will abbreviate $B^F(x, r)$ as $B(x, r)$ when there is no ambiguity about the underlying space $F$, or such choice is not important.
	
	A {\em partial flag} in $\R^d$ is an ordered tuple $\bE = (E_i)_{i = 1}^{q}$ of linear subspaces of $\R^d$ such that 
	\aryst
	\{0\}\subsetneq E_{1} \subsetneq \cdots \subsetneq E_{q} = \R^d.
	\earyst
	Setting $E_0 = \{0\}$, its {\em successive complements} are the subspaces
	\aryst
	F_j = E_j \cap E_{j-1}^{\perp}, \qquad j \in \{1, \cdots, q\}.
	\earyst
	Given $x \in \R^d$, a partial flag $\bE$ and $\br = (r_1, \cdots, r_q) \in (0, \infty]^{q}$, the {\em $(x, \bE, \br)$-slab} is
	\begin{equation}
		\label{eqn:slab definition}
		S(x, \bE, \br):= x + B^{F_1}(r_1) + \cdots + B^{F_q}(r_q).
	\end{equation}

	\subsection{Symbolic notation}
	Recall from  \Cref{subsec:main tech result} the two-sided full shift
	$(\Sigma, \sigma)$, the Bernoulli measure $\nu$, the spaces $\Sigma_{\pm}$
	with the measures $\nu_{\pm}$, and the decomposition
	$\bi = (\bi_{-}, \bi_{+})$, under which $\Sigma = \Sigma_{-} \times \Sigma_{+}$
	and $\nu = \nu_{-} \otimes \nu_{+}$.
	By a slight abuse of notation, we denote the one-sided shift factor of
	$\sigma$ by $(\Sigma_{+}, \sigma)$ and the one-sided shift factor of
	$\sigma^{-1}$ by $(\Sigma_{-}, \sigma^{-1})$.
	For $\bi_{-} \in \Sigma_{-}$, $\bi_{+} \in \Sigma_{+}$ and $i \in I$, we write
	$\bi_{-} i \in \Sigma_{-}$ for the sequence $\bi_{-}$ followed by $i$, and
	$i \bi_{+} \in \Sigma_{+}$ for the sequence $\bi_{+}$ preceded by $i$; thus
	\aryst
	\sigma(\bi_{-}, \bi_{+}) = (\bi_{-} i_0, \ \sigma \bi_{+}), \qquad
	\sigma^{-1}(\bi_{-}, \bi_{+}) = (\sigma^{-1} \bi_{-}, \ i_{-1} \bi_{+}).
	\earyst
	
	Given an integer $n > 0$ and $(i_{1}, \cdots, i_{n} ) \in I^n$, we denote
	\begin{equation}
		\label{eqn:f iterate}
		f_{(i_1, \cdots, i_n)} = f_{i_n} \circ \cdots \circ f_{i_1}. 
	\end{equation} 
	For each $\bi_{-} \in \Sigma_{-}$, $\bi_{+} \in \Sigma_{+}$ and $n > 0$, we denote
	\aryst
	f_{\bi_{+}}^{n} = f_{i_{n-1}} \circ \cdots \circ f_{i_{0}}, \qquad f_{\bi_{-}}^{-n} = f_{i_{-n}}^{-1} \circ \cdots \circ f_{i_{-1}}^{-1}, \qquad
	\bi_{+}[n] = (i_{0}, \cdots, i_{n-1}).
	\earyst
	We define the cylinder sets in $\Sigma_{-}$ and the canonical partition $\cP$ of $\Sigma$ by
	\begin{equation}
		\label{eqn:partition Lambda}
		[i] = \{ \bi_{-} \in \Sigma_{-} : i_{-1} = i \}, \quad i \in I, \qquad \cP = \{ [i] \times \Sigma_{+} : i \in I \}.
	\end{equation}
	
	The coding map $\Pi : \Sigma_{-} \to \Lambda$,
	\aryst
	\Pi( \bi_{-} ) = \lim_{n \to \infty} f_{i_{-1}} \circ \cdots \circ f_{i_{-n}}(0) = \lim_{n \to \infty} f_{\bi_{-}[n]}(0),
	\earyst
	is a continuous surjection with $\Pi_{*} \nu_{-} = \mu$, and it satisfies
	\begin{equation}
		\label{eq coding equivariance}
		\Pi(\bi_{-}) = f_{i_{-1}} \big( \Pi(\sigma^{-1} \bi_{-}) \big), \qquad \bi_{-} \in \Sigma_{-}.
	\end{equation}
	
	Let $\hat\mu = \mu \times \nu_{+}$ and, with a slight abuse of notation, let
	\aryst
	\sigma : \Lambda \times \Sigma_{+} \to \Lambda \times \Sigma_{+}, \qquad \sigma(x, \bi_{+}) = (f_{i_0}(x), \sigma \bi_{+}).
	\earyst
	By \eqref{eq coding equivariance}, $\Pi \times \mathrm{id} : \Sigma \to \Lambda \times \Sigma_{+}$ satisfies $(\Pi \times \mathrm{id}) \circ \sigma = \sigma \circ (\Pi \times \mathrm{id})$ and $(\Pi \times \mathrm{id})_{*} \nu = \hat\mu$; in particular, $\hat\mu$ is $\sigma$-invariant.
	Without SSC, the map $\sigma$ on $\Lambda \times \Sigma_{+}$ need not be invertible, and we shall work on $(\Sigma, \sigma, \nu)$, using $\Pi \times \mathrm{id}$ as a factor map.

	\subsection{Oseledets subspaces and their distributions}
	Let $\chi_j$, $V_j$, $U^j$, $E_j$ be given by \cref{thm Oseledets}.
	For each $j \in \{1, \cdots, d\}$, we let
	\enmt
	\item $\nu^{+}_j = (V_j)_{*} \nu_{+}$ be the distribution of $V_j(\bi_{+})$ on $\Gr(d, d-j+1)$ as $\bi_{+}$ is distributed according to $\nu_{+}$;
	\item $\nu^{-}_j = (U^j)_{*} \nu_{-}$ be the distribution of $U^j(\bi_{-})$ on $\Gr(d, j)$ as $\bi_{-}$ is distributed according to $\nu_{-}$,
	\eenmt
	and we set $\nu^{+}_{d+1} = \delta_{\{0\}}$.
	\begin{rema} \label{rem stationary}
		By \cref{thm Oseledets}, we have $V_j(\bi_{+}) = M_{i_0}^{-1} V_j(\sigma \bi_{+})$ and $U^{j}(\bi_{-}) = M_{i_{-1}} U^{j}(\sigma^{-1} \bi_{-})$. Since $\nu$ is a product measure, this implies
		\aryst
		\nu^{+}_j &=& \sum_{i \in I} p_i \, (M_i^{-1})_{*} \nu^{+}_j , \qquad
		\nu^{-}_j \ = \ \sum_{i \in I} p_i \, (M_i)_{*} \nu^{-}_j .
		\earyst
		We do not use any uniqueness of stationary measures.
	\end{rema}
	
	\subsection{Partitions and fibrewise measures}
	\label{subsec:partitions}
	For $k \in \{1, \cdots, d+1\}$ and $\nu_{+}$-a.e. $\bi_{+}$, let $\xi_k(\bi_{+})$ denote the measurable partition of $\Lambda$ into the sets $\Lambda \cap (x + V_k(\bi_{+}))$, $x \in \Lambda$.
	Let $\{ \mu^{V_k(\bi_{+})}_x \}_{x \in \Lambda}$ be a system of conditional measures of $\mu$ with respect to $\xi_k(\bi_{+})$ \cite{Rohlin}; $\mu^{V_k(\bi_{+})}_x$ is a probability measure on $\Lambda \cap (x + V_k(\bi_{+}))$.
	
	For $k \in \{1, \cdots, d+1\}$, let $\mathcal{A}_k$ be the partition of $\Sigma$ with atoms
	\aryst
	\mathcal{A}_k(\bi_{-}, \bi_{+}) = \{ (\bj_{-}, \bi_{+}) : \Pi(\bj_{-}) - \Pi(\bi_{-}) \in V_k(\bi_{+}) \}
	\earyst
	(on the $\nu_{+}$-null set of $\bi_{+}$ where $V_k(\bi_{+})$ is undefined, we take singletons).
	Let $\{ \nu^{\mathcal{A}_k}_{\bi} \}_{\bi \in \Sigma}$ be a system of conditional measures of $\nu$ with respect to $\mathcal{A}_k$.
	Since $\nu = \nu_{-} \otimes \nu_{+}$ and $\xi_k(\bi_{+})$ pull back to $\mathcal{A}_k$ on the slice $\Sigma_{-} \times \{\bi_{+}\}$, for $\nu$-a.e. $\bi$, the image of $\nu^{\mathcal{A}_k}_{\bi}$ under $(\bj_{-}, \bi_{+}) \mapsto \Pi(\bj_{-})$ is the fiberwise measure $\mu^{V_k(\bi_{+})}_{\Pi(\bi_{-})}$.
	
	Let $\sigma \mathcal{A}_k$ be the partition of $\Sigma$ whose atom containing $\bi$ is 
	\aryst
	(\sigma \mathcal{A}_k)(\bi):= \sigma\big( \mathcal{A}_k(\sigma^{-1} \bi) \big).
	\earyst
	We have
	\ary
	\label{eqn:sigma A}
	\sigma \mathcal{A}_k = \mathcal{A}_k \vee \cP.
	\eary
	
	For $\nu_{+}$-a.e. $\bi_{+}$, fix an orthonormal basis $e_1(\bi_{+}), \cdots, e_d(\bi_{+})$ of $\R^d$ adapted to the flag $V_1(\bi_{+}) \supset \cdots \supset V_d(\bi_{+})$, i.e.
	\aryst
	e_j(\bi_{+}) \in V_j(\bi_{+}) \cap V_{j+1}(\bi_{+})^{\perp}, \qquad j \in \{1, \cdots, d\};
	\earyst
	the choice can be made measurably in $\bi_{+}$, and we fix such a choice.
	For $y \in \R^d$, we write $y_j = \langle y, e_j(\bi_{+}) \rangle$, so that $y = \sum_{j} y_j e_j(\bi_{+})$.

	For $N \geq 0$, let $\cD_N$ be the partition of $\R$ into the dyadic intervals $[k 2^{-N}, (k+1) 2^{-N})$, $k \in \Z$, and let $\cD_N(s)$ be the atom containing $s \in \R$; for $t \in \R$, we write $\cD_N + t$ for the translated partition and $(\cD_N + t)(s)$ for its atom containing $s$.
	
	Let $m \in \{1, \cdots, d-1\}$. For each translate $t = (t_{m+1}, \cdots, t_d) \in \R^{d-m}$, we define a partition of $\Lambda$ refining $\xi_{m+1}(\bi_{+})$ as follows. For $x \in \Lambda$ and $N \geq 0$, let
	\aryst
	\cD^{t}_{m+1, \bi_{+}, N}(x) &=& \{ y \in \Lambda : y_j = x_j \ (j \leq m), \ y_j \in (\cD_N + t_j)(x_j) \ (j \geq m+1) \};
	\earyst
	the sets $\cD^{t}_{m+1,\bi_+,N}(x)$, $x \in \Lambda$, are the atoms of a measurable partition $\cD^{t}_{m+1,\bi_+,N}$ of $\Lambda$. For $x \in \Lambda$ and $r > 0$, we also set
	\aryst
	T^{(m)}_r(x, \bi_{+}) &=& \{ y \in \Lambda : y_j = x_j \ (j \leq m), \ |y_{m+1} - x_{m+1}| < r \};
	\earyst
	this is a tube, not a partition. For every translate $t$ and every $k \geq 0$, we have
	\aryst
	\xi_{m+1}(\bi_{+}) \ \preceq \ \cD^{t}_{m+1, \bi_{+}, N}, \qquad
	\cD^{t}_{m+1, \bi_{+}, N}(x) \ \subset \ T^{(m)}_{2^{-N}}(x, \bi_{+}).
	\earyst
	In particular, for each fixed $t$ the partitions are increasing in $N$. In \cref{sec proofofprop fiberwisedim} below, the translate will be chosen depending on $n$ as a random translate (\cref{subsec:random partition}), in order to gain the boundary regularity \eqref{eq randomdyadicpartitions}; until then, the translate plays no role and may be taken to be $0$, in which case we drop the superscript $t$ from the notation.

	\subsection{Entropies and the Ledrappier--Young formula}\label{subsec ly}
	Define the fiberwise entropies of $\cF$ by
	\begin{equation}
		\label{eqn:fiberwise entropies}
		h_k = H_{\nu}\big( \cP \,\big|\, \widehat{\mathcal{A}_k} \big), \quad k \in \{1, \cdots, d+1\},
	\end{equation}
	where $H_{\nu}(\cdot \, | \, \cdot)$ stands for conditional entropy and $\widehat{\mathcal{A}_k}$ is the $\sigma$-algebra generated by $\mathcal{A}_k$. In particular, we have
	\aryst
	h_{d+1 }=H(\bp) =-\sum_{i\in I} p_i\log p_i  \quad \text{and}\quad \ h_{ 1} = H_{\nu_{-}}\big( \cP \,\big|\, \Pi^{-1} \mathcal{B} \big) \ \geq \ 0,
	\earyst
	where $\mathcal{B}$ is the Borel $\sigma$-algebra of $\R^d$.
	We define the {\em entropy of $\cF$} by
	\begin{equation}
		\label{eq entropy F}
		H(\cF) = h_{d+1} - h_1 = H(\bp) - H_{\nu_{-}}\big( \cP \,\big|\, \Pi^{-1} \mathcal{B} \big) \in [0, H(\bp)],
	\end{equation}
	{ which is also called the Furstenberg entropy.}
	Set
	\begin{equation}
		\label{eq def Dk}
		D_k = \sum_{\ell = 1}^{k-1} \frac{h_{\ell+1} - h_{\ell}}{|\chi_{\ell}|}, \qquad k \in \{1, \cdots, d+1\}.
	\end{equation}
	
	The following is \cite[Theorems 1.3 and 1.4]{fengDimensionInvariantMeasures2023}; see also \cite[Theorem 2.3]{BK}.

	\begin{thm}[Exact dimensionality and the Ledrappier--Young formula] \label{thm Exactdim}
		For every $k \in \{1, \cdots, d+1\}$ and $\hat\mu$-a.e. $(x, \bi_{+}) \in \Lambda \times \Sigma_{+}$, writing $V = V_k(\bi_{+})$, the following hold.
		\enmt
		\item The projection $(P^{V})_*\mu$ is exact dimensional at $P^{V} x$, with dimension $D_k$:
		\aryst
		\lim_{r \to 0} \frac{ \log \, (P^{V})_* \mu \big( B(P^{V} x, r) \big) }{ \log r } = D_k .
		\earyst
		In particular, $\dim_{\mathrm H} (P^{V})_* \mu = D_k \leq k-1$ for $\nu^{+}_k$-a.e. $V \in \Gr(d, d-k+1)$.
		\item The fibrewise measure $\mu^{V}_x$ is exact dimensional at $x$, with dimension $D_{d+1} - D_k$:
		\aryst
		\lim_{r \to 0} \frac{ \log \mu^{V}_x \big( B(x, r) \big) }{ \log r } = D_{d+1} - D_k.
		\earyst
		\item If $k \leq d$, the projection of $\mu^{V}_x$ along $V_{k}(\bi_{+})$ is exact dimensional at $x$, with dimension $D_{k+1} - D_k$:
		\ary\label{eq exactdim}
		\lim_{r \to 0} \frac{ \log \mu^{V}_x \big(   T^{(k-1)}_r  (x, \bi_{+}) \big) }{ \log r } = D_{k+1} - D_k.
		\eary
		\eenmt
		In particular, $\mu$ is exact dimensional with
		\aryst
		\dim \mu = D_{d+1}.
		\earyst
	\end{thm}
	Here $T^{(0)}_r(x, \bi_{+}) = \{ y \in \Lambda : |y_1 - x_1| < r \}$, in accordance with the definition of $T^{(m)}_r$ for $m = 0$. The following lemma is immediate from \cref{thm Exactdim} and $|\chi_j|$ increasing.
	\begin{lemma}\label{lem:equivalence}
		Under the same assumption of \cref{thm Exactdim}, we have the following three statements are equivalent:
		\begin{itemize}
			\item We have
			\[ D_{d+1}=\max\{\delta_1+\cdots+\delta_d : \ \delta_j\in[0,1],\ H(\cF)=\sum_{1\leq j\leq d} |\chi_j|\delta_j   \}.  \]
			\item We have
			\[ \dimH\mu=    m+\frac{H(\cF)+\sum_{1\leq j\leq m}\chi_j}{ |\chi_{m+1}|},  \]
			where $m\leq d-1$ is the largest integer  satisfying  $\sum_{1\leq j\leq m}|\chi_j|\leq H(\cF)$.
			\item For every $1 \leq m \leq d$, for  $\nu_m^{+}$-a.e.  $V \in \Gr(d, d-m+1)$, we have 
			\aryst
			\dim P^{V} \mu = \min\{m - 1, \dim \mu \}.
			\earyst
		\end{itemize}
	\end{lemma}
	
	Equipped with exact dimension and entropy, we now deduce \cref{main thm} and \cref{main thm2} from \cref{main technical}.

	\begin{proof}[Proof of \cref{main thm} and \cref{main thm2}]
		For $|I|=1$ the attractor is a single point, $\mu$ is a Dirac mass, and all the conclusions below are immediate; we assume $|I|\geq2$ from now on.
		
		Fix an enumeration $I=\{1,\cdots, N\}$, $N\geq 2$.
		Denote the IFS in question by $\cF = (\{ f_{i}(x) = M_i x + a_i \}_{i \in I}, (p_i)_{i \in I} )$.
		Since we do not assume that $\cF$ has SSC, we cannot apply \cref{main technical} to $\cF$ directly.
		To resolve this issue, we use the \lq\lq lifting\rq\rq\ construction of \cite[Section 4]{BK}, which we now describe.
		
		Choose $\lambda$ with
		\aryst
		0 < \lambda < \min\Big\{ \frac{1}{|I|}, \ \min_{i \in I, \, v \in \R^d \setminus \{0\}} \frac{ \| M_i v \| }{\| v \|} \Big\}.
		\earyst 
		Set $b_i = (i-1)/(N-1)$ for $i \in I$.
		For each $i \in I$, we define
		\aryst
		g_i(x, y) = (f_i(x), \lambda y + b_i ), \quad x \in \R^d, \ y \in \R,
		\earyst
		and we denote the IFS $(\{g_i \}_{i \in I}, (p_i)_{i \in I} )$ on $\R^{d+1}$ by $\cG$, which has SSC. 
		
		Since $\cF$ has Zariski dense linear parts, the Lyapunov spectrum $\chi_1 > \cdots > \chi_d$ of $\cF$ is simple.
		For every $n \geq 1$ and $(i_1, \cdots, i_n) \in I^n$ we have
		\aryst
		\min_{v \neq 0} \frac{\| M_{i_n} \cdots M_{i_1} v \|}{\|v\|} \ \geq \ \prod_{k = 1}^{n} \min_{v \neq 0} \frac{\| M_{i_k} v \|}{\|v\|} \ > \ \lambda^n ,
		\earyst
		hence $\chi_d \geq \log \min_{i, v} \frac{\| M_i v \|}{\|v\|} > \log \lambda$.
		Since the linear cocycle of $\cG$ is the direct sum of the linear cocycle of $\cF$ and the constant cocycle $\lambda$ on $\R e_{d+1}$, the Lyapunov spectrum of $\cG$ is
		\aryst
		\chi_1 > \cdots > \chi_d > \chi_{d+1} := \log \lambda ,
		\earyst
		which is simple.
		Moreover, for $\nu$-a.e. $\bi = (\bi_{-}, \bi_{+})$, the Oseledets subspaces and flags of $\cG$ are given by
		\begin{align}
			\label{eqn:Oseledets V F G}V^{\cG}_{q}( \bi_{+}) &= V^{\cF}_{q}( \bi_{+}) \oplus \R, \quad q \in \{1, \cdots, d\}\\
			\label{eqn:Oseledets F G}
			E^{\cG}_{q}(\bi_{-}, \bi_{+}) &= E^{\cF}_{q}(\bi_{-}, \bi_{+}) \oplus \{0\}, \quad q \in \{1, \cdots, d\}, \\
			E^{\cG}_{d+1}(\bi_{-}, \bi_{+}) &=V_{d+1}^\cG= \{0\} \oplus \R = \R e_{d+1}. \nonumber
		\end{align}
		Since $U^{m-1,\cG}=U^{m-1,\cF}\oplus\{0\}$ and $E^{\cG}_{q}=E^{\cF}_{q}\oplus\{0\}$ by \eqref{eqn:Oseledets F G}, while $V^{\cG}_{m+1}=V^{\cF}_{m+1}\oplus\R e_{d+1}$, the angle $\measuredangle(V^{\cG}_{m+1},U^{m-1,\cG}\oplus E^{\cG}_{q})$ equals $\measuredangle(V^{\cF}_{m+1},U^{m-1,\cF}\oplus E^{\cF}_{q})$: the minimum defining the latter is attained on vectors orthogonal to $e_{d+1}$. Hence $\cG$ inherits the ONC at level $d$ from $\cF$ due to Zariski density, so \cref{main technical} applies to $\cG$ with the same $d$; we then get, for $1\leq m\leq d+1$
		\[ \dim P^{V_{m}^{\cG}} \mu^{\cG}=\min\{m-1,\dim P^{V_{d+1}^\cG}\mu^{\cG} \}= \min\{m-1,\dim \mu^{\cF} \}%
		. \]
		
		Due to \eqref{eqn:Oseledets V F G}, we notice that 
		\[ \dim P^{V_{m}^{\cG}} \mu^{\cG}=\dim P^{V_{m}^{\cF}} \mu^{\cF}. \]
		Hence we can deduce Theorems \ref{main thm2} and \ref{main thm}, in the Zariski dense case, from \cref{lem:equivalence}.

		To see the last statement in Theorem \ref{main thm}, it suffices to see that if  $\cF$ has the SOSC property, then we have  $H(\cF) = H(\bp)$
		This follows from \cite[Corollary 2.7, 2.8]{barany_hausdorff_2017}.  
	\end{proof}    
	
	\subsection{Additional structure under SSC} \label{subsec SSC}
	In this subsection we assume in addition that $\cF$ has SSC.
	
	One property of SSC is for any pair of words $\bw,\bw'\in \cup_{n\geq 1} I^n$
	\begin{equation}\label{equ separation}
		f_{\bw}(\Lambda)\cap f_{\bw'}(\Lambda)=\emptyset
	\end{equation}
	unless one of $\bw,\bw'$ is a suffix of the other.

	Then $\Pi : \Sigma_{-} \to \Lambda$ is a homeomorphism, and $\Pi \times \mathrm{id}$ is a conjugacy between $(\Sigma, \sigma, \nu)$ and $(\Lambda \times \Sigma_{+}, \sigma, \hat\mu)$; we identify the two systems.
	Under this identification,
	\enmt
	\item $\cP = \{ f_i(\Lambda) \times \Sigma_{+} : i \in I \}$, and we also write $[i] = f_i(\Lambda) \subset \Lambda$;
	\item $\mathcal{A}_k(x, \bi_{+}) = \{ (y, \bi_{+}) : y \in \Lambda, \ y - x \in V_k(\bi_{+}) \}$;
	\item $h_{1} = 0$ and $H(\cF) = H(\bp)$;
	\item $\sigma : \Lambda \times \Sigma_{+} \to \Lambda \times \Sigma_{+}$ is a homeomorphism with $\sigma^{-1}(x, \bi_{+}) = (\Phi x, \ x[1] \bi_{+})$, where $\Phi = \Pi \sigma^{-1} \Pi^{-1} : \Lambda \to \Lambda$ is the expanding map with $\Phi|_{f_i(\Lambda)} = f_i^{-1}$, and $x[n] = \bi_{-}[n]$ for $x = \Pi(\bi_{-})$; 
	\item $U^k(x) = U^k(\Pi^{-1} x)$ is a measurable map from $\Lambda$ to $\Gr(d, k)$, and $\nu^{-}_k = (U^k)_{*} \mu$.
	\eenmt
	In particular, $\sigma$ is a bijection of $\Lambda \times \Sigma_{+}$ preserving $\hat\mu$, $\cP$ is a finite partition, and the fibrewise entropies \eqref{eqn:fiberwise entropies} satisfy
	\begin{equation} \label{eqn:hk SSC}
		h_k = H_{\hat\mu}(\cP \mid \mathcal{A}_k) \ \leq \ H_{\hat\mu}(\cP) \ \leq \ \log |I| \ < \ \infty, \qquad k \in \{1, \cdots, d\}.
	\end{equation}
	Let $\{ \hat\mu^{\mathcal{A}_k}_{(x, \bi_{+})} \}_{(x, \bi_{+}) \in \Lambda \times \Sigma_{+}}$ be the system of conditional measures of $\hat\mu$ with respect to $\mathcal{A}_k$, and let $\pi : \Lambda \times \Sigma_{+} \to \Lambda$ be the projection $\pi(x, \bi_{+}) = x$.
	We have
	\ary \label{eqn:pi cond}
	\pi_{*} \hat\mu^{\mathcal{A}_k}_{(x, \bi_{+})} = \mu^{V_k(\bi_{+})}_x \quad \mbox{for } \hat\mu\mbox{-a.e. } (x, \bi_{+}).
	\eary
	
	For $n \geq 1$, let
	\ary \label{eqn:def Pn}
	\cP_n := \bigvee_{j=0}^{n-1} \sigma^j \cP ,
	\eary
	with the convention that $\cP_0$ is the trivial partition.
	
	We have for every $k \in \{1, \cdots, d\}$ and for every $n \geq 0$,
	\ary \label{eqn:sigma n A}
	\sigma^n \mathcal{A}_k &=& \mathcal{A}_k \vee \cP_n. 
	\eary
	Moreover,
	\begin{enumerate}
		\item[(a)] $\mathcal{A}_k \leq \sigma \mathcal{A}_k \leq \sigma^2 \mathcal{A}_k \leq \cdots$, and $\bigvee_{n \geq 0} \sigma^n \mathcal{A}_k$ is the partition of $\Lambda \times \Sigma_{+}$ into points.
		\item[(b)] If $k \leq d-1$, then for every $n \geq 0$,
		\ary \label{eqn:sigma A k+1}
		\sigma^n \mathcal{A}_{k+1} = \mathcal{A}_{k+1} \vee \sigma^n \mathcal{A}_k .
		\eary
	\end{enumerate}
	
	\begin{lemma} \label{lem entropy iterates}
		Assume that $\cF$ has SSC, and let $k \in \{1, \cdots, d\}$. Then
		\ary \label{eqn:hk sigma}
		h_k = H_{\hat\mu}(\sigma \mathcal{A}_k \mid \mathcal{A}_k) = \int - \ln \hat\mu^{\mathcal{A}_k}_{(x, \bi_{+})} \big( (\sigma \mathcal{A}_k)(x, \bi_{+}) \big) \, d\hat\mu(x, \bi_{+}) ,
		\eary
		and, for every $n \geq 0$,
		\ary
		H_{\hat\mu}(\sigma^n \mathcal{A}_k \mid \mathcal{A}_k) \ = \ H_{\hat\mu}(\cP_n \mid \mathcal{A}_k) &=& n h_k . \label{eqn:n h k}
		\eary
	\end{lemma}
	\begin{proof}
		Since $h_k = H_{\hat\mu}(\cP \mid \mathcal{A}_k)$ by \eqref{eqn:hk SSC}, \eqref{eqn:sigma A} gives
		\aryst
		h_k = H_{\hat\mu}(\cP \vee \mathcal{A}_k \mid \mathcal{A}_k) = H_{\hat\mu}(\sigma \mathcal{A}_k \mid \mathcal{A}_k) = \int - \ln \hat\mu^{\mathcal{A}_k}_{(x, \bi_{+})} \big( (\sigma \mathcal{A}_k)(x, \bi_{+}) \big) \, d\hat\mu(x, \bi_{+}) ,
		\earyst
		which is \eqref{eqn:hk sigma}.
		Next, we show that for every $n\geq 1$, 
		\ary
		\label{eqn:H sigma A}
		H_{\hat\mu}(\sigma^{n+1} \mathcal{A}_k \mid \sigma^n \mathcal{A}_k) = H_{\hat\mu}(\sigma^{n} \mathcal{A}_k \mid \sigma^{n-1} \mathcal{A}_k).
		\eary
		Write $y = (x, \bi_{+})$.
		The map $\sigma^{-1}$ (item (4)) is a bijection preserving $\hat\mu$ that sends each atom of $\sigma^n \mathcal{A}_k$ onto an atom of $\sigma^{n-1} \mathcal{A}_k$. This gives, for $\hat\mu$-a.e. $y$,
		\aryst
		(\sigma^{-1})_{*} \hat\mu^{\sigma^{n} \mathcal{A}_k}_{y} = \hat\mu^{\sigma^{n-1} \mathcal{A}_k}_{\sigma^{-1} y} .
		\earyst
		Combining this with the $\sigma$-invariance of $\hat\mu$ to change variables, we obtain
		\aryst
		&&H_{\hat\mu}(\sigma^{n+1} \mathcal{A}_k \mid \sigma^{n} \mathcal{A}_k)
		= - \int \ln \hat\mu^{\sigma^n \mathcal{A}_k}_{y} \big( (\sigma^{n+1} \mathcal{A}_k)(y) \big) \, d\hat\mu(y) \\
		&=& - \int \ln \hat\mu^{\sigma^{n-1} \mathcal{A}_k}_{\sigma^{-1} y} \big( (\sigma^n \mathcal{A}_k)(\sigma^{-1} y) \big) \, d\hat\mu(y)
		= - \int \ln \hat\mu^{\sigma^{n-1} \mathcal{A}_k}_{y} \big( (\sigma^n \mathcal{A}_k)(y) \big) \, d\hat\mu(y) \\
		&=& H_{\hat\mu}(\sigma^n \mathcal{A}_k \mid \sigma^{n-1} \mathcal{A}_k).
		\earyst
		Since $\mathcal{A}_k \leq \sigma \mathcal{A}_k \leq \cdots \leq \sigma^n \mathcal{A}_k$, the chain rule for conditional entropy, \eqref{eqn:H sigma A} and  \eqref{eqn:hk sigma} give $H_{\hat\mu}(\sigma^n \mathcal{A}_k \mid \mathcal{A}_k) = \sum_{j=0}^{n-1} H_{\hat\mu}(\sigma^{j+1} \mathcal{A}_k \mid \sigma^j \mathcal{A}_k) = n h_k$, and $H_{\hat\mu}(\sigma^n \mathcal{A}_k \mid \mathcal{A}_k) = H_{\hat\mu}(\cP_n \mid \mathcal{A}_k)$ by \eqref{eqn:sigma n A}; this is \eqref{eqn:n h k}.
	\end{proof}

	The following proposition is obtained by applying \cite[Theorem A]{araujo_holder_2016} to $(\Lambda, \Phi, \mu)$ and the cocycle $x \mapsto M_{x[1]}^{-1}$, whose $n$-th iterate at $x = \Pi(\bi_{-})$ is $M^{-n}_{\bi_{-}}$.%

	\begin{prop}[H\"older continuity of the Oseledets subspaces] \label{prop Holdermap}
		There is $\alpha > 0$ such that for every $\eps > 0$, there is a subset $\Lambda' = \Lambda'(\cF, \eps) \subset \Lambda$ with $\mu(\Lambda') > 1 - \eps$ such that for every $k \in \{1, \cdots, d\}$, the map $x \mapsto U^k(x)$ is a $C^{\alpha}$ map from $\Lambda'$ to $\Gr(d, k)$.
	\end{prop}
	Henceforth, we fix $\alpha \in (0, 1)$ as given by \cref{prop Holdermap} for the rest of the paper.

	\subsection{Regularity of Furstenberg measures}

	\begin{lemma} \label{prop flag transversal}
		Let $\mathcal{F}$ be an affine IFS on $\mathbb{R}^d$ with  Zariski dense linear parts. Then for every $1 \leq m \leq d-1$ and every $W \in \Gr(d, m)$,
		\aryst
		\nu^{+}_{m+1} \big( \{ V \in \Gr(d, d-m) : V \cap W \neq \{0\} \} \big) = 0 .
		\earyst
	\end{lemma}
	\begin{proof}
		Since the semigroup generated by $\{M_i\}_{i\in I}$ is Zariski dense in $\mathrm{GL}_d(\R)$, the semigroup generated by $\{M_i^{-1}\}_{i \in I}$ is also Zariski dense. For $1 \leq k \leq d-1$ and $g \in \mathrm{GL}_d(\R)$, write $\rho_k(g) = \wedge^k g \in \mathrm{GL}(\wedge^k \R^d)$. 
		Hence, $\{\rho_k(M_i^{-1})\}_{i \in I}$ is strongly irreducible and proximal as a subsemigroup of $\GL(\wedge^k \R^d)$ \cite{goldsheid1989lyapunov,GuivarchRaugi89}, i.e., hypothesis (1.6) of \cite{Rapaport24} holds. Therefore, the claim is \cite[Lemma 2.12]{Rapaport24} applied to $(M_i^{-1})_{i \in I}$.
	\end{proof}
	
	The following result follows immediately from this lemma.
	\begin{cor}
		\label{cor:ONC Zariski dense}
		Let $\mathcal{F}$ be an affine IFS on $\mathbb{R}^d$ with   Zariski dense linear parts. Then $\mathcal{F}$ has the ONC at level $d$.
	\end{cor}
	
	\begin{cor} \label{cor generic hyperplane ae}
		Assume that $\cF$ has a simple Lyapunov spectrum and the ONC at level $d_0$.
		Then for every $m \in \{1, \ldots, d_0 - 1\}$ and $(\nu_- \otimes \nu_+ \otimes \nu_+)$-a.e. $(\bi_-, \bi_+, \bj_+)$,
		\aryst
		V_{m+1}(\bi_{+}) \cap \big( U^{m-1}(\bi_-) + V_{m+1}(\bj_{+}) \big) \quad \mbox{is a hyperplane in } V_{m+1}(\bi_{+}).
		\earyst
	\end{cor}
	\begin{proof}
		Fix $m \in \{1, \ldots, d_0 - 1\}$.
		By \Cref{thm Oseledets} (3) and the ONC at level $d_0$ (with $q = m+1$), there is a $(\nu_- \otimes \nu_+ \otimes \nu_+)$-conull set $\Omega_1 \subset \Sigma_- \times \Sigma_+ \times \Sigma_+$ such that for every $(\bi_-, \bi_+, \bj_+) \in \Omega_1$,
		\begin{align}
			\label{eqn:generic hyperplane}
			&E_{m+1}(\bi_-, \bi_+) \subset V_{m+1}(\bi_+), \qquad \dim \big( U^{m-1}(\bi_-) \oplus E_{m+1}(\bi_-, \bi_+) \big) = m,\\
			&\measuredangle \big( V_{m+1}(\bj_{+}), \, U^{m-1}(x) \oplus E_{m+1}(\bi_-, \bi_{+}) \big) > 0 .
		\end{align}
		Fix $(\bi_-, \bi_+, \bj_+) \in \Omega_1$ and write $U = U^{m-1}(\bi_-)$, $V = V_{m+1}(\bi_+)$, $V' = V_{m+1}(\bj_+)$, $E = U \oplus E_{m+1}(\bi_-, \bi_+)$. Then \eqref{eqn:generic hyperplane} implies $V' \oplus E = \R^d$, and hence $V + (U + V') = \R^d$.
		By Grassmann's formula,
		\aryst
		\dim \big( V \cap (U + V') \big) = \dim V + \dim (U + V') - \dim \big( V + (U + V') \big)= d - m - 1 ,
		\earyst
		completing the proof.
	\end{proof}
	
	\begin{cor}
		\label{cor:theta q}
		Assume that $\cF$ has a simple Lyapunov spectrum and the ONC at level $d_0$.
		Then for every $m \in \{1, \ldots, d_0 - 1\}$, $q\in \{m+1,\cdots,d_0\}$ and $(\nu_- \otimes \nu_+ \otimes \nu_+)$-a.e. $(\bi_-, \bi_+, \bj_+)$,
		$$\Theta_q(x, \bi_{+}, \bj_{+}) := \measuredangle\big( H_{U^{m-1}(\bi_-), \bi_{+}}(\bj_{+}), E_q(\bi_-, \bi_{+}) \big)>0,$$
		where $H_{U^{m-1}(\bi_-), \bi_{+}}(\bj_{+}):=V_{m+1}(\bi_{+}) \cap \big( U^{m-1}(\bi_-) + V_{m+1}(\bj_{+}) \big)$.
	\end{cor}
	
	\begin{proof}
		Fix a triple $(\bi_-, \bi_{+}, \bj_{+})$ at which the Oseledets subspaces are defined, $U^{m-1}(\bi_-) \cap V_{m+1}(\bi_{+}) = \{0\}$ (\Cref{thm Oseledets}), and the conclusion of \cref{def ONC} holds. These conditions hold for a.e.\ triple.
		Put
		$$U = U^{m-1}(\bi_-), \quad V = V_{m+1}(\bi_{+}), \quad V' = V_{m+1}(\bj_{+}) \quad \text{and} \quad L = E_q(x, \bi_{+}).$$
		Since $q \geq m+1$, we have $L \subset V$, and since $L$ is a line, 
		\aryst
		\Theta_q = 0 \ \Longleftrightarrow L \subset V \cap (U + V')\ \Longleftrightarrow \ L \subset U + V' .
		\earyst
		Suppose this were the case, and write $0 \neq e \in L$ as $e = u + v'$ with $u \in U$, $v' \in V'$.
		Then $v' = e - u$ belongs to $(U \oplus L) \cap V'$, and $v' \neq 0$, since otherwise $e = u \in U \cap V = \{0\}$.
		Hence
		$$\big( U^{m-1}(x) \oplus E_q(x, \bi_{+}) \big) \cap V_{m+1}(\bj_{+}) \neq \{0\},$$
		contradicting the ONC at level $d_0$.
		Therefore $\Theta_q > 0$ almost everywhere.
	\end{proof}

	\section{Strategy for the proof of Theorem \ref{main technical}}
	
	For the rest of the paper, $\cF = ( \{ f_i(x) = M_i x + a_i \}_{i \in I}, (p_i)_{i \in I} )$ is an affine IFS on $\R^d$, and we assume
	\enmt
	\item $\cF$ has the SSC property,
	\item the Lyapunov spectrum of $\cF$ is simple,
	\item $\cF$ has the ONC at level $d_0$ for some $d_0 \in \{1, \cdots, d\}$. 
	\eenmt
	We now state the key proposition of this paper.
	
	\begin{prop} \label{prop diminc}
		Let $m \in \{1, \cdots, d_0   - 1 \}$ satisfy that $D_m = m -1$ and $D_{m+1} < m$. Then we have $D_{d_0+1} = D_{m+1}$.
	\end{prop}
	\begin{proof}[Proof of \cref{main technical}]
		By \cref{thm Exactdim}, $D_1=0$ and $D_j\le j-1$ for every $j$; moreover $D_{j+1}-D_j$, being the dimension of a measure on a line, is at most $1$.
		Let $m\in\{1,\cdots,d_0-1\}$ be the largest index with $D_m=m-1$. Then $D_{m+1}<m$ or $m=d_0-1$ and $D_{d_0}=d_0-1$. 
		
		For the first case, \cref{prop diminc} gives $D_{d_0+1}=D_{m+1}$. By \cref{thm Exactdim} again, $\dim P^V\mu=D_j$ for $\nu_j^+$-a.e. $V\in\Gr(d,d-j+1)$, so it remains to identify $D_j$ with $\min\{j-1,D_{d_0+1}\}$. If $ j\le m$, then $D_j=j-1\le D_{d_0+1}$. If $d_0+1\ge j>m$, then $D_j=D_{d_0+1}$ by monotonicity, and $D_{d_0+1}=D_{m+1}<m\le j-1$. Hence $D_j=\min\{j-1,D_{d_0+1}\}$.
		
		In the second case $m=d_0-1$ and $D_{d_0}=d_0-1$. The conclusion follows similarly.
	\end{proof}

	\section{Fiberwise measure and Projection measure}

	\subsection{Support of fiberwise measures}

	We will denote by $\Omega^{0}_{k} \subset I^{\N}$ a subset with full $\nu_{+}$-measure such that the statement in \cref{thm Exactdim} holds with $V = V_k(\bi_{+})$ for every $\bi_{+} \in \Omega^{0}_k$.

	We need the following lemmas, on more refined information on fiberwise measures.
	\begin{lemma} \label{lem smallersupport}
		Let $m \in \{1, \cdots, d_0 - 2\}$ satisfy that $D_{m+1} = D_{m+2}$.  Then for $(\mu \otimes \nu_{+})$-a.e. $(x, \bi_{+})$,  the measure $\mu^{V_{m+1}(\bi_{+})}_x$ is supported on $x + V_{m+2}(\bi_{+})$, and 
		\aryst
		\mu^{V_{m+1}(\bi_{+})}_x |_{x + V_{m+2}(\bi_{+})} = \mu^{V_{m+2}(\bi_{+})}_x.
		\earyst
	\end{lemma}
	\Cref{lem smallersupport} can be proved by running the proof of \cite[Proposition 2]{LX} (Step 3 there, up to equation (2.4)) %
	by replacing $(M, \mu, f)$ of \cite{LX} with $(\Lambda \times {\Sigma_{+}}, \hat{\mu}, \sigma^{-1})$. More precisely,
	we apply the argument to the partitions
	$\mathcal{A}_{m+1} \preceq \mathcal{A}_{m+2}$, and the conditional measures  $\hat{\mu}^{\mathcal{A}_{m+1}}_{(x, \bi_{+})}$, $\hat{\mu}^{\mathcal{A}_{m+2}}_{(x, \bi_{+})}$;
	\cite[Lemma 3]{LX} is verified in \eqref{eqn:sigma A k+1} with $k = m+1$; the identity $H_\mu(f^{-n} \eta \mid \eta) = n H_\mu(f^{-1} \eta \mid \eta)$ is \eqref{eqn:n h k}.
	The argument of \cite{LX} then yields the counterpart of \cite[(2.4)]{LX}: for $\hat{\mu}$-a.e. $(x, \bi_{+})$,
	\ary \label{eqn:cond equal}
	\hat{\mu}^{\mathcal{A}_{m+1}}_{(x, \bi_{+})} = \hat{\mu}^{\mathcal{A}_{m+2}}_{(x, \bi_{+})}.
	\eary
	Using \eqref{eqn:pi cond} gives, for $\hat{\mu}$-a.e. $(x, \bi_{+})$, that $\mu^{V_{m+1}(\bi_{+})}_x = \mu^{V_{m+2}(\bi_{+})}_x$ and that this measure is supported on $x + V_{m+2}(\bi_{+})$.

	As an immediate corollary, we obtain the following.
	\begin{cor} \label{cor supportoffiberwisemeasure}
		Let $m \in \{1, \cdots, d_0-1\}$ and let $m_0 \in \{m, \cdots, d_0-1\}$ be the largest integer with $D_{m_0+1} = D_{m+1}$. For $(\mu \otimes \nu_{+})$-a.e. $(x, \bi_{+})$,  the measure $\mu^{V_{m+1}(\bi_{+})}_x$ is supported on $x + V_{m_0+1}(\bi_{+})$, and 
		\aryst
		\mu^{V_{m+1}(\bi_{+})}_x |_{x + V_{m_0+1}(\bi_{+})} = \mu^{V_{m_0+1}(\bi_{+})}_x.
		\earyst
	\end{cor}

	\subsection{Regularity of projected measures}

	We will use the following notion related to a separable complete metric space $X$.
	\begin{defi}[Covering number and Regular measure] \label{def regularmeasure}
		Given a subset $A \subset X$ and $\rho > 0$, we denote by $N_{\rho}(A)$ the minimal number of $\rho$-balls in $X$ needed to cover $A$. 
		
		Given a Borel probability measure $\omega$ on $X$,  $\beta \geq 0$, $ \eps  > 0$ and  $\rho, \eta \in (0, 1)$. We say that $\omega$ is  
		\enmt
		\item {\em $(\beta, \rho, \eta, \eps)$-lower regular}  if every Borel subset $A \subset X$ with $\omega(A) \geq \eta$ satisfies $N_{\rho}(A) \geq \rho^{ - \beta + \eps}$,
		\item {\em $(\beta, \rho, \eta, \eps)$-upper regular} if there is a Borel subset $A \subset X$ with $\omega(A) \geq 1 - \eta$ such that $N_{\rho}(A) \leq \rho^{- \beta - \eps}$.
		\eenmt
		We say that $\omega$ is {\em $(\beta, \rho, \eta, \eps)$-regular} if $\omega$ is both lower regular and upper regular with the same parameters.
	\end{defi}
	
	We also have the following notion, which is closely related to \cref{def regularmeasure}.
	\begin{defi}  \label{def upperfrostman}
		Given $\rho > 0$ and a Borel probability measure $\omega$ on $X$, we say that $\omega$ is {\em $(\beta, \rho, \eta)$-upper Frostman} if there is a decomposition $\omega = \omega' + \omega''$ into two positive measures so that $| \omega'' | < \eta$ and 
		\aryst
		\sup_{x \in X} \omega'(B(x, \rho)) < \rho^{\beta}.
		\earyst
	\end{defi}

	The following two statements are elementary, but they isolate the step that is used repeatedly below. The first is the general counting principle; the second is the case of balls.
	
	\begin{lemma} \label{lem counting}
		Let $\omega$ be a probability measure on a set $X$, let $\cK$ be a family of subsets of $X$, and let $\eta' > \eta \geq 0$ and $t > 0$. Suppose that there are positive measures $\omega', \omega''$ with $\omega = \omega' + \omega''$, $|\omega''| \leq \eta$ and $\sup_{K \in \cK} \omega'(K) \leq t$. Then every $A \subset X$ with $\omega(A) \geq \eta'$ requires at least $(\eta' - \eta)/t$ members of $\cK$ to cover it.
	\end{lemma}
	\begin{proof}
		Let $\{ K_i \}_{i \in \cI} \subset \cK$ cover $A$. Then
		\aryst
		\eta' - \eta \ \leq \ \omega(A) - |\omega''| \ \leq \ \omega'(A) \ \leq \ \sum_{i \in \cI} \omega'(K_i) \ \leq \ |\cI| \, t ,
		\earyst
		so $|\cI| \geq (\eta' - \eta)/t$.
	\end{proof}
	
	\begin{cor} \label{lem upperfrostman implies regular}
		Let $\eta' > \eta > 0$, $\eps > 0$ and $\beta \in (0, d)$. Let $\omega$ be a $(\beta, \rho, \eta)$-upper Frostman probability measure on $X$ for some $\rho \ll_{\eps, \eta', \eta} 1$. Then $\omega$ is $(\beta, \rho, \eta', \eps)$-lower regular.
	\end{cor}
	\begin{proof}
		Let $\omega = \omega' + \omega''$ be the decomposition given by \cref{def upperfrostman}. Apply \cref{lem counting} with $\cK$ the family of balls of radius $\rho$ and $t := \rho^{\beta}$. It gives $N_{\rho}(A) > (\eta' - \eta) \, \rho^{-\beta}$ for every Borel set $A$ with $\omega(A) \geq \eta'$. Since $\rho \ll_{\eps, \eta', \eta} 1$, the right-hand side exceeds $\rho^{-\beta + \eps}$.
	\end{proof}

	In applications, we will typically take $X$ as a Euclidean space $\R^k$ for some $k \geq 1$.
	By a slight abuse of notation, given $\lambda > 0$,  we define the dilation
	\aryst
	\Dil_{\lambda} : \R^k \to \R^k, \qquad x \mapsto \lambda x.
	\earyst
	
	Given $m \in \{1, \dots,  d_0-1\}$, $\rho, \delta > 0$ and $\bi_{+} \in \Sigma_+$, we set
	\aryst
	W(\bi_{+}, m, \rho, \delta) := S\big( 0, ( V_{m+1}(\bi_+), \R^d ), ( \rho, \delta ) \big),
	\earyst
	the slab defined in \eqref{eqn:slab definition}. 	 
	
	\begin{lemma} \label{lem regularmeasureeverywhere}
		Let $m \in \{1, \cdots, d_0 - 1\}$ and let $C > 0$. For every $\eps, \eta > 0$ and every $n \gg_{\cF,   C, \eps, \eta} 1$,  there is a subset $\Gamma(n) \subset \Lambda \times \Sigma_{+}$ with $(\mu \otimes \nu_{+})(\Gamma(n)) > 1 - \eta$, such that  for every $(x, \bi_{+}) \in \Gamma(n)$,  the measure $$(\Dil_{2^{C n}})_*(P^{V_{m+1}(\bi_{+})})_* \mu^{x + W(\bi_{+}, m, 2^{-C n}, 2^{- C n}) }$$ is $(D_{m+1}, 2^{-  n}, \eps, \eps)$-lower regular. %
	\end{lemma}
	
	\begin{proof}
		We introduce a local notation, which is only used in this proof.
		Let $\cD_{Cn, (C+1)n}^{\bi_+}$ be the dyadic partition with side length $2^{-Cn}$ along $V_{m+1}(\bi_{+})$ and $2^{-(C+1)n}$ along $V_{m+1}(\bi_{+})^\perp$, and $\cD_{Cn, (C+1)n}^{\bi_+}(x)$ denote the atom containing $x$. 
		
		We first prove a claim: for any $C, \eps > 0$, there exist a subset $\Gamma' \subset \Lambda \times \Sigma_{+}$ with $(\mu \otimes \nu_{+})(\Gamma') > 1 - \eps$ and $n_0 = n_0(\cF,  C, \eps)$ such that for any $n \geq n_0$ and any $(x, \bi_{+}) \in \Gamma'$, we have
		\begin{equation}\label{equ:d-upper-bound}
			\mu(\cD_{Cn, (C+1)n}^{\bi_+}(x))\leq 2^{-(C(D_{d+1}-D_{m+1})+ (C+1) D_{m+1}-\eps)n} 
		\end{equation}

		\begin{proof}
			The main reason that we state the claim with a partition is that the density theorem works for general partitions. 
			
			By the exact dimensionality of the fiberwise and the projected measures (\Cref{thm Exactdim}) and Egorov's theorem, for each $\eps > 0$, there exist a subset $J_1 \subset \Lambda \times \Sigma_{+}$ with $(\mu \otimes \nu_{+})(J_1) > 1 - \eps^{3}$ and $n_0$ such that for all $(x, \bi_{+}) \in J_1$ and $r < 2^{-C n_0}$, we have
			\begin{align}
				\label{equ:fibre-upp}\mu_{x}^{V_{m+1}(\bi_+)}(B(x,r))&\leq r^{D_{d+1}-D_{m+1}-\eps} \\
				\label{equ:proj-upp}(P^{V_{m+1}(\bi_+)})_{*} \mu \big( B( P^{V_{m+1}(\bi_+)} x, r ) \big) &\leq r^{D_{m+1}-\eps}
			\end{align}
			By the Lebesgue density theorem for partitions and Egorov's theorem (do the argument for each piece $f_i(\Lambda)$, $i \in I$, then put together), we can find a subset $J_2 \subset J_1$ with $(\mu \otimes \nu_{+})(J_2) > 1 - 2\eps$ such that for all $(x, \bi_{+}) \in J_2$ and $n > n_0$, we have
			\[ \mu(J_1\cap \cD_{Cn, (C+1)n}^{\bi_+}(x))\geq \frac{1}{2}\mu(\cD_{Cn, (C+1)n}^{\bi_+}(x)). \]
			Therefore for $(x, \bi_{+}) \in J_2$, we have
			\aryst &&\mu(\cD_{Cn, (C+1)n}^{\bi_+}(x))\leq 2\mu(J_1\cap \cD_{Cn, (C+1)n}^{\bi_+}(x))\\
			&&=2\int \mu_y^{V_{m+1}(\bi_+)}(\cD_{Cn, (C+1)n}^{\bi_+}(x)\cap J_1) \, d\mu(y). 
			\earyst
			The integration region for $y$ can shrink to $ (P^{V_{m+1}(\bi_+)})^{-1}(P^{V_{m+1}(\bi_+)}\cD_{Cn, (C+1)n}^{\bi_+}(x))$.
			
			If $(y + V_{m+1}(\bi_+)) \cap J_1 \cap \cD_{Cn, (C+1)n}^{\bi_+}(x)$ is non-empty, take an element $(z, \bi_{+})$ in the intersection; since the fibre-width of the atom is $2^{-Cn}$, by \eqref{equ:fibre-upp} we have
			\aryst\mu_y^{V_{m+1}(\bi_+)}(\cD_{Cn, (C+1)n}^{\bi_+}(x)\cap J_1)&\leq& \mu_z^{V_{m+1}(\bi_+)}(B^{V_{m+1}(\bi_+)}(z, 2^{-Cn + 1}))\\
			&\leq& 2^{-Cn(D_{d+1}-D_{m+1}-\eps)}. \earyst
			If the intersection is empty, the bound is also true. Hence, since the projected atom $P^{V_{m+1}(\bi_+)}\cD_{Cn, (C+1)n}^{\bi_+}(x)$ has transverse width $2^{-(C+1)n}$, by \eqref{equ:proj-upp} we obtain
			\aryst 
			&&\mu(\cD_{Cn, (C+1)n}^{\bi_+}(x))\leq 2^{-Cn(D_{d+1}-D_{m+1}-\eps)} \, (P^{V_{m+1}(\bi_+)})_{*} \mu \big( P^{V_{m+1}(\bi_+)} \cD_{Cn, (C+1)n}^{\bi_+}(x) \big) \\
			&&\leq 2^{-Cn(D_{d+1}-D_{m+1}-\eps)} \, (P^{V_{m+1}(\bi_+)})_{*} \mu \big( B( P^{V_{m+1}(\bi_+)} x, 2^{-(C+1)n + 1} ) \big) \\
			&&\leq 2^{-(C(D_{d+1}-D_{m+1})+(C+1)D_{m+1})n + (2C+1)\eps n}.
			\earyst
			The factor $2^{(2C+1)\eps n}$ is absorbed by enlarging $\eps$.	
		\end{proof}
		
		By the exact dimensionality of $\mu$ (\Cref{thm Exactdim}) and Egorov's theorem, we can find a subset $\Gamma_1 \subset \Lambda \times \Sigma_{+}$ with $(\mu \otimes \nu_{+})(\Gamma_1) > 1 - \eps/2$ and $n_1 = n_1(\cF, C, \eps)$, such that for every $(x, \bi_{+}) \in \Gamma_1$ and $n \geq n_1$, we have \begin{equation}\label{equ:w-lower-bound}
			\mu(x+W(\bi_+,2^{-Cn}))\geq 2^{-(CD_{d+1}+\eps)n},
		\end{equation}
		where we abbreviate $W(\bi_+, 2^{-Cn}) := W(\bi_+, m, 2^{-Cn}, 2^{-Cn})$.
		By the Lebesgue density theorem for cubes 
			applied to $\Gamma'$, there exist a subset $\Gamma_2 \subset \Lambda \times \Sigma_{+}$ with $(\mu \otimes \nu_{+})(\Gamma_2) \geq 1 - 2\eps$ and $n_2 = n_2(\cF, C, \eps)$ such that for any $n \geq n_2$ and any $(x, \bi_{+}) \in \Gamma_2$, we have
			\begin{equation}\label{equ:w-density}
				\frac{\mu((x+W(\bi_+,2^{-Cn}))\cap \Gamma')}{\mu(x+W(\bi_+,2^{-Cn}))}\geq 1-\eps/2 
			\end{equation}

			For any $(x, \bi_{+}) \in \Gamma_1 \cap \Gamma_2$, we verify that the measure 
			\[\omega := (\Dil_{2^{C n}})_*(P^{V_{m+1}(\bi_{+})})_* \mu^{x + W(\bi_{+}, 2^{-Cn}) }\] 
			satisfies the hypothesis of \cref{lem upperfrostman implies regular} for $n\geq \max\{n_1,n_2\}$.
			
			Set $\Psi := \Dil_{2^{Cn}} \circ P^{V_{m+1}(\bi_{+})}$ and $\nu := \mu^{x + W(\bi_+, 2^{-Cn})}$, so that $\omega = \Psi_{*} \nu$. Split
			\aryst
			\nu' := \nu( \, \cdot \, \cap \Gamma') , \qquad \nu'' := \nu( \, \cdot \, \setminus \Gamma') , \qquad \omega' := \Psi_{*} \nu' , \qquad \omega'' := \Psi_{*} \nu'' ,
			\earyst
			so that $\omega = \omega' + \omega''$. By \eqref{equ:w-density}, $|\omega''| = \nu(\Gamma'^{c}) \leq \eps/2$.
			
			Next we bound the balls of $\omega'$. Let $y \in \R^{d-m}$. As noticed above, $(x + W(\bi_+, 2^{-Cn})) \cap \Psi^{-1}(B(y, 2^{-n})) \cap \Gamma'$ can be covered by at most $C_d$ atoms of the form $\cD_{Cn, (C+1)n}^{\bi_+}(z_j)$ with $z_j \in \Gamma'$, and by \eqref{equ:d-upper-bound} and \eqref{equ:w-lower-bound} each of these atoms has $\nu$-mass at most $2^{-(D_{m+1}-\eps)n}$. Hence
			\aryst
			\omega'\big( B(y, 2^{-n}) \big) \ = \ \nu\big( \Psi^{-1}(B(y, 2^{-n})) \cap \Gamma' \big) \ \leq \ C_d \, 2^{-(D_{m+1}-\eps)n} \ \leq \ 2^{-(D_{m+1}-2\eps)n}
			\earyst
			for $n$ large, since $C_d \leq 2^{\eps n}$. Writing $\rho := 2^{-n}$ and $\beta := D_{m+1}$, this says exactly that $\omega$ is $(\beta, \rho, \eps/2)$-upper Frostman.
			
			\Cref{lem upperfrostman implies regular}, applied with $\eta = \eps/2$, $\eta' = \eps$ and exponent $\eps$, therefore gives, for every Borel subset $A$ with $\omega(A) \geq \eps$,
			\aryst
			N_{\rho}(A) \ > \ \rho^{-\beta + \eps} \ =  \rho^{-D_{m+1} + \eps} .
			\earyst
			Since $A$ is an arbitrary subset with $\omega(A) \geq \eps$, we conclude that $\omega$ is $(D_{m+1}, 2^{-n}, \eps, \eps)$-lower regular.
		\end{proof}
		
		The way we  use  the Besicovitch covering argument in the above theorem will repeat several times later. We record here a general version.
		\begin{lemma} \label{lem besicovich}
			Let $\omega$ be a Borel probability measure on $\R^d$, let $\Gamma \subset \R^d$ be a Borel subset, let $\eta \in (0, 1)$, and let $W = S(0, \bE, \br)$ be a slab centered at the origin (see \eqref{eqn:slab definition}). Set
			\aryst
			A := \big \{ x \in \R^d : \omega \big ( (x + W) \cap \Gamma \big ) \geq (1 - \eta) \, \omega(x + W) \big \}.
			\earyst
			Then we have
			\aryst
			\omega(A^{c}) \leq \frac{C_B}{\eta} \, \omega(\Gamma^{c}),
			\earyst
			where $C_B > 0$ is a constant depending only on $d$. 
		\end{lemma}
		\begin{proof}
			We cover $A^c$ by $x+W$ with $x\in A^c$. By Besicovitch, we can find a collection of points $x_j$ with $x_j+W$ covering $A^c$ and with bounded multiplicity $C_B$. Then we have %
			\[\mu(A^c)\leq \sum \mu(x_j+ W)\leq \sum \mu((x_j+ W  )\cap \Gamma^c)/\eta \leq \frac{C_B}{\eta}\mu(\Gamma^c).\]
		\end{proof}

		\section{Fiberwise measure: spreading of the projection measure}\label{sec proofofprop fiberwisedim} 
		
		\subsection{Setting and main result} \label{subsec setting1}

		In this section, we fix some $m \in \{1, \cdots, d_0-1\}$ satisfying $D_{m+1} <  D_{{d_0}+1}$.
		We denote by $m_0 \in \{m, \cdots, d_0 - 1 \}$ the largest number such that $D_{m+1} = D_{m_0 +1}$.
		In particular, by  \cref{cor supportoffiberwisemeasure}, $\mu_x^{V_{m+1}(\bi_+)}=\mu_x^{V_{m_0+1}(\bi_+)}$ for a.e. $(x,\bi_+)$.

		\smallskip
		
		\noindent{\bf $\bullet$ Conditional measures, Dyadic partitions.}

		We denote
		\aryst
		\mu_{x, \bi_{+}} := \mu_x^{V_{m+1}(\bi_{+})}.
		\earyst
		We will abusively regard $\mu_{x, \bi_{+}}$ as an element of $\cP(x + V_{m+1}(\bi_{+})  )$ or $\cP(\R^d)$, via the inclusion $x + V_{m+1}(\bi_{+}) \subset \R^d$.
		
		We will need the restriction of $\mu_{x, \bi_{+}}$ to dyadic cubes at the scales used in the proof. The partitions are chosen as follows. For each $n$, we apply \Cref{lem:avg} (with $d - m$ in place of $d$ and with $(C, N) = (C_0, n)$) to the push-forward of $\mu \otimes \nu_{+}$ under $(x, \bi_{+}) \mapsto ( x_j )_{j \geq m+1}$, and fix a translate $t = t(n) \in \R^{d-m}$ as given there. For each $N$ in the window
		\aryst
		C_0 n \ \leq \ N \ < \ (C_0 + 1) n,
		\earyst
		we then use the partitions $\cD_{m+1, \bi_{+}, N} := \cD^{t(n)}_{m+1, \bi_{+}, N}$ of Subsection \ref{subsec:partitions}; thus each atom is a cube of side length $2^{-N}$ in a translation of $V_{m+1}(\bi_{+})$. Since all partitions for $N$ in this window come from the {\em same} translate, they are increasing in $N$, each refining $\cD_{m+1, \bi_{+}, C_0 n}$; the families chosen for different values of $n$ are unrelated and are never used simultaneously, and we suppress the dependence on $n$ from the notation. 
		The translate $t(n)$ is chosen from \cref{lem:avg} so that the following holds: 
		We set
		\aryst
		\partial_{c_0} \cD_{m+1, \bi_{+}, N}(x):=\{y \in \cD_{m+1, \bi_{+}, N}(x): d(y, \partial \cD_{m+1, \bi_{+}, N}(x)) < c_0 \, 2^{-N}\}.
		\earyst
		where $\partial \cD_{m+1, \bi_{+}, N}(x)$ denotes the boundary of the atom relative to the translation of $V_{m+1}(\bi_{+})$ containing it. 
		For this fixed $n$, we have that for every $c_0 \in (0,1/2]$,
		\begin{equation} \label{eq randomdyadicpartitions}
			\frac{1}{n}\sum_{0\leq k< n} (\mu \otimes \nu_{+})\{ (x, \bi_{+}) : x\in \partial_{c_0} \cD_{m+1, \bi_{+}, C_0n+k}(x))\}  \leq 9 d \sqrt{c_0}.
		\end{equation}

		\smallskip
		
		\noindent{\bf $\bullet$ Ball spreading property.}

		\begin{defi} \label{def spreading}
			Let $n, l \geq 1$ be integers, $\eps > 0$.
			A measure $\zeta \in \cP(\R)$ is {\rm $(n, l, \eps)$-ball spreading} if 
			\aryst
			\zeta \Big\{ x : \frac{1}{n} | \{ 1 \leq k \leq n : \zeta( B(x,2^{-k}) ) \leq 2 \zeta( B(x,2^{-k-l}) ) \}| < \eps  \Big\} > 1 - \eps.
			\earyst
		\end{defi}

		\smallskip

		\noindent{\bf $\bullet$ Rescaled times and self-couplings of $\nu_{+}$.}

		For each $M  \in  \GL_d(\R)$, recall that we denote by $\alpha_1(M),  \cdots, \alpha_d(M)$ the set of singular values of $M$ such that $\alpha_1(M) \geq \cdots \geq \alpha_d(M)$. For a subspace $V \subset \R^d$ of dimension $k$, we write
		\[ \mathrm{J}_V(M) := \frac{\| (\wedge^{k} M) \, v \|}{\|v\|}, \qquad 0 \neq v \in \wedge^{k} V. \]
		
		We denote by $\cB_{k}$ the sub-$\sigma$-algebra on $\Sigma_{+}$ generated by the functions of $\omega = (\omega_i)_{i \in \N}$ depending only on $\omega_0, \cdots, \omega_{k-1}$. 
		Given some integer $n > 0$, we will denote
		\ary
		\label{eqn:Tn c}
		T_n =  \lfloor n \log 2 / (\chi_{m} - \chi_{m_0+1})\rfloor,\qquad 2c=\frac{\chi_{m} - \chi_{m+1}}{\chi_{m} - \chi_{m_0+1}} .
		\eary
		We set
		\aryst
		\cC_n := \cB_{T_n}, \quad n \geq 1.
		\earyst
		By definition, the atom of $\cC_n$ containing $\bi_{+} = (i_{k})_{k \in \N} $, denoted by $\cC_n(\bi_{+})$, is  
		\aryst
		\{ \bj_{+} = (j_{k})_{k \in \N} \in \Sigma_{+} :  j_k  = i_{k}, \ 0 \leq k \leq T_n  - 1 \}.
		\earyst
		We define a self-coupling  $\zeta_{n}$  of $\nu_{+}$ by letting
		\begin{equation}
			\label{eqn:self coupling}
			\zeta_{n}|_{ \cC_n(\bi_{+}) \times \cC_n(\bj_{+}) } = \begin{cases}
				0, & \cC_n(\bi_{+}) \neq \cC_n(\bj_{+}), \\
				\frac{1}{\nu_{+}(\cC_n(\bi_{+}))}\nu_{+} |_{\cC_n(\bi_{+})} \otimes \nu_{+} |_{\cC_n(\bi_{+})},& \cC_n(\bi_{+}) = \cC_n(\bj_{+}).
			\end{cases}
		\end{equation}
		It is straightforward to see that there is a unique self-coupling of $\nu_{+}$ with the above property.

		\smallskip
		
		\noindent{\bf $\bullet$  Main result.}

		The following is the main result of this section, the proof of which will occupy the remainder of this section.
		\begin{prop} \label{prop fiberwisedim}
			There is $\eps_0 = \eps_0(\cF) > 0$ such that the following holds. For every integer $C_0 \gg_{\cF} 1$ and  every $ \eta > 0$, 
			there is an integer $l = l( \cF,   \eta) > 0$ such that for every $\eps \in (0, \eps_0)$,  for every $n \gg_{\cF,   C_0, \eps, \eta} 1$,  there is a Borel subset $\widetilde\Gamma(n) \subset \Lambda \times \Sigma_{+} \times \Sigma_{+}$ with $ (\mu \otimes \zeta_n)   (\widetilde\Gamma(n) ) > 1 - \eta$  
			such that    the following is true.
			\enmt
			\item \label{itm 1P52}  For every $(x, \bi_{+}, \bj_{+}) \in \widetilde\Gamma(n)$, we have  
			\ary \label{eq Vm+1iVm+1jclose}
			d(V_{m_0+1}(\bi_{+}), V_{m+1}(\bj_{+})) \leq 2^{-  (1 - \eps) n},
			\eary
			\begin{equation}\label{equ:angle-Vmm}
				d(V_{m+1}(\bi_{+}), V_{m+1}(\bj_{+})) \mathrel{\leq}  2^{- 2c n + \eps n},
			\end{equation}
			\aryst
			&& \measuredangle( V_{m+1}(\bi_{+}), U^{m-1}(x) ),\ \measuredangle( V_{m+1}(\bj_{+}), U^{m-1}(x) ) \geq 2^{ - \eps n},  
			\earyst
			where $c\in (0,1/2]$ is the constant given in \eqref{eqn:Tn c}.
			\item  For every $(x, \bi_{+}, \bj_{+}) \in { \widetilde\Gamma(n)}$, the probability measure
			$$(\Dil_{2^{(C_0+1)n}})_*(P^{V_{m+1}(\bj_{+}) + U^{m-1}(x)})_* ( \mu_{x, \bi_{+}} )^{\cD_{m+1, \bi_{+}, C_0 n}(x) }$$ is $(n, l, \eta+2\eps)$-ball spreading and the diameter of the support is less than $2^{\eps n}$.
			\eenmt
		\end{prop}
		In the application, we use this scaling factor $\Dil_{2^{(C_0+1)n}}$. For the proof, the following equivalent form of the second conclusion of \cref{prop fiberwisedim} is more convenient: it normalizes by the factor $T(\tilde{x},n)$ produced by the single-scale estimate, rather than by the fixed dilation $2^{(C_0+1)n}$. By \cref{lem Qtospreading} the two differ by at most $2^{\eps n}$, which accounts for the $2\eps$ in its spreading parameter.
		\enmt
		\item[(2')]  {\begingroup\makeatletter\def\@currentlabel{2'}\makeatother\label{itm 2P52}\endgroup}  
		For every $\tilde{x}=(x, \bi_{+}, \bj_{+}) \in \widetilde\Gamma(n)$, there exists $T(\tilde{x},n)>0$ with 
		$$|\log_2 T(\tilde{x},n)-(C_0+1)n|\leq 2\eps n, $$
		such that  the probability measure $$(\Dil_{T(\tilde{x},n)})_*(P^{V_{m+1}(\bj_{+}) + U^{m-1}(x)})_* ( \mu_{x, \bi_{+}} )^{ \cD_{m+1, \bi_{+}, C_0 n}(x)   }$$ is $(n, l, \eta)$-ball spreading.
		\eenmt
		Here $T(\tilde{x},n) = 2^{C_0 n} / \lambda(x, \bi_{+}, \bj_{+})$, with $\lambda$ as in \eqref{eqn:lambda norm} below; this factor is chosen so that the projected atom $(P^{V_{m+1}(\bj_{+}) + U^{m-1}(x)})_*\mu_{x,\bi_+}^{\cD_{m+1, \bi_{+}, C_0 n}(x)}$ has support of diameter at most $\sqrt{d-m}$. 
		
		\smallskip
		
		\noindent{\bf $\bullet$ Strategy of the proof.}
		
		The proof of \cref{prop fiberwisedim} has two steps, and it rests on four ingredients.
		
		\enmt
		\item {\em Step 1: a typical triple satisfies \eqref{itm 1P52}.} We construct a set $\widetilde\Gamma(n)$ of triples with $(\mu \otimes \zeta_n)(\widetilde\Gamma(n)) > 1 - \eta/2$ on which \eqref{itm 1P52} holds. The conditions in \eqref{itm 1P52} are closeness and angle conditions on the Oseledets subspaces, and they hold on a large set by the Oseledets theorems.
		
		\item {\em Step 2: the ball spreading property.} We shrink $\widetilde\Gamma(n)$ to the triples with $Q > 1 - \eta^7$, using \cref{prop fiberwisedimatsinglescale} to pass from atom-goodness to sequence-goodness and \cref{lem goodonaverage} to show that most atoms are good; the set removed has measure at most $4\eta^2$. On what remains, \cref{lem Qtospreading} gives the ball spreading, and hence item \eqref{itm 2P52}.
		
		\eenmt
		
		Step 2 is organised around two notions of goodness. Both concern an atom $D$ of $\cD_{m+1, \bi_{+}, N}$ and the normalized restriction $\mu^{D}_{x, \bi_{+}}$ of $\mu_{x, \bi_{+}}$ to $D$:
		
		\enmt
		\item the atom $D$ is {\em good at level $l$} if $\mu^{D}_{x, \bi_{+}}$ gives mass $< \eta$ to every slab of width $2^{-N-l}$ in the fiber parallel to $V_{m_0+2}(\bi_{+})$ (\cref{def goodcube});
		\item the sequence $\bj_{+}$ is {\em good for $(D,l)$} if the same holds for the slabs parallel to $K = V_{m_0+1}(\bi_{+}) \cap J$, where $J = V_{m+1}(\bj_{+}) + U^{m-1}(x)$ (\cref{def goodsequence}).
		\eenmt
		
		The logical order is then the following. \cref{lem goodonaverage} shows that, for $l$ large, most atoms at most of the scales $C_0 n + k$, $0 \leq k < n$, are good. \Cref{prop fiberwisedimatsinglescale} then passes from the atom to the sequence, so that $\bj_{+}$ is good for most atoms at most scales, and \cref{lem Qtospreading} converts this into the ball spreading of $(P^{J})_* \mu_{x, \bi_{+}}^{\cD_{m+1, \bi_{+}, C_0 n}(x)}$ dilated by $T(\tilde{x}, n)$. \Cref{lem NCcondition} enters here as well: it uses ONC and supplies the transversality of the hyperplane $H'$ to the Oseledets subspace $E_{m_0+1}(x', \bi'_{+})$ that \cref{prop fiberwisedimatsinglescale} requires.
		
		Among the four, \cref{lem goodonaverage} is where the dimension gap $D_{m+1} < D_{d_0+1}$ does its work. It enters in \cref{lem goodonaverage}: by \eqref{eq exactdim}, \aryst \lim_{r \to 0} \frac{ \log \mu_{x, \bi_{+}}( \Tube_r(x, \bi_{+}) ) }{ \log r } = D_{m_0+2} - D_{m_0+1} > 0 , \earyst so the positivity of this exponent is what makes the tubes carry small mass. 
		\smallskip
		
		\begin{center}
			\small
			\begin{tabular}{@{}p{0.31\textwidth}p{0.35\textwidth}p{0.12\textwidth}p{0.12\textwidth}@{}}
				\toprule
				Ingredient & Role & Stated in & Proved in \\
				\Cref{lem NCcondition} & transversality from ONC, used for the ball spreading & \S\ref{subsec projectionmeasure} & \S\ref{subsec projectionmeasure}  \\
				\cref{lem goodonaverage} & most atoms are good; uses the dimension gap & \S\ref{subsec toolbox} & \S\ref{subsec proofsoflemma511} \\
				\Cref{prop fiberwisedimatsinglescale} & atom-goodness $\Rightarrow$ sequence-goodness & \S\ref{subsec toolbox} & \S\ref{subsec toolbox}\\
				\Cref{lem Qtospreading} & goodness at most scales $\Rightarrow$ ball spreading & \S\ref{subsec proofofmainprop} & \S\ref{subsec proofsoflemma512} \\
				\bottomrule
			\end{tabular}
		\end{center}

		\smallskip

		\subsection{Projection of a measure} \label{subsec projectionmeasure}
		We introduce a decomposition of the measure $\mu\otimes\zeta_n$. For a probability measure $\theta$ on $\Sigma_{+}$ and a measurable map $X$ defined $\theta$-a.e.\ on $\Sigma_{+}$, we write $\big( X(\bj_{+}) \big)_{\bj_{+} \sim \theta}$ for the distribution of $X_* \theta$.

		We record the properties of $\zeta_n$ that will be used. All of them follow directly from the definition.
		Write
		\aryst
		\nu_{+}^{\cC_n(\bi_{+})} := \frac{ \nu_{+}( \, \cdot \, \cap \cC_n(\bi_{+}) ) }{ \nu_{+}( \cC_n(\bi_{+}) ) }
		\earyst
		for the normalized restriction of $\nu_{+}$ to the cylinder $\cC_n(\bi_{+})$, and let $t := T_n$.
		\enmt
		\item For every Borel set $A \subset \Lambda \times \Sigma_{+} \times \Sigma_{+}$,
		\aryst
		(\mu \otimes \zeta_n)(A) = \int 1_A(x, \bi_{+}, \bj_{+}) \, d\nu_{+}^{\cC_n(\bi_{+})}(\bj_{+}) \, d(\mu \otimes \nu_{+})(x, \bi_{+}) .
		\earyst
		\item For all measurable $G \subset \Lambda \times \Sigma_{+}$ and $A \subset \Sigma_{+}$,
		\ary \label{eq coupling estimate}
		&&(\mu \otimes \zeta_n) \big\{ (x, \bi_{+}, \bj_{+}) \ : \ (x, \bi_{+}) \in G, \ \bj_{+} \in A \big\} \\
		&\geq& \ 1 - (\mu \otimes \nu_{+})(G^{c}) - \nu_{+}(A^{c}).\nonumber
		\eary
		\item Since $\nu_{+}$ is a Bernoulli measure,
		\ary \label{eqn:sigma nu}
		(\sigma^{t})_* \, \nu_{+}^{\cC_n(\bi_{+})} = \nu_{+} .
		\eary
		For $\bj_{+} \in \cC_n(\bi_{+})$ we set
		\ary \label{eq x'bi'bj'}
		(x', \bi'_{+}, \bj'_{+}) := \big( f^{t}_{\bi_{+}}(x), \ \sigma^{t} \bi_{+}, \ \sigma^{t} \bj_{+} \big) .
		\eary
		If $(x, \bi_{+}, \bj_{+})$ is distributed according to $\mu \otimes \zeta_n$, then $(x', \bi'_{+}, \bj'_{+})$ is distributed according to $\mu \otimes \nu_{+} \otimes \nu_{+}$. For every measurable $B \subset \Lambda \times \Sigma_{+} \times \Sigma_{+}$ and every $(x, \bi_{+})$,
		\aryst
		\nu_{+}^{\cC_n(\bi_{+})} \big\{ \bj_{+} : (x', \bi'_{+}, \bj'_{+}) \in B \big\} = \nu_{+} \big\{ \bj'_{+} : (x', \bi'_{+}, \bj'_{+}) \in B \big\} .
		\earyst
		\eenmt
		\smallskip
		
		Given $W \in \Gr(d, m-1)$, we denote
		\aryst
		H_{W, \bi_{+}}( \bj_{+} ) := V_{m+1}(\bi_{+} ) \cap ( W + V_{m+1}(\bj_{+}) ).
		\earyst
		By \Cref{cor generic hyperplane ae} and Fubini, for $(\mu \otimes \nu_{+})$-a.e.\ $(x, \bi_{+})$, the set of $\bj_{+}$ for which $H_{U^{m-1}(x), \bi_{+}}(\bj_{+})$ is not a hyperplane of $V_{m+1}(\bi_{+})$ is $\nu_{+}$-null; it is then $\nu_{+}^{\cC_n(\bi_{+})}$-null for every $n$, since $\nu_{+}^{\cC_n(\bi_{+})}\ll \nu_+$.

		\noindent{\bf $\bullet$ The map $M$.}
		For $\bj_{+} \in \cC_n(\bi_{+})$, the sequences $\bi_{+}$ and $\bj_{+}$ have the same first $t$ symbols, so that we may set
		\ary \label{eq M def}
		M \ := \ Df^{t}_{\bi_{+}} \ = \ Df^{t}_{\bj_{+}} .
		\eary
		This notation, with $t = T_n$ and $\bj_{+} \in \cC_n(\bi_{+})$, is used throughout the section. With $(x', \bi'_{+}, \bj'_{+})$ as in \eqref{eq x'bi'bj'}, the equivariance of the Oseledets flags (\Cref{thm Oseledets}) yields
		\begin{equation} \label{eq M equivariance}
			M V_{k}(\bi_{+}) = V_{k}(\bi'_{+}), \quad M V_{k}(\bj_{+}) = V_{k}(\bj'_{+}), \quad M U^{k}(x) = U^{k}(x'), \quad M E_{k}(x, \bi_{+}) = E_{k}(x', \bi'_{+})
		\end{equation}
		for every $k$. In particular,
		\ary \label{eq H pullback}
		H_{U^{m-1}(x), \bi_{+}}(\bj_{+}) \ = \ M^{-1} H_{U^{m-1}(x'), \bi'_{+}}(\bj'_{+}) .
		\eary
		For $x \in \Lambda$ and $\bj_{+} \in \Sigma_{+}$, we write
		\begin{equation} \label{eq J def}
			J = J(x, \bj_{+}) := V_{m+1}(\bj_{+}) + U^{m-1}(x) ,
		\end{equation}
		so that $H_{U^{m-1}(x), \bi_{+}}(\bj_{+}) = V_{m+1}(\bi_{+}) \cap J$. For $\bj_{+} \in \cC_n(\bi_{+})$ and $(x', \bi'_{+}, \bj'_{+})$ as above,
		\aryst
		J' := J(x', \bj'_{+}) = MJ \qquad \mbox{and} \qquad H' := H_{U^{m-1}(x'), \bi'_{+}}(\bj'_{+}) = V_{m+1}(\bi'_{+}) \cap J'.
		\earyst

		For every $(x, \bi_{+})$, with $(x', \bi'_{+}, \bj'_{+})$ as in \eqref{eq x'bi'bj'}, we have the equality of laws
		\aryst
		\big( H_{U^{m-1}(x'), \bi'_{+}}(\bj'_{+}) \big)_{\bj_{+} \sim \nu_{+}^{\cC_n(\bi_{+})}}
		\ = \
		\big( H_{U^{m-1}(x'), \bi'_{+}}(\bj'_{+}) \big)_{\bj'_{+} \sim \nu_{+}},
		\earyst
		because $\bj'_{+} = \sigma^{t} \bj_{+}$ and \eqref{eqn:sigma nu}; both sides are probability measures on the set of subspaces of $V_{m+1}(\bi'_{+})$.
		For $(\mu \otimes \nu_{+})$-a.e.\ $(x, \bi_{+})$, this measure is supported on $\Gr(V_{m+1}(\bi'_{+}), d-m-1)$, by \Cref{cor generic hyperplane ae} applied at $(x', \bi'_{+})$, which is admissible for a.e.\ $(x, \bi_{+})$ since $(x, \bi_{+}) \mapsto (x', \bi'_{+})$ preserves $\mu \otimes \nu_{+}$.

		We have the following. 
		\begin{lemma} \label{lem NCcondition}
			\label{cor NCcondition shifted}
			Assume that $\cF$ has the ONC at level $d_0$. For every $\eta > 0$, there is $\delta = \delta(\cF, \eta) > 0$ such that the following holds. For every $n \geq 1$, there is a set $\Gamma_0(n) \subset \Lambda \times \Sigma_{+}$ with $(\mu \otimes \nu_{+})(\Gamma_0(n)) > 1 - \eta$ such that, for every $q \in \{m+1, \cdots, d_0\}$ and every $(x, \bi_{+}) \in \Gamma_0(n)$, with $(x', \bi'_{+}, \bj'_{+})$ as in \eqref{eq x'bi'bj'},
			\ary \label{eq NC section shifted}
			\nu_{+}^{\cC_n(\bi_{+})} \Big\{ \bj_{+} \ : \sin \measuredangle\big( H_{U^{m-1}(x'), \bi'_{+}}(\bj'_{+}), \, E_q(x', \bi'_{+}) \big) > \delta \Big\} \ > \ 1 - \eta .
			\eary
		\end{lemma}
		
		\begin{proof}
			Fix $q \in \{m+1, \cdots, d_0\}$ and write
			$$\Theta_q(x, \bi_{+}, \bj_{+}) := \sin \measuredangle\big( H_{U^{m-1}(x), \bi_{+}}(\bj_{+}), E_q(x, \bi_{+}) \big),$$
			a measurable function defined for $(\mu \otimes \nu_{+} \otimes \nu_{+})$-a.e.\ triple.
			It follows from \Cref{cor:theta q} that $\Theta_q > 0$ almost everywhere.
			Using the continuity of measure, there exists $\delta_q = \delta_q(\cF, \eta) > 0$ such that
			\aryst
			(\mu \otimes \nu_{+} \otimes \nu_{+}) \{ \Theta_q \leq \delta_q \} \ < \ \Big( \frac{\eta}{d_0 - m} \Big)^2 .
			\earyst
			By Fubini and Markov's inequality, the set
			\aryst
			\Gamma(q) := \Big\{ (x, \bi_{+}) \ : \ \nu_{+} \{ \bj_{+} : \Theta_q(x, \bi_{+}, \bj_{+}) \leq \delta_q \} < \frac{\eta}{d_0 - m} \Big\}
			\earyst
			has $(\mu \otimes \nu_{+})$-measure $> 1 - \frac{\eta}{d_0 - m}$.
			Set $\delta := \min\{\delta_q : m+1\leq q\leq d_0\}$ and $\Gamma := \bigcap_{m+1\leq q\leq d_0} \Gamma(q)$.
			Then, $(\mu \otimes \nu_{+})(\Gamma) > 1 - \eta$, and for $(x, \bi_{+}) \in \Gamma$ and every $q$, 
			\begin{equation}
				\label{eq NC section} 
				\nu_{+} \{ \bj_{+} : \Theta_q(x, \bi_{+}, \bj_{+}) > \delta \} \ \geq \ \nu_{+} \{ \bj_{+} : \Theta_q(x, \bi_{+}, \bj_{+}) > \delta_q \} \geq \ 1 - \eta.
			\end{equation}
			Set $$\Gamma_0(n) := \{ (x, \bi_{+}) : (x', \bi'_{+}) \in \Gamma \}.$$ We have
			$$(\mu \otimes \nu_{+})(\Gamma_0(n)) = (\mu \otimes \nu_{+})(\Gamma) > 1 - \eta .$$
			Now fix $(x, \bi_{+}) \in \Gamma_0(n)$ and $q$.
			Since $\bj'_{+} = \sigma^{T_n} \bj_{+}$ and $(x', \bi'_{+})$ does not depend on $\bj_{+}$, the set on the LHS of the equation below is the preimage under $\sigma^{T_n}$ of the set on the RHS, so \eqref{eqn:sigma nu} gives
			\aryst
			&&\nu_{+}^{\cC_n(\bi_{+})} \Big\{ \bj_{+} : \sin \measuredangle\big( H_{U^{m-1}(x'), \bi'_{+}}(\bj'_{+}), E_q(x', \bi'_{+}) \big) > \delta \Big\}\\ 
			&=& \nu_{+} \Big\{ \bj'_{+} : \sin \measuredangle\big( H_{U^{m-1}(x'), \bi'_{+}}(\bj'_{+}), E_q(x', \bi'_{+}) \big) > \delta \Big\}
			>1 - \eta ,
			\earyst
			the inequality being \eqref{eq NC section} at $(x', \bi'_{+}) \in \Gamma$. 
		\end{proof}

		\subsection{Goodness, transfer and the single-scale estimate} \label{subsec toolbox}
		We collect here the notation and the auxiliary results used in the proof of \cref{prop fiberwisedim}.
		
		\noindent{\bf $\bullet$ Tubes.}
		Let $\Xi \subset \Lambda \times \Sigma_{+}$ be the full measure set of $(x, \bi_{+})$ at which \cref{cor supportoffiberwisemeasure} applies and $U^{m-1}(x) \cap V_{m+1}(\bi_{+}) = \{0\}$. For $(x, \bi_{+}) \in \Xi$, the measure $\mu_{x, \bi_{+}}$ is supported on
		\ary \label{eq A def}
		A(x, \bi_{+}) := x + V_{m_0+1}(\bi_{+}) .
		\eary
		For $(x, \bi_{+}) \in \Xi$ and $\bj_{+} \in \Sigma_{+}$ with $U^{m-1}(x) \cap V_{m+1}(\bj_{+}) = \{0\}$, the space $J = J(x, \bj_{+})$ {of \eqref{eq J def}} is a hyperplane of $\R^d$, and we put
		\begin{align}
			\label{eqn:lambda norm}
			\  K&=K(x, \bi_{+}, \bj_{+}) := V_{m_0+1}(\bi_{+}) \cap J = V_{m_0+1}(\bi_{+})  \cap (V_{m+1}(\bj_{+}) + U^{m-1}(x)),\\
			\label{eqn:lambda norm2}
			\  \lambda&=\lambda(x, \bi_{+}, \bj_{+}) := \sup \big\{ \|P^{J} w\| \ : \ w \in V_{m_0+1}(\bi_{+}), \ \|w\| = 1 \big\}. 
		\end{align}
		When $(x, \bi_{+})$ and $\bj_{+}$ are clear from the context, we simply write $A$, $J$, $K$, and $\lambda$.

		For $(x, \bi_{+}) \in \Xi$, $y \in A(x, \bi_{+})$, a hyperplane $W$ of $V_{m_0+1}(\bi_{+})$ and $r > 0$, we write
		\begin{equation} \label{eq slab def}
			\widehat\Tube_r(y, W) := \{ z \in A(x, \bi_{+}) \ : \ d(z - y, W) < r \} \quad \mbox{and} \quad \Tube_r(y, W) := \widehat\Tube_r(y, W) \cap \Lambda
		\end{equation}
		for the tube of width $r$ around $y + W$ in the fiber $A(x, \bi_{+})$ and for its trace on $\Lambda$. For $W = V_{m_0+2}(\bi_{+})$, we abbreviate $\widehat\Tube_r(y, \bi_{+}) := \widehat\Tube_r(y, V_{m_0+2}(\bi_{+}))$ and $\Tube_r(y, \bi_{+}) := \Tube_r(y, V_{m_0+2}(\bi_{+}))$.
		\smallskip

		\noindent{\bf $\bullet$ Good atoms.}
		Let $\bi_{+} \in \Sigma_{+}$, let $N \geq 1$ be an integer, and let $D$ be an atom of $\cD_{m+1, \bi_{+}, N}$.
		The atom $D$ is contained in a single fiber $x + V_{m+1}(\bi_{+})$, and the conditional measure $\mu_{x, \bi_{+}} = \mu_x^{V_{m+1}(\bi_{+})}$ depends only on this fiber. Consequently, whether $\mu_{x, \bi_{+}}(D) > 0$ occurs and, if so, the normalized restriction
		\aryst
		\mu^{D}_{x, \bi_{+}} := \frac{ \mu_{x, \bi_{+}}( \, \cdot \, \cap D ) }{ \mu_{x, \bi_{+}}(D) }
		\earyst
		does not depend on the choice of $x \in D$.
		
		\begin{defi} \label{def goodcube}
			Let $l \geq 1$ be an integer and $0 < \eta \leq 1$. Let $\bi_{+} \in \Sigma_{+}$, $N \geq 1$, and let $D$ be an atom of $\cD_{m+1, \bi_{+}, N}$ with $\mu_{x, \bi_{+}}(D) > 0$ for $x \in D$.
			We say that $D$ is {\em $\eta$-good at level $l$} if writing $\zeta_D := (P^{V_{m_0+2}(\bi_{+})})_* \mu^{D}_{x, \bi_{+}}$, 
			\ary \label{eq good atom}
			\sup_{p} \zeta_D(B(p, 2^{-N-l})) < \eta.
			\eary
			By the above, this is a property of $D$, $N$ and $\bi_{+}$ only.
		\end{defi}
		In terms of the tubes \eqref{eq slab def}, for $(x, \bi_{+}) \in \Xi$ and $D = \cD_{m+1, \bi_{+}, N}(x)$ with $\mu_{x, \bi_{+}}(D) > 0$, condition \eqref{eq good atom} reads
		\ary \label{eq good atom tube}
		\mu^{D}_{x, \bi_{+}} \big( \Tube_{2^{-N-l}}(y, \bi_{+}) \big) \ < \ \eta \qquad \mbox{for every } y \in A(x, \bi_{+}) .
		\eary

		If $D$ is $\eta$-good at level $l$, we have
		\ary
		\label{eqn:supp zeta 1}
		\diam \Big( \supp \zeta_D\Big) \ \in \ 2^{-N} \big[ 2^{-l}, \ \sqrt{d} \, \big] .
		\eary
		The upper bound follows from $\diam (D) \leq \sqrt{d} \, 2^{-N}$. For the lower bound, if $\diam(\supp \zeta_D) < 2^{-N-l}$, a single ball of radius $2^{-N-l}$ centered in $\supp \zeta_D$ would have $\zeta_D$-mass $1 \geq \eta$. 
		
		One important property of goodness:
		\begin{prop} \label{lem goodonaverage}
			For every $\eta > 0$, there is an integer $l = l(\cF, \eta) > 0$ such that  for any $n \gg_{\cF,   \eta} 1$, we have
			\aryst
			\frac{1}{n} \sum_{k = 0}^{n-1}  
			\mu \otimes \nu_{+} \{(x,\bi_+):\ \mbox{$\cD_{ m+1,\bi_{+}, C_0n+ k}(x)$ is $\eta^8$-good at level $l$} \}  \geq  1 - \eta^{9}.
			\earyst
		\end{prop}
		
		\smallskip
		\noindent{\bf $\bullet$ Good sequences for an atom.}
		\begin{defi} \label{def goodsequence}
			Let $l \geq 1$ be an integer and $0 < \eta \leq 1$. Let $(x, \bi_{+}) \in \Xi$, $N \geq 1$, and let $D$ be an atom of $\cD_{m+1, \bi_{+}, N}$ contained in the fiber $x + V_{m+1}(\bi_{+})$ with $\mu_{x, \bi_{+}}(D) > 0$. Let $\bj_{+}$ be such that {$U^{m-1}(x) \cap V_{m+1}(\bj_{+}) = \{0\}$ and} $K(x, \bi_{+}, \bj_{+})$ (given by \eqref{eqn:lambda norm}) is a hyperplane of $V_{m_0+1}(\bi_{+})$.
			We say that $\bj_{+}$ is {\em $\eta$-good for $(D, l)$ with respect to $(x, \bi_{+})$} if
			\ary \label{eq good sequence}
			\sup_{p} \, (P^{V_{m+1}(\bj_{+}) + U^{m-1}(x)})_* \mu^{D}_{x, \bi_{+}} \big( B(p, \lambda 2^{-N-l}) \big) < \eta.
			\eary
			Recall that $\lambda$ in the above expression is given by \eqref{eqn:lambda norm2}.
			Unlike \cref{def goodcube}, this definition depends on $x$ through $U^{m-1}(x)$.
		\end{defi}
		As before, if $\bj_{+}$ is $\eta$-good for $(D, l)$ with respect to $(x, \bi_{+})$, then $$\diam \big( P^{J}(\supp \mu^{D}_{x, \bi_{+}}) \big) \in \lambda 2^{-N} [2^{-l}, \sqrt{d}].$$

		The following lemma, whose proof is given in \cref{app linear algebra}, transfers good cubes between nearby hyperplanes.  
		\begin{lemma}\label{lem slab transfer}
			Let $(x, \bi_{+}) \in \Xi$, and let $D$, $l$, $\eta$, $\bj_{+}$ and $K = K(x, \bi_{+}, \bj_{+})$ be as in \cref{def goodcube} and \cref{def goodsequence}.
			If $D$ is $\eta$-good at level $l$ and
			\aryst
			\| P^{K} - P^{V_{m_0+2}(\bi_{+})} \| \ \leq \ \frac{ 2^{-l} }{ \sqrt{d} } ,
			\earyst
			then $\bj_{+}$ is $4 \eta$-good for $(D, l)$ with respect to $(x, \bi_{+})$.
		\end{lemma}

		The following lemma, whose proof is also given in \cref{app linear algebra}, pulls the hyperplane $H_{U^{m-1}(x'), \bi'_{+}}(\bj'_{+})$ back from time $T_n$ to time $0$, where it is close to $V_{m_0+2}(\bi_{+})$.
		
		\begin{lemma}\label{lem kernel pullback}
			Let $n \geq 1$, $t := T_n$, and $\eps > 0$. Let $(x, \bi_{+}) \in \Xi$ and $\bj_{+} \in \cC_n(\bi_{+})$. Let $(x', \bi'_{+}, \bj'_{+})$ be as in \eqref{eq x'bi'bj'}, and let $M = Df^{t}_{\bi_{+}} = Df^{t}_{\bj_{+}}$ be the map \eqref{eq M def}. Assume that
			\enmt
			\item[(a)] $\big\| M |_{V_{m_0+2}(\bi_{+})} \big\| \leq e^{t(\chi_{m_0+2} + \eps)}$ and $\| M v' \| \geq e^{t(\chi_{m_0+1} - \eps)} \| v' \|$ for $v' \in E_{m_0+1}(x, \bi_{+})$;
			\item[(b)] $H' = H_{U^{m-1}(x'), \bi'_{+}}(\bj'_{+})$ defined in \eqref{eq J def} is a hyperplane of $V_{m+1}(\bi'_{+})$, and
			\ary \label{equ:hum-1}
			\sin \measuredangle\big( H', \, E_{m_0+1}(x', \bi'_{+}) \big) \ > \ e^{-\eps n} .
			\eary
			\eenmt
			Then $K = K(x, \bi_{+}, \bj_{+})$ defined in \eqref{eqn:lambda norm} 
			is a hyperplane of $V_{m_0+1}(\bi_{+})$, and
			\ary \label{eq kernel close}
			d\big( V_{m_0+2}(\bi_{+}), \, K \big) \ \leq \ e^{\eps n} \, e^{- t (\chi_{m_0+1} - \chi_{m_0+2} - 2 \eps)} .
			\eary
		\end{lemma}
		
		We use the following result on the single scale information on the projected measure.
		
		\begin{prop} \label{prop fiberwisedimatsinglescale}
			Assume that $C_0 \gg_{\cF} 1$.
			There are $\eps_0 = \eps_0(\cF) > 0$ and $ C_1 =  C_1(\cF) > 0$ such that the following holds.
			Let $\eta > 0$, $\eps \in (0, \eps_0)$, and let $l \geq 1$ be an integer.
			For every integer $n \gg_{\cF,   \eta, \eps, l} 1$, there exists a set $\Gamma_n \subset \Lambda \times \Sigma_{+}$ with $(\mu \otimes \nu_{+})(\Gamma_n) > 1 - \eta^{9}$ such that the following holds. For every $(x, \bi_{+}) \in \Gamma_n$, there exists a set $B_n(x, \bi_{+}) \subset \cC_n(\bi_{+})$ with $\nu_{+}^{\cC_n(\bi_{+})}\big( B_n(x, \bi_{+}) \big) \geq 1 - \eta^{9}$ and the following two properties hold for every $\bj_{+} \in B_n(x, \bi_{+})$:
			\enmt
			\item for every $0 \leq k<n$ and every atom $D$ of $\cD_{m+1, \bi_{+}, C_0 n + k}$ contained in $\cD_{m+1, \bi_{+}, C_0 n}(x)$, if $D$ is $\eta^{8}$-good at level $l$, then $\bj_{+}$ is $\eta^{4}$-good for $(D, l)$ with respect to $(x,\bi_+)$;
			\item $\lambda(x, \bi_{+}, \bj_{+}) \geq 2^{-n} e^{-   C_1   \eps n}$ with $\lambda(x,\bi_+,\bj_+)$ given in \eqref{eqn:lambda norm}. 
			\eenmt
		\end{prop}

		\begin{proof}

			If $m_0 = d-1$, then $\chi_{m_0+2}$, $V_{m_0+2}(\bi_{+})$ and $\kappa$ are undefined; we adopt the conventions ($\chi_{d+1} = \chi_d-1$, $V_{d+1} = \{0\}$).

			Now assume that $m_0 < d -1$.
			Put
			\begin{equation}
				\label{eqn:kappa c1 epsilon0}
				\kappa := \frac{ \chi_{m_0+1} - \chi_{m_0+2} }{ \chi_m - \chi_{m_0+1} } > 0 ,
				\quad
				c_1 := \frac{2}{\chi_m - \chi_{m_0+1}} + \frac{1}{\log 2} ,
				\quad
				\eps_0 := \min \Big\{ \frac{ \chi_{m_0+1} - \chi_{m_0+2} }{2}, \ \frac{\kappa}{2 c_1} \Big\}.
			\end{equation}
			These constants depend on $\cF$ only. Throughout the proof, $\eps \in (0, \eps_0)$ is the number given in the statement.

			We fix an arbitrary $\eta > 0$.
			We let $n \gg_{\cF,   \eta} 1$ be a large integer.
			
			Recall that $\Xi$ has full $(\mu \otimes \nu_{+})$-measure. Since the marginal of $\mu \otimes \zeta_n$ on $(x, \bj_{+})$ is $\mu \otimes \nu_{+}$, the triples with $(x, \bj_{+}) \in \Xi$ have full $(\mu \otimes \zeta_n)$-measure.
			
			Throughout the proof, $t = T_n$ and $(x', \bi'_{+}, \bj'_{+})$ are as in \eqref{eq x'bi'bj'}.
			We need a second set of triples.
			Let $\widetilde\Gamma_3(n) \subset \Lambda \times \Sigma_{+} \times \Sigma_{+}$ be the set of $(x, \bi_{+}, \bj_{+})$ with $\bj_{+} \in \cC_n(\bi_{+})$ such that, with $J = V_{m+1}(\bj_{+}) + U^{m-1}(x)$, $J' = V_{m+1}(\bj'_{+}) + U^{m-1}(x') = MJ$ and $H' = V_{m+1}(\bi'_{+}) \cap J'$ as in \eqref{eq J def},
			\enmt
			\item[(i)] $\| M v \| \leq e^{t(\chi_m + \eps)} \| v \|$ for $v \in E_m(x, \bj_{+})$;
			\item[(ii)] $\measuredangle\big( E_m(x, \bj_{+}), \, J \big) > e^{-\eps n}$ and $\measuredangle\big( E_m(x', \bj'_{+}), \, J' \big) > e^{-\eps n}$;
			\item[(iii)] $\sin \measuredangle( \ell', J' ) > e^{-\eps n}$, where $\ell' := V_{m+1}(\bi'_{+}) \cap (H')^{\perp}$.
			\eenmt
			By \Cref{thm Oseledets}, each of the conditions (i) and (ii) fails on a set of $\mu \otimes \zeta_n$-measure $< \eta^{21}/400$ once $n \gg 1$.
			For (iii), $\measuredangle(\ell', J') > 0$ if and only if $H'$ is a proper hyperplane of $V_{m+1}(\bi'_{+})$, which holds a.e. by \Cref{cor generic hyperplane ae} applied to $(x', \bi'_{+}, \bj'_{+})$; by the continuity of measure, the sublevel set $\{ \sin \measuredangle(\ell', J') \leq e^{-\eps n} \}$ has measure $< \eta^{21}/400$ once $n \gg 1$.
			Hence
			\aryst
			(\mu \otimes \zeta_n) \big( \widetilde\Gamma_3(n) \big) \ > \ 1 - \eta^{21} / 100 .
			\earyst
			
			We let $\Gamma_n \subset \Lambda \times \Sigma_{+}$ be the set of $(x, \bi_{+})$ in $\Xi$ such that:
			\enmt
			\item $\big\| Df^{t}_{\bi_{+}} |_{V_{m_0+2}(\bi_{+})} \big\| \leq e^{t(\chi_{m_0+2} + \eps)}$ and $\| Df^{t}_{\bi_{+}} v' \| \geq e^{t(\chi_{m_0+1} - \eps)} \| v' \|$ for $v' \in E_{m_0+1}(x, \bi_{+})$;
			\item $(x, \bi_{+}) \in \Gamma_0(n)$, the set given by \cref{lem NCcondition} with $\eta^{10}$ in place of $\eta$; in particular, with $\delta = \delta(\cF, \eta^{10})$ from that lemma,
			\aryst
			\nu_{+}^{\cC_n(\bi_{+})} \Big\{ \bj_{+} \ : \ \measuredangle\big( H', \, E_{m_0+1}(x', \bi'_{+}) \big) > \delta \Big\} \ > \ 1 - \eta^{10},
			\earyst
			\item {the set of $\bj_{+}$ with $(x, \bj_{+}) \in \Xi$ and} $(x, \bi_{+}, \bj_{+}) \in \widetilde\Gamma_3(n)$ has $\nu_{+}^{\cC_n(\bi_{+})}$-measure at least $1 - \eta^{10}$.
			\eenmt
			We prove the proposition with this set $\Gamma_n$. We first show that for $n \gg 1$,
			\ary
			\label{eqn:Gamma n estimate}
			(\mu \otimes \nu_{+})(\Gamma_n) \ > \ 1 - 2 \eta^{10} \ > \ 1 - \eta^{9}.
			\eary
			By \Cref{thm Oseledets} and \cref{cor Oseledets restriction}, item (1) holds outside a set of $(x, \bi_{+})$ of $\mu \otimes \nu_{+}$-measure $< \eta^{11}$ once $n\gg 1$. For item (2), \cref{lem NCcondition} yields that $(\mu\otimes \nu_+)(\Gamma_0(n))>1-\eta^{10}$. For item (3), we have $$(\mu\otimes \zeta_n)\big(\widetilde\Gamma_3(n)^c\big)<\eta^{21}/100,$$ since the triples with $(x, \bj_{+}) \in \Xi$ have full measure.
			By Markov's inequality and the fact that $\mu \otimes \zeta_n$ disintegrates over $(x, \bi_{+})$ with fiber measures $\nu_{+}^{\cC_n(\bi_{+})}$, item (3) fails on a set of $\mu\otimes \nu_+$-measure $<2\eta^{11}$.   This justifies \eqref{eqn:Gamma n estimate}.
			
			Fix an arbitrary $(x, \bi_{+}) \in \Gamma_n$ and set 
			\aryst
			B_n(x, \bi_{+}):=\{\bj_{+} \in \cC_n(\bi_{+})\mid { (x, \bj_{+}) \in \Xi,}\ (x, \bi_{+}, \bj_{+}) \in \widetilde\Gamma_3(n)\,\,\text{and}\,\,\eqref{equ:hum-1}\,\, \text{holds}\}.
			\earyst
			Since $(x, \bj_{+}) \in \Xi$, we have $U^{m-1}(x) \cap V_{m+1}(\bj_{+}) = \{0\}$, and hence $J = J(x, \bj_{+})$ is a hyperplane of $\R^d$ and $\lambda(x, \bi_{+}, \bj_{+})$ is defined.
			By items (2) and (3) of the definition of $\Gamma_n$, and since $\sin \measuredangle(H', E_{m_0+1}(x', \bi'_{+})) > \delta(\cF, \eta^{10})$ implies \eqref{equ:hum-1} once $n\gg 1$ such that $e^{-\eps n} <\delta(\cF, \eta^{10})$, we have
			\aryst
			\nu_{+}^{\cC_n(\bi_{+})} (B_n(x,\bi_+))\geq 1 - 2 \eta^{10} \geq 1 - \eta^{9}.
			\earyst
			Let $K = K(x, \bi_{+}, \bj_{+})$ be as in \eqref{eqn:lambda norm}.
			Item (1) of the definition of $\Gamma_n$ is hypothesis (a) of \cref{lem kernel pullback}, and \eqref{equ:hum-1} is hypothesis (b). Hence $K$ is a hyperplane of $V_{m_0+1}(\bi_{+})$ and \eqref{eq kernel close} holds. Recall the constants $\kappa,c_1,\epsilon_0$ defined in \eqref{eqn:kappa c1 epsilon0}.
			Since $t = T_n \geq n \log 2 / (\chi_m - \chi_{m_0+1}) - 1$ and $\epsilon<\epsilon_0$, once $n \gg_{\cF, l} 1$,
			\aryst
			&&\| P^{K} - P^{V_{m_0+2}(\bi_{+})} \| \ = \  d\big( V_{m_0+2}(\bi_{+}), K \big)\\
			&\leq& e^{\chi_{m_0+1} - \chi_{m_0+2}} \, 2^{-(\kappa - c_1 \eps) n}\leq \ 2^{- \kappa n / 2} \ \leq \ \frac{ 2^{-l} }{ \sqrt{d} },
			\earyst
			which is the hypothesis of \cref{lem slab transfer}.
			Since $D$ is $\eta^{8}$-good at level $l$, \cref{lem slab transfer} shows that $\bj_{+}$ is $4 \eta^{8}$-good for $(D, l)$ with respect to $(x, \bi_{+})$, hence $\eta^{4}$-good for $(D, l)$ as $\eta \ll 1$.
			This proves the first assertion of the proposition.
			
			Let $M$, $J$, $J' = MJ$ and $H'$ be as in \eqref{eq M def} and \eqref{eq J def}, and let $\ell'$ be as in condition (iii) of $\widetilde\Gamma_3(n)$.
			Fix a unit vector $e \in E_{m_0+1}(x, \bi_{+})$ and set $e' := M e / |M e|$, a unit vector of $E_{m_0+1}(x', \bi'_{+})$.
			By definition, $\lambda = \lambda(x, \bi_{+}, \bj_{+}) \geq | P^{J} e |$.
			
			We apply \cref{lem lowerboundforpojection}, \eqref{eq proj ratio}, with this $M$ and $J$, with $L := E_m(x, \bj_{+})$, so that $ML = E_m(x', \bj'_{+})$, with $v := e$ and a nonzero $w \in L$.
			The angle hypothesis \eqref{eqn:angle bound} of the lemma holds with $\rho = e^{-\eps n}$ by condition (ii) of $\widetilde\Gamma_3(n)$, and we obtain
			\ary \label{eq lambda lower 1}
			| P^{J} e | \ \geq \ e^{-\eps n} \cdot | P^{J'} e' | \cdot \frac{ |M e| }{ |M w| / |w| } .
			\eary
			We bound the three factors.
			For the second, write $e' = h + u$ with $h \in H'$ and $u \in \ell'$. Then $d(e', H') = |u|$, and since $h \in J'$,
			\aryst
			| P^{J'} e' | \ = \ d(e', J') \ = \ d(u, J') \ = \ |u| \, \sin \measuredangle(\ell', J') \ \geq \ e^{-\eps n} \, d(e', H') \ > \ e^{-2 \eps n} ,
			\earyst
			by condition (iii) of $\widetilde\Gamma_3(n)$ and \eqref{equ:hum-1}.
			For the third, item (1) of the definition of $\Gamma_n$ gives $|M e| \geq e^{t (\chi_{m_0+1} - \eps)}$, and condition (i) of $\widetilde\Gamma_3(n)$ gives $|M w| / |w| \leq e^{t (\chi_m + \eps)}$.
			Inserting these bounds into \eqref{eq lambda lower 1} and using $t = T_n \leq \frac{n \log 2}{\chi_m-\chi_{m_0+1}}$, 
			\ary \label{eq lambda lower}
			\lambda \ \geq \ | P^{J} e | \ \geq \ e^{-3 \eps n} \, e^{- 2 \eps t} \, e^{- t (\chi_m - \chi_{m_0+1})} \ \geq \ 2^{-n} \, e^{- { C_1} \, \eps n},
			\eary
			where ${ C_1} := 3 + \frac{2 \log 2}{\chi_m-\chi_{m_0+1}} $. This proves the second assertion of the proposition and completes the proof.
		\end{proof}

		\subsection{Proof of \Cref{prop fiberwisedim}} \label{subsec proofofmainprop}
		
		The proof follows the roadmap given at the end of \S\ref{subsec setting1}.  We carry out the steps in turn.

		\subsubsection*{Step 1:}
		
		Let $(\zeta_n \in \cP(\Sigma_{+}\times \Sigma_{+}) )_{n \in \N}$ be the sequence of self-couplings of $\nu_{+}$ given in \eqref{eqn:self coupling}. 
		
		Fix $C_0, \eps, \eta > 0$. Without loss of generality, we may let $\eta \ll_{\cF } 1$.
		Let $l > 0$ be an integer, which will only be used in Step 2, and will  eventually be determined by \cref{lem goodonaverage}. To simplify notation, throughout the proof, we use  ``$n \gg 1$'' to mean $n \gg_{\cF,   C_0, \eps, \eta, l} 1$.

		Let $\Gamma_0(n)$ be given by \cref{lem NCcondition} with $\eta^{10}$ in place of $\eta$, so that $$(\mu \otimes \nu_{+})(\Gamma_0(n)) > 1 - \eta^{10}.$$
		
		Let $\widetilde\Gamma_1(n)$ be the set of $(x, \bi_{+}, \bj_{+})$ satisfying both \eqref{eq Vm+1iVm+1jclose} and \eqref{equ:angle-Vmm}.
		
		By \Cref{thm Oseledets} and by letting $n \gg 1$, there is a subset $\widetilde\Gamma_2(n) \subset \Lambda \times \Sigma_{+} \times \Sigma_{+}$ with $(\mu \otimes \zeta_n)(\widetilde\Gamma_2(n)) > 1 - \eta^{21}/100$ such that for every $(x, \bi_{+}, \bj_{+}) \in \widetilde\Gamma_2(n)$, we have
		\ary \label{eq angleseparation}
		&& \measuredangle( U^{m}(x), V_{m+1}(\bi_+)), \measuredangle( U^{m}(x), V_{m+1}(\bj_+)), \\
		&& \measuredangle( U^{m}(x'), V_{m+1}(\bi'_+)), \measuredangle( U^{m}(x'), V_{m+1}(\bj'_+)) > 2^{-\eps n}. \nonumber
		\eary
		Here $(x', \bi'_{+}, \bj'_{+})$ is given by \eqref{eq x'bi'bj'}.
		
		We let
		\ary
		\widetilde\Gamma(n):= { (\Gamma_0(n) \times \Sigma_{+}) \cap \widetilde\Gamma_1 (n) \cap \widetilde\Gamma_2 (n)}.
		\eary
		
		\textbf{Let us first prove item \eqref{itm 1P52} in \Cref{prop fiberwisedim}:} the first two displays of item \eqref{itm 1P52} are the two conditions defining $\widetilde\Gamma_1(n)$, and the third display follows from \eqref{eq angleseparation}, since $U^{m-1}(x) \subset U^{m}(x)$.
		
		We now prove that
		\ary \label{eq measurelower}
		(\mu \otimes \zeta_n)(\widetilde\Gamma (n)) > 1 - \frac{\eta}{2}.
		\eary
		By construction, and because the marginal of $\mu \otimes \zeta_n$ on $(x, \bi_{+})$ is $\mu \otimes \nu_{+}$, the $(\mu \otimes \zeta_n)$-measures of ${\Gamma_0 (n)\times \Sigma_{+}}$ and ${\widetilde\Gamma_2 (n)}$ are at least $1 - \eta^{10}$. It suffices to show that 
		\ary
		(\mu \otimes \zeta_n)(\widetilde\Gamma_1 (n)) > 1 - \eta/10.
		\eary
		
		Take a typical $(x, \bi_{+}, \bj_{+})$ with respect to $\mu \otimes \zeta_n$. By definition of $\zeta_n$, $\bj_{+} \in \cC_n(\bi_{+})$; let $M = Df^{t}_{\bi_{+}} = Df^{t}_{\bj_{+}}$ be the map \eqref{eq M def}, with $t = T_n$.
		Let $V := V_{m+1}(\bj_{+})$. We consider the following five conditions on $M$:
		\ary
		\label{eq svgap}
		\alpha_{m}( M ) &> &2^{(1-\frac{\eps }{3})n} \alpha_{m_0+1}(M),\\
		\label{eq svgap2}
		\alpha_m(M) &>& 2^{(2{c} - \frac{\eps}{3}) n} \, \alpha_{m+1}(M) , \\
		\label{equ dfjv}
		\mathrm{J}_{V}(M)&\leq& 2^{\frac{\eps n}{3}}\alpha_{m+1}(M)\cdots \alpha_d(M), \\
		\label{equ dfiv}
		\sup_{v \in V_{m_0+1}(\bi_{+}) \setminus \{0\}} \frac{\|Mv\|}{\|v \|}&\leq& 2^{\frac{\eps n}{3} }\alpha_{m_0+1}(M),\\
		\label{equ dfiv2}
		\sup_{v \in V_{m+1}(\bi_{+}) \setminus \{0\}} \frac{\| M v \|}{\| v \|} &\leq& 2^{\frac{\eps n}{3}} \, \alpha_{m+1}(M) .
		\eary
		
		We first show that the conditions imply \eqref{eq Vm+1iVm+1jclose} and \eqref{equ:angle-Vmm} of \cref{prop fiberwisedim}.  Assume \eqref{eq svgap}--\eqref{equ dfiv2}, and let $L$ be a line in $V_{m_0+1}(\bi_{+})$.
		Then \cref{lem distance-control}, applied with $m' = m_0$ and $C = 2^{\eps n / 3}$ (its hypotheses are \eqref{equ dfiv} and \eqref{equ dfjv}), together with \eqref{eq svgap} gives
		\aryst
		d\big( V_{m+1}(\bj_{+}), L \big) \ \leq \ 2^{2 \eps n / 3} \, \frac{\alpha_{m_0+1}(M)}{\alpha_m(M)} \ < \ 2^{-(1-\eps) n}.
		\earyst
		Since $L \subset V_{m_0+1}(\bi_{+})$ was an arbitrary line, taking the maximum over $L$ yields \eqref{eq Vm+1iVm+1jclose}.
		
		Likewise, for every line $L \subset V_{m+1}(\bi_{+})$, \cref{lem distance-control} applied with $V = V_{m+1}(\bj_{+})$, $m' = m$ and $C = 2^{\eps n / 3}$ (hypotheses \eqref{equ dfiv2} and \eqref{equ dfjv}) gives, by \eqref{eq svgap2},
		\aryst
		d\big( V_{m+1}(\bj_{+}), L \big) \ \leq \ 2^{2 \eps n / 3} \, \frac{\alpha_{m+1}(M)}{\alpha_m(M)} \ < \ 2^{-2cn + \eps n} ,
		\earyst
		and taking the maximum over $L$ gives \eqref{equ:angle-Vmm}.
		
		Next we show that \eqref{eq svgap}--\eqref{equ dfiv2} hold outside small sets. This is the content of \cref{lem Oseledets statistics}. Applied with $\eta / 20$ in place of $\eta$, it gives a Borel set $A \subset \Sigma_{+}$ with $\nu_{+}(A) > 1 - \eta / 20$ such that for $n \gg 1$, every $\bi_{+} \in A$ satisfies \eqref{eq svgap}, \eqref{eq svgap2}, \eqref{equ dfiv}, and \eqref{equ dfiv2} with $M = Df^{t}_{\bi_{+}}$, and every $\bj_{+} \in A$ satisfies \eqref{equ dfjv} with $M = Df^{t}_{\bj_{+}}$ and $V = V_{m+1}(\bj_{+})$ (this is \eqref{eq stat jac} of the lemma, read at $\bj_{+}$).

		Hence, if $\bi_{+} \in A$ and $\bj_{+} \in A \cap \cC_n(\bi_{+})$, then the five conditions hold for the common map $M = Df^{t}_{\bi_{+}} = Df^{t}_{\bj_{+}}$, and by the above $(x, \bi_{+}, \bj_{+}) \in \widetilde\Gamma_1(n)$ for every $x \in \Lambda$. Since $\zeta_n$ is carried by $\{ (\bi_{+}, \bj_{+}) : \bj_{+} \in \cC_n(\bi_{+}) \}$, \eqref{eq coupling estimate} with $G = \Lambda \times A$ and the same set $A$ gives, for $n \gg 1$,
		\aryst
		&&(\mu \otimes \zeta_n)(\widetilde\Gamma_1(n)) \ \geq \ (\mu \otimes \zeta_n) \big\{ (x, \bi_{+}, \bj_{+}) : \bi_{+} \in A, \ \bj_{+} \in A \big\} \\
		&\geq& \ 1 - 2 \, \nu_{+}(A^{c})> 1 - \eta / 10.
		\earyst

		\subsubsection*{Step 2:}
		Before verifying item \eqref{itm 2P52} in \Cref{prop fiberwisedim}, we need some preparation.

		Our proof relies on two notions of goodness, introduced in \S\ref{subsec toolbox}. The first is a property of an atom of $\cD_{m+1, \bi_{+}, N}$ alone, and the second involves a base point $x$ and a second sequence $\bj_{+}$ as well.
		
		We assemble the single scale information we obtained from \cref{prop fiberwisedimatsinglescale}.

		Consider the function $Q : \Lambda \times \Sigma_{+} \times \Sigma_{+} \to \R$ given by
		\aryst
		&&Q(x, \bi_{+}, \bj_{+})\\
		&=& \frac{1}{n} \sum_{k = 0}^{n - 1} 
		\sum_{\substack{D \in \cD_{m+1, \bi_{+}, C_0n+ k} \\  D \subset \cD_{m+1, \bi_{+}, C_0 n}(x) } } \mu_{x, \bi_{+}}^{\cD_{m+1, \bi_{+}, C_0 n}(x)}(D) \cdot 1 \{ \mbox{$D$ is $\eta^8$-good at level $l$} \}  \cdot  \\
		&&   1\{ \mbox{$\bj_{+}$ is $\eta^4$-good for $( D, l)$}   \}.
		\earyst 
		By \cref{prop fiberwisedimatsinglescale},   we have 
		\aryst
		&& \int Q(x, \bi_{+}, \bj_{+}) d(\mu \otimes \zeta_n)(x, \bi_{+}, \bj_{+}) \\
		&\geq& \frac{1 - \eta^9}{n} \sum_{k = 0}^{n - 1}  
		\int_{\Gamma_n} \Big (
		\sum_{\substack{D \in \cD_{m+1, \bi_{+}, C_0n+ k} \\  D \subset \cD_{ m+1,\bi_{+}, C_0 n}(x) } }   \mu_{x, \bi_{+}}^{\cD_{m+1, \bi_{+}, C_0 n}(x)}(D)  1 \{ \mbox{$D$ is $\eta^8$-good at level $l$} \} \Big ) d(\mu \otimes \nu_{+})(x, \bi_{+}).
		\earyst
		Notice that since the goodness only depends on the subcube in $\cD_{C_n+k}$ and by the property of conditional measures we have
		\aryst
		&&  \frac{1}{n} \sum_{k = 0}^{n-1} \int \Big (
		\sum_{\substack{D \in \cD_{ m+1,\bi_{+}, C_0n+ k} \\  D \subset \cD_{ m+1,\bi_{+}, C_0 n}(x) } }  \mu_{x, \bi_{+}}^{\cD_{m+1,\bi_{+}, C_0 n}(x)}(D)  1 \{ \mbox{$D$ is $\eta^8$-good at level $l$} \} \Big ) d(\mu \otimes \nu_{+})(x, \bi_{+})  \\
		&=&  \frac{1}{n} \sum_{k = 0}^{n-1} \int  
		1 \{ \mbox{$\cD_{ m+1,\bi_{+}, C_0n+ k}(x)$ is $\eta^8$-good at level $l$} \}   d(\mu \otimes \nu_{+})(x, \bi_{+}).
		\earyst
		
		We also need the following Lemma.

		\begin{lemma} \label{lem Qtospreading}
			Let $l \geq 1$ be an integer and $0 < \eta \leq 1/16$.
			Let $\tilde{x} = (x, \bi_{+}, \bj_{+})$ with $(x, \bi_{+}) \in \Xi$ be such that $\lambda = \lambda(x, \bi_{+}, \bj_{+})$ in \eqref{eqn:lambda norm} is defined and positive, and set
			$T(\tilde{x}, n) := 2^{C_0 n} / \lambda(x, \bi_{+}, \bj_{+})$.
			Let $c_d := \lfloor \log_2 (1 + \sqrt{d}) \rfloor + 1$. %
			Assume that $n \geq (c_d + 1)/\eta$ and 
			$Q(x, \bi_{+}, \bj_{+}) > 1 - \eta^{7}$. Then %
			the measure   $$(\Dil_{T(\tilde{x}, n)})_*(P^{V_{m+1}(\bj_{+}) + U^{m-1}(x)})_*  \mu_{x, \bi_{+}} ^{\cD_{m+1, \bi_{+}, C_0 n}(x)}$$ is $(n, l + c_d + 1, 3 \eta)$-ball spreading.
			
			If moreover \eqref{eq Vm+1iVm+1jclose} holds and $\lambda(x, \bi_{+}, \bj_{+}) \geq 2^{-(1+\eps) n}$, then
			\aryst
			\big| \log_2 T(\tilde{x}, n) - (C_0 + 1) n \big| \ \leq \ \eps n .
			\earyst
			
		\end{lemma}

		We can now complete the proof of \cref{prop fiberwisedim}.
		In what follows, \cref{prop fiberwisedimatsinglescale} is applied with $\eps \log 2 / { C_1}$ in place of $\eps$, where ${ C_1} = { C_1}(\cF)$ is the constant in its item (2). Let $\Gamma_n$ and $B_n(x,\bi_+)$ be the sets given in this proposition. Its item (2) then reads $\lambda(x, \bi_{+}, \bj_{+}) \geq 2^{-(1+\eps) n}$, and its item (1) is unchanged.
		We have $(\mu \otimes \nu_{+})(\Gamma_n) > 1 - \eta^{9}$ for every $n \gg_{\cF,   \eta} 1$. Together with \cref{lem goodonaverage}, we have 
		\aryst
		\int Q(x, \bi_{+}, \bj_{+}) d(\mu \otimes \zeta_n)(x, \bi_{+}, \bj_{+}) \geq (1 - 2\eta^9)( 1 - \eta^{9})   > 1 - 3 \eta^{9}.
		\earyst
		
		Let
		\aryst
		\widetilde\Gamma_4(n):=\{(x,\bi_+,\bj_+): (x,\bi_+)\in \Gamma_n, \ \bj_+\in B_n(x,\bi_+),\ Q(x,\bi_+,\bj_+)>1-\eta^7\}. 
		\earyst
		The Markov's inequality gives $$(\mu \otimes \zeta_n) \{ Q \leq 1 - \eta^7 \} \leq \eta^{-7} \int (1 - Q) \, d(\mu \otimes \zeta_n) < 3 \eta^2.$$ And the set of triples with $(x, \bi_{+}) \in \Gamma_n$ and $\bj_{+} \in B_n(x, \bi_{+})$ has $(\mu \otimes \zeta_n)$-measure at least $(1 - \eta^9)^2$. Hence $$(\mu \otimes \zeta_n)(\widetilde\Gamma_4(n)) > 1 - 4 \eta^2.$$
		We replace the set $\widetilde\Gamma(n)$ of Step 1 by $\widetilde\Gamma(n) \cap \widetilde\Gamma_4(n)$. By \eqref{eq measurelower}, the new set has $(\mu \otimes \zeta_n)$-measure $> 1 - \eta/2 - 4 \eta^2 > 1 - \eta$, as $\eta \ll 1$.

		Let $\tilde{x} = (x, \bi_{+}, \bj_{+}) \in \widetilde\Gamma(n)$; we verify the hypotheses of \cref{lem Qtospreading} at $\tilde{x}$. By item (2) of \cref{prop fiberwisedimatsinglescale}, $\lambda = \lambda(x, \bi_{+}, \bj_{+}) \geq 2^{-(1+\eps) n} > 0$, so that $T(\tilde{x}, n) = 2^{C_0 n}/\lambda$ is defined, and \eqref{eq Vm+1iVm+1jclose} holds since $\tilde{x} \in \widetilde\Gamma_1(n)$.
		
		Steps 1 and 2 remain valid with $\eta$ replaced by $\eta/3$, and then $Q(\tilde{x}) > 1 - (\eta/3)^7$, because $\tilde{x} \in \widetilde\Gamma_4(n)$, while $(\mu \otimes \zeta_n)(\widetilde\Gamma(n)) > 1 - \eta$. Taking also $l > c_d + 1$ and $n \geq 3(c_d + 1)/\eta$, which costs nothing, \cref{lem Qtospreading} applies with $\eta/3$ and $l - c_d - 1$ in place of $\eta$ and $l$, and gives exactly the $(n, l, \eta)$-ball spreading asserted in item \eqref{itm 2P52}. As $\tilde{x} \in \widetilde\Gamma(n)$ was arbitrary, this completes the proof.

		In the next two sections, we will give the proofs of \cref{lem goodonaverage} and \cref{lem Qtospreading}. This will complete the proof of \cref{prop fiberwisedim}.

		\subsection{Most atoms are good: proof of \cref{lem goodonaverage}} \label{subsec proofsoflemma511}

		We use the tubes $\Tube_r(y, \bi_{+}) = \widehat\Tube_r(y, \bi_{+}) \cap \Lambda$ of \eqref{eq slab def} around the direction $V_{m_0+2}(\bi_{+})$, and the reformulation \eqref{eq good atom tube} of goodness in terms of them.

		Note that \eqref{eq exactdim} states that for $(\mu \otimes \nu_{+})$-a.e. $(x, \bi_{+})$, we have 
		\ary 
		\lim_{r \to 0} \frac{ \log \mu_{x, \bi_{+}}( \Tube_r(x, \bi_{+})  ) }{ \log r } = D_{m_0 + 2} - D_{m_0 + 1} > 0.
		\eary
		We let $n \gg_{\cF,  \eta} 1$ be a large integer.

		\smallskip

		\noindent{\bf $\bullet$ Stopping time.}
		For $x \in \Lambda$ and $t \geq 0$, write $x[t]\bi_+ = (i_{-t}(x),\dots,i_{-1}(x),i_0, i_1, \dots)\in \Sigma_{+}$ so that $f_x^{-t}:= (f_{x[t]})^{-1}$ is an affine map (see \eqref{eqn:f iterate}). We denote by $L_{x,t}$ its linear part. By the equivariance of the Oseledets flags (\Cref{thm Oseledets}),
		\ary
		\label{eqn:L equivariant}
		L_{x,t} \, V_{k}(\bi_{+}) = V_{k}(x[t]\bi_{+}) \qquad \mbox{for every } k .
		\eary
		We define a sequence of functions $( \lambda_t :  \Lambda  \times \Sigma_{+} \to \R )_{t \geq 0}$ by
		\ary \label{eq scaling}
		2^{\lambda_{t}(x, \bi_{+})} := \frac{ \mathrm{J}_{V_{m_0+1}(\bi_{+})}(L_{x,t}) }{ \mathrm{J}_{V_{m_0+2}(\bi_{+})}(L_{x,t}) } ,
		\eary
		the ratio of the Jacobians of $L_{x,t}$ on $V_{m_0+1}(\bi_{+})$ and on $V_{m_0+2}(\bi_{+})$. It is the expansion rate of the width of tubes:
		\ary \label{eq tube identity}
		f_{x}^{-t} \big( \widehat\Tube_{r}(x, \bi_{+}) \big) = \widehat\Tube_{r 2^{\lambda_{t}(x, \bi_{+})}} \big( f_{x}^{-t}(x), x[t]\bi_{+} \big), \qquad r > 0 .
		\eary
		Since we have the equivariant relation \eqref{eqn:L equivariant}, applying
		\Cref{lem base times height} to $L = L_{x,t}$, $V = V_{m_0+1}(\bi_{+})$ and $W = V_{m_0+2}(\bi_{+})$, we obtain $$d(L_{x,t} v, V_{m_0+2}(x[t]\bi_{+})) = 2^{\lambda_t(x, \bi_{+})} \, d(v, V_{m_0+2}(\bi_{+}))$$ for $v \in V_{m_0+1}(\bi_{+})$. Since $f_x^{-t}(x + v) - f_x^{-t}(x) = L_{x,t} v$, this gives \eqref{eq tube identity}.

		The increments of $\lambda_t$ are bounded above and below. Let $M_i$ denote the linear part of $f_i$, and put
		\aryst
		C_f := \max_{i \in I} \log_2 \| M_i^{-1} \| = \max_{i \in I} \log_2 \frac{1}{\alpha_d(M_i)} , \qquad c_- := \min_{i \in I} \log_2 \frac{1}{\| M_i \|} > 0,
		\earyst
		so that $c_- > 0$ because the maps $f_i$ are contractions. By the chain rule for Jacobians on subspaces and \Cref{lem base times height},
		\aryst
		2^{\lambda_{t+1}(x, \bi_{+}) - \lambda_t(x, \bi_{+})} = \frac{ \mathrm{J}_{V'_t}({M'}) }{ \mathrm{J}_{W'_t}({M'}) } = d({M'} u', {M'} W'_t) ,
		\earyst
		where ${M'} := M_{i_{{-t-1}}(x)}^{-1}$, $V'_t := L_{x,t} V_{m_0+1}(\bi_{+})$, $W'_t := L_{x,t} V_{m_0+2}(\bi_{+})$, and
		$u'$ is a unit vector of $V'_t \cap (W'_t)^{\perp}$. We obtain
		\ary \label{eq lambda increments}
		\lambda_0(x, \bi_{+}) = 0, \qquad c_- \ \leq \ \lambda_{t+1}(x, \bi_{+}) - \lambda_t(x, \bi_{+}) \ \leq \ C_f \qquad (t \geq 0) .
		\eary

		Let $l = l(\cF, \eta) > 4 C_f + 4$ be a large integer. We let $\Gamma_0 = \Gamma_0(l)$ be the set of $(x, \bi_{+}) \in \Lambda \times \Sigma_{+}$ such that
		\enmt
		\item[(G1)] $\mu_{x, \bi_{+}} \big( \Tube_{2^{-l/4}}(x, \bi_{+}) \big) < \eta^{10}$;
		\item[(G2)] for every $t \geq l$ and every unit vector $e \in E_{m_0+1}(x, \bi_{+})$,
		\aryst
		\big\| Df^{t}_{\bi_{+}} |_{V_{m_0+2}(\bi_{+})} \big\| \ \leq \ | Df^{t}_{\bi_{+}} e | ;
		\earyst
		\item[(G3)] $\sin \measuredangle \big( E_{m_0+1}(x, \bi_{+}), V_{m_0+2}(\bi_{+}) \big) > 3 \cdot 2^{-l/4}$.
		\eenmt
		Condition (G1) will be used in the estimate for Type 1 below, and (G2) and (G3) will be used in the estimate for Type 2.
		
		\Cref{thm Oseledets} and \eqref{eq exactdim} imply that, provided that $l \gg_{\cF, \eta} 1$, each of (G1)--(G3) fails on a set of measure $< \eta^{30}/3$.  This gives
		\ary \label{eq Gamma0big}
		(\mu \otimes \nu_{+})(\Gamma_0) > 1 - \eta^{30}.
		\eary

		We define a sequence of stopping times $( \tau_k :    \Lambda \times \Sigma_{+} \to \N )_{k \geq 0}$ with respect to the natural filtration by 
		\aryst
		\tau_k( x, \bi_{+}  ) = \inf \{ t \in \Z_{+} :    \lambda_t( x, \bi_{+}  ) > C_0 n +  k + l/2 \},
		\earyst
		This is the stopping time such that the image under $f_x^{-t}$ of the band $\widehat\Tube_{2^{-l-C_0n-k}}(x,\bi_{+})$ has width at least $2^{-l/2}$, by \eqref{eq tube identity}.
		
		By \eqref{eq lambda increments}, $\tau_k$ is finite and
		\begin{equation}
			\label{eq tau bounds}
			\frac{ C_0 n + k + l/2 }{ C_f } \ \leq \ \tau_k(x, \bi_{+}) , \qquad C_0 n + k + l/2 \ < \ \lambda_{\tau_k}(x, \bi_{+}) \ \leq \ C_0 n + k + l/2 + C_f .
		\end{equation}
		From now on, we assume $n \geq C_f \, l / C_0$. This gives 
		$\tau_k(x, \bi_{+}) \geq l$ for every $k \geq 0$ and every $(x, \bi_{+})$, and hence condition (G2) is available at time $t = \tau_k(x, \bi_{+})$.

		\noindent{\bf $\bullet$ Completing the proof.}

		Fix $0 \leq k < n$, and write $N := C_0 n + k$ throughout; the notation $\cD_{\bi_+,N}(x) := \cD_{m+1,\bi_{+}, C_0 n + k}(x)$ is also used below.
		For each $( x, \bi_{+} ) \in  \Lambda \times \Sigma_{+}$, we set 
		\aryst
		T_{n, k}(   x, \bi_{+} ) :=  f_{x}^{- \tau_k( x, \bi_{+})} ( \cD_{ \bi_{+}, N}(x)).
		\earyst
		We have    
		\aryst
		T_{n,k}(x, \bi_{+}) \cap \Lambda \ = \ f_{x}^{- \tau_k( x, \bi_{+})} \Big( \cD_{ \bi_{+}, N}(x) \cap f_{x[\tau_k(x, \bi_{+})]}(\Lambda) \Big) .
		\earyst
		
		We say that $T_{n, k}(  x, \bi_{+} )$ is {\it full} if
		\aryst
		\mathcal{A}_{m_0+1} \big( \sigma^{- \tau_k( x, \bi_{+})}(x,\bi_{+}) \big) \cap \Lambda \ \subset \ T_{n, k}(  x, \bi_{+} ).
		\earyst
		Here $\mathcal{A}_{m_0+1}$ is  given in Subsection \ref{subsec:partitions}.

		We claim that if $(x, \bi_{+}) \in \Gamma_0$, $\sigma^{- \tau_k(x, \bi_{+})}(x, \bi_{+}) \in \Gamma_0$, and $T_{n, k}( x, \bi_{+} )$ is not full, then
		\begin{equation} \label{eq closetoboundary}
			d \big( x, \partial \cD_{ \bi_{+}, N}(x) \big) < 2^{-l/4 - N} .
		\end{equation}
		This is \cref{lem fullness}, proved in \cref{app linear algebra}.

		We say that  $(x, \bi_{+} ) \in  \Lambda \times \Sigma_{+}$ is of
		\enmt
		\item[Type 1 for $k$:]  if $(x, \bi_{+}) \in \Gamma_0$, $\sigma^{ - \tau_k(x, \bi_{+})}(x, \bi_{+}) \in \Gamma_0$, and $T_{n, k}(  x, \bi_{+} )$ is full;
		\item[Type 2 for $k$:] if $(x, \bi_{+}) \in \Gamma_0$, $\sigma^{ - \tau_k(x, \bi_{+})}(x, \bi_{+}) \in \Gamma_0$, but   $T_{n, k}(   x, \bi_{+} )$ is not full;
		\item[Type 3 for $k$:] if $(x, \bi_{+}) \notin \Gamma_0$ or $\sigma^{ - \tau_k(x, \bi_{+})}(x, \bi_{+}) \notin \Gamma_0$.
		\eenmt
		
		For $q \in \{1, 2, 3\}$,  we denote by $\Gamma_{q, k}$ the set of points in $\Lambda \times \Sigma_{+}$ of Type $q$ for $k$.

		{\em $\bullet$ Estimates for  Type 1:}

		Let $A = A(x, \bi_{+})$ (see \eqref{eq A def}), where $(x, \bi_{+}) \in \Xi$. By \eqref{eq scaling}, $\lambda_t(y, \bi_{+})$ depends on $y$ only through $y[t]$. Hence, the event $\{ \tau_k(y, \bi_{+}) = t \}$ is determined by $y[t]$. If $t = \tau_k(y, \bi_{+})$ and $z \in \Lambda$ satisfy $z[t] = y[t]$, then $\tau_k(z, \bi_{+}) = t$. And in the family of words $\cW_k := \{ y[\tau_k(y, \bi_{+})] : y \in \Lambda \}$, no word is a proper suffix of another.
		
		By the strong separation condition \eqref{equ separation}, $f_{\bw}(\Lambda) \cap f_{\bw'}(\Lambda) = \emptyset$ unless one of the words $\bw, \bw'$ is a suffix of the other. Together with the above observation, distinct words of $\cW_k$ give disjoint cylinders.
		
		Let $y \in A \cap \Lambda$, $t := \tau_k(y, \bi_{+})$ and $\bw := y[t]$, so that $\sigma^{-t}(y, \bi_{+}) = (f_{\bw}^{-1}(y), \bw \bi_{+})$. By the equivariance of the Oseledets flags,
		\aryst
		f_{\bw} \big( \cal A_{m_0+1}(\sigma^{-t}(y, \bi_{+})) \cap \Lambda \big) \ = \ A \cap f_{\bw}(\Lambda) ,
		\earyst
		and $T_{n,k}(y, \bi_{+})$ is full if and only if  $A \cap f_{\bw}(\Lambda) \subset \cD_{\bi_+,N}(x)$.
		
		Now let
		\aryst
		\cW_k(x) &:=& \big\{ y[\tau_k(y, \bi_{+})] \ : \ y \in \cD_{\bi_+,N}(x) \cap A \cap \Lambda, \ (y, \bi_{+}) \in \Gamma_{1,k} \big\} \subset \cW_k , \\
		E_{\bw} &:=& A \cap f_{\bw}(\Lambda) \quad (\bw \in \cW_k(x)) .
		\earyst
		We have
		\ary \label{eq type1 decomposition}
		\big\{ y \in \cD_{\bi_+,N}(x) \cap A \ : \ (y, \bi_{+}) \in \Gamma_{1,k} \big\} \ \subset \ \bigsqcup_{\bw \in \cW_k(x)} E_{\bw} \ \subset \ \cD_{\bi_+,N}(x) .
		\eary
		Moreover, for $\bw \in \cW_k(x)$ and $z \in E_{\bw}$, we have $\tau_k(z, \bi_{+}) = |\bw|$, %
		and writing
		\aryst
		(z, \bi_{+})' \ := \ \sigma^{-|\bw|}(z, \bi_{+}) \ = \ (f_{\bw}^{-1}(z), \bw \bi_{+}),
		\earyst
		we obtain from the equivariance of the fiber measures under the skew product that
		\ary \label{eq fibre equivariance}
		\mu_{x, \bi_{+}}( \, \cdot \, \cap E_{\bw}) \ = \ \mu_{x, \bi_{+}}(E_{\bw}) \ (f_{\bw})_* \, \mu_{(z, \bi_{+})'} \qquad \mbox{for every } z \in E_{\bw} .
		\eary
		
		Fix $y \in A$. For every $\bw \in \cW_k(x)$ such that $E_{\bw} \cap \Tube_{2^{-N-l}}(y, \bi_{+})$ contains a point of Type 1 for $k$, fix one such point $y_{\bw}$. Then $(y_{\bw}, \bi_{+})' \in \Gamma_0$ by the definition of Type 1. Applying \eqref{eq fibre equivariance} with $z = y_{\bw}$ to the set $\Tube_{2^{-N-l}}(y,\bi_+)$ gives
		\[ \mu_{x,\bi_+} \big( E_{\bw} \cap \Tube_{2^{-N-l}}(y,\bi_+) \big) = \mu_{x,\bi_+}(E_{\bw}) \ \mu_{(y_{\bw},\bi_+)'} \big( f_{\bw}^{-1} \Tube_{2^{-N-l}}(y,\bi_+) \big) . \]
		It follows from \eqref{eq tau bounds} and the assumption $l > 4 C_f + 4$ that %
		$$2^{-l-N+1+\lambda_{|\bw|}(y_{\bw}, \bi_{+})} \le 2^{-l/2 + C_f + 1} < 2^{-l/4}.$$  Moreover, as $y_{\bw} \in \Tube_{2^{-N-l}}(y, \bi_+)$, %
		applying \eqref{eq tube identity} at $y_{\bw}$ with $t = |\bw|$, we obtain
		\[ f_{\bw}^{-1} \Tube_{2^{-N-l}}(y, \bi_+) \ \subset \ \widehat\Tube_{2^{-N-l+1+\lambda_{|\bw|}(y_{\bw}, \bi_{+})}} \big( (y_{\bw}, \bi_+)' \big) \ \subset \ \widehat\Tube_{2^{-l/4}} \big( (y_{\bw}, \bi_+)' \big) . \]
		Hence, since $(y_{\bw}, \bi_{+})' \in \Gamma_0$ satisfies (G1), %
		\[ \mu_{(y_{\bw},\bi_+)'} \big( f_{\bw}^{-1} \Tube_{2^{-N-l}}(y,\bi_+) \big) \ \leq \ \mu_{(y_{\bw},\bi_+)'} \big( \Tube_{2^{-l/4}} ( (y_{\bw},\bi_+)' ) \big) \ \leq \ \eta^{10} . \]
		Summing over $\bw \in \cW_k(x)$ and using \eqref{eq type1 decomposition} yields, for every $y \in A$,
		\begin{align}
			\label{eq type1 estimate}
			&\mu_{x,\bi_+}\big( \Tube_{2^{-l-N}}(y,\bi_+) \cap \Gamma_{1,k} \cap \cD_{\bi_+,N}(x) \big)\\
			\leq &\ \sum_{\bw \in \cW_k(x)} \mu_{x,\bi_+} \big( E_{\bw} \cap \Tube_{2^{-N-l}}(y,\bi_+) \cap \Gamma_{1,k} \big) \\
			\leq&\eta^{10} \sum_{\bw \in \cW_k(x)} \mu_{x,\bi_+}(E_{\bw}) \ \leq \ \eta^{10} \mu_{x,\bi_+}(\cD_{\bi_+,N}(x)). \nonumber
		\end{align}

		{\em $\bullet$  Estimates for  Type 2:}
		We have seen that in this case, \eqref{eq closetoboundary} holds.
		Thus 
		we obtain
		\ary \label{eq integraloftype2}
		&& 	\frac{1}{n} \sum_{k = 0}^{n-1}\int  \frac{  \mu_{x, \bi_{+}}(\cD_{ \bi_{+}, C_0n+ k }(x) \cap \Gamma_{2, k} )   }{ \mu_{x, \bi_{+}}(\cD_{ \bi_{+}, C_0n+ k }(x) ) }     d(\mu \otimes \nu_{+})(x, \bi_{+})   \\
		&\leq&  \frac{1}{n} \sum_{k = 0}^{n-1}\int \frac{\mu_{x,\bi_{+}}(  \partial_{2^{-l/4}} \cD_{ \bi_{+}, C_0n+ k}(x) )}{\mu_{x,\bi_{+}}(\cD_{ \bi_{+}, C_0n+ k}(x) )}   d(\mu \otimes \nu_{+})(x, \bi_{+}) \nonumber\\
		& =&  \frac{1}{n} \sum_{k = 0}^{n-1}\int 1_{\partial_{2^{-l/4}} \cD_{ \bi_{+}, C_0n+ k}(x) }(x,\bi_{+}) d(\mu \otimes \nu_{+})(x, \bi_{+}) \ \leq \ 9 d \cdot 2^{-l/8},\nonumber
		\eary
		where the last inequality is \eqref{eq randomdyadicpartitions} with $c_0 = 2^{-l/4} \leq 1/2$.
		
		{\em $\bullet$  Estimates for Type 3:} 
		By definition, we have  
		\ary \label{eq type31}
		&& \frac{1}{n} \sum_{k = 0}^{n-1}
		\int  \frac{  \mu_{x, \bi_{+}}(\cD_{ \bi_{+}, C_0n+ k }(x) \cap \Gamma_{3, k} )   }{ \mu_{x, \bi_{+}}(\cD_{ \bi_{+}, C_0n+ k }(x) ) }     d(\mu \otimes \nu_{+})(x, \bi_{+}) \\ 
		&=& \frac{1}{n} \sum_{k = 0}^{n-1}
		\int   1_{  \Gamma_{3, k}    }(x, \bi_{+})    d(\mu \otimes \nu_{+})(x, \bi_{+}) \nonumber  \\
		&\leq& \frac{1}{n} \sum_{k = 0}^{n-1}
		\int   1_{  \Gamma_{0}^c    }( \sigma^{ - \tau_k(x, \bi_{+})}(x, \bi_{+})   )    d(\mu \otimes \nu_{+})(x, \bi_{+}) \ + \ (\mu \otimes \nu_{+})(\Gamma_0^c) . \nonumber
		\eary
		The last term comes from the condition $(x, \bi_{+}) \notin \Gamma_0$ in the definition of Type 3 and is already of the required size by \eqref{eq Gamma0big}. It remains to bound the first.
		Our key observation is the following:  the collection of sequences $( \tau_k( x, \bi_{+}) )_{k \geq 0}$ is {\it uniformly quasi-isometric}:
		there exist $a_1 > 1, b_1 > 0$ such that for any $(x,   \bi_{+})  \in \Lambda \times \Sigma_{+}$  and integers $k_1 > k_2 \geq 0$ we have
		\aryst
		a_1^{-1}(k_1 - k_2) - b_1 < \tau_{k_1}(x, \bi_{+})  - \tau_{k_2}( x, \bi_{+}) < a_1(k_1 - k_2) + b_1.
		\earyst
		Indeed, \eqref{eq tau bounds} gives
		$$\lambda_t(x,\bi_+)-(C_0 n + k_2 + l/2+C_f) \leq \lambda_t(x,\bi_+)-\lambda_{\tau_{k_2}}(x, \bi_{+}) \ <\lambda_t(x,\bi_+)-(\ C_0 n + k_2 + l/2).$$
		By the increment formula given in \eqref{eq lambda increments}, $\lambda_t$ exceeds $C_0 n + k_1 + l/2$ after at most $(k_1 - k_2)/c_- + 1$ further steps and not before $(k_1 - k_2)/C_f - 1$ further steps. This gives the display with $a_1 := \max\{ C_f, 1/c_-,2 \}$ and $b_1 := 2$, constants depending on $\cF$ only.
		Thus, there exists $C > 1$ such that the first term on the right of \eqref{eq type31} is bounded by the
		\begin{equation}
			\label{eq integraloftype32} 
			\frac{C}{n} \sum_{t = 0}^{Cn-1}
			\int   1_{  \Gamma_{0}^c    }( \sigma^{ - t}(x, \bi_{+})   )    d(\mu \otimes \nu_{+})(x, \bi_{+})
			\leq C^2 (\mu \otimes \nu_{+})(\Gamma_0^c) < C^2  \eta^{30},
		\end{equation}
		so that
		\ary
		\eqref{eq type31} < (C^2 + 1) \eta^{30} . 
		\eary

		\textbf{We are now ready to finish the proof.}
		
		By summarizing the above discussions, we see the following. For the estimate for $\Gamma_{2,k}$, we now choose $l = l(\cF,  \eta) > 4 C_f + 4$ so large that \eqref{eq Gamma0big} holds and that $9 d \cdot 2^{-l/8} \leq \eta^{18}$. The right hand side of \eqref{eq integraloftype2} is then at most $\eta^{18}$. Suppose that
		\ary \label{eq types23 small}
		\mu_{x, \bi_{+}}\big( \cD_{ \bi_{+}, C_0n+ k }(x) \cap (\Gamma_{2, k} \cup \Gamma_{3,k}) \big) \ \leq \ \frac{ \eta^8 }{2} \, \mu_{x, \bi_{+}}\big( \cD_{ \bi_{+}, C_0n+ k }(x) \big) .
		\eary
		Then, for every $y \in A(x, \bi_{+})$, since the three types partition $\cD_{\bi_{+}, C_0n+k}(x)$, the estimate \eqref{eq type1 estimate} for Type 1 together with \eqref{eq types23 small} gives
		\aryst
		\mu^{\cD_{\bi_{+}, C_0n+k}(x)}_{x, \bi_{+}} \big( \Tube_{2^{-l-N}}(y, \bi_{+}) \big) \ \leq \ \eta^{10} + \frac{\eta^8}{2} \ < \ \eta^8 ,
		\earyst
		which implies $\cD_{\bi_{+}, C_0n+ k}(x)$ is $\eta^8$-good at level $l$ by \eqref{eq good atom} in the form \eqref{eq good atom tube}. Consequently, $\cD_{\bi_{+}, C_0n+ k}(x)$ is {\it not} $\eta^8$-good at level $l$ only if
		
		\aryst
		\mu_{x, \bi_{+}}(\cD_{ \bi_{+}, C_0n+ k }(x) \cap (\Gamma_{2, k} \cup \Gamma_{3,k}) )  >  \frac{ \eta^8 }{2} \mu_{x, \bi_{+}}(\cD_{ \bi_{+}, C_0n+ k }(x) ).
		\earyst
		In summary, we have 
		\aryst
		&& 1 \{ \mbox{ $\cD_{\bi_{+}, C_0n+ k}(x) $ is {\it not} $\eta^8$-good at level $l$} \}     \\
		&\leq& 1\Big \{  \mu_{x, \bi_{+}}(\cD_{ \bi_{+}, C_0n+ k }(x) \cap \Gamma_{2, k}  )    >  \frac{ \eta^8 }{4} \mu_{x, \bi_{+}}(\cD_{ \bi_{+}, C_0n+ k }(x) )   \Big \} \\
		& + & 1\Big \{  \mu_{x, \bi_{+}}(\cD_{ \bi_{+}, C_0n+ k }(x) \cap \Gamma_{3, k}  )    >  \frac{ \eta^8 }{4} \mu_{x, \bi_{+}}(\cD_{ \bi_{+}, C_0n+ k }(x) )   \Big \}.
		\earyst
		By Markov's inequality, \eqref{eq integraloftype2}  and \eqref{eq integraloftype32},  the integral of the last line above (average over $0\leq k\leq n-1$) with respect to $\mu \otimes \nu_{+}$ is bounded from above by
		\aryst
		\frac{4}{ \eta^8 }( (C^2 + 1) (\mu \otimes \nu_{+} )(   \Gamma_0^c ) + 9 d \cdot 2^{-l/8} ).
		\earyst
		By \eqref{eq Gamma0big} and the choice of $l$ above, this is at most $\frac{4}{\eta^8} \big( (C^2 + 1) \eta^{30} + \eta^{18} \big) < \eta^{9}$ for $\eta \ll_{\cF, C_0} 1$. This is the choice of $l = l(\cF,  \eta)$ in the statement. The requirement $n \geq C_f \, l / C_0$ made after \eqref{eq tau bounds} is the only place where $n$ has to be large in terms of $l$.
		This finishes the proof.

		\subsection{Goodness implies spreading: proof of Lemma \ref{lem Qtospreading}} \label{subsec proofsoflemma512}

		{\em The bound on $T(\tilde{x}, n)$.} Assume that \eqref{eq Vm+1iVm+1jclose} holds and that $\lambda=\lambda(x,\bi_+,\bj_+) \geq 2^{-(1+\eps) n}$. Since $V_{m+1}(\bj_{+}) \subset J=V_{m+1}(\bj_{+}) + U^{m-1}(x)$, every unit vector $w \in V_{m_0+1}(\bi_{+})$ satisfies $$\| P^{J} w \| = d(w, J) \leq d(w, V_{m+1}(\bj_{+})) \leq d\big( V_{m_0+1}(\bi_{+}), V_{m+1}(\bj_{+}) \big).$$ So \eqref{eq Vm+1iVm+1jclose} gives $\lambda \leq 2^{-(1 - \eps) n}$. Together with the lower bound $\lambda \geq 2^{-(1+\eps) n}$ assumed above, we have
		\aryst
		\log_2 T(\tilde{x}, n) - (C_0 + 1) n = \log_2 (1/\lambda) - n\in [-\epsilon n, \epsilon n].
		\earyst

		Next, we establish the spreading property.
		By Markov's inequality, applied to the average over $0 \leq k \leq n-1$ that defines $Q$, there is a subset $\mathcal{S} \subset \{0, \cdots, n-1\}$ with $|\mathcal{S}| > (1 - \eta^3) n$, so that for every $k \in \mathcal{S}$, we have
		\aryst
		\sum_{\substack{D \in \cD_{m+1, \bi_{+}, C_0n+ k} \\  D \subset \cD_{m+1, \bi_{+}, C_0 n}(x) } }  \mu_{x, \bi_{+}}^{\cD_{m+1, \bi_{+}, C_0 n}(x)}(D) \cdot 1 \{ \mbox{$D$ is $  \eta^8 $-good at level $l$} \} \cdot   1\{ \mbox{$\bj_{+}$ is $ \eta^4$-good for $(D, l)$}   \} 
		> 1 - \eta^4.
		\earyst
		Let us denote
		\aryst
		\theta_D := (P^{V_{m+1}(\bj_{+}) + U^{m-1}(x)})_* ( \mu_{x, \bi_{+}}^{D} ), \qquad D \in \cD_{m+1, \bi_{+}, C_0 n + k}, \ D \subset \cD_{m+1, \bi_{+}, C_0 n}(x).
		\earyst
		Fix some $k \in \mathcal{S}$.
		Let $\cR \subset  \cD_{m + 1, \bi_{+}, C_0n+ k}$ be the set of $D$ contained in  $\cD_{m + 1 , \bi_{+}, C_0 n}(x)$ such that $D$ is $  \eta^8$-good at level $l$, and $\bj_{+}$ is $ \eta^4$-good for $( D, l)$. Then
		\aryst
		\sum_{D \in \cR} \mu_{x, \bi_{+}}^{\cD_{m + 1, \bi_{+}, C_0 n}(x)}(D) > 1 - \eta^4.
		\earyst
		Write $N := C_0 n + k$. Then for each $D \in \cR$, by \eqref{eq good sequence}, we have
		\aryst
		\sup_{p} \ \theta_D \big( B(p, \lambda\, 2^{-N - l}) \big) \ < \ \eta^4 ,
		\earyst
		Moreover, $\diam(\supp \theta_D) \leq \sqrt{d-m}\, \lambda\, 2^{-N}$.
		Set
		\aryst
		\zeta := (P^{V_{m+1}(\bj_{+}) + U^{m-1}(x)})_*  (  \mu_{x, \bi_{+}}^{\cD_{m + 1, \bi_{+}, C_0 n}(x)} )  =  \omega' + \omega''
		\earyst
		where
		\aryst
		\omega' =  \sum_{D \in \cR} \mu_{x, \bi_{+}}^{\cD_{m + 1, \bi_{+}, C_0 n}(x)}(D) \, \theta_D , \qquad
		\omega'' =   \sum_{D \notin \cR} \mu_{x, \bi_{+}}^{\cD_{m + 1, \bi_{+}, C_0 n}(x)}(D) \, \theta_D ,
		\earyst
		the sums running over the atoms $D \in \cD_{m + 1, \bi_{+}, C_0 n + k}$ contained in $\cD_{m + 1, \bi_{+}, C_0 n}(x)$.
		By the definition of $k\in\cS$, we have 
		\aryst
		| \omega'' | &\leq&   \sum_{D \notin \cR} \mu_{x, \bi_{+}}^{\cD_{m  + 1, \bi_{+}, C_0 n}(x)}(D) \leq   \eta^4.
		\earyst

		Let $c_d$ be as in \cref{lem Qtospreading}, so that $2^{c_d} > 1 + \sqrt{d-m}$. For $D \in \cR$, the diameter bound above shows that $\supp \theta_D \subset B(y, 2^{c_d} \lambda 2^{-N})$ whenever $\supp \theta_D$ meets $B(y, \lambda 2^{-N-l})$. Hence, for every $y \in (V_{m+1}(\bj_+)+U^{m-1}(x))^{\perp}$,
		\begin{align}\label{equ:omega'}
			\omega'\big( B(y, \lambda 2^{-N-l}) \big)
			=& \sum_{\substack{D \in \cR \\ \supp \theta_D \cap B(y, \lambda 2^{-N-l}) \neq \emptyset}} \mu^{\cD_{m+1, \bi_{+}, C_0 n}(x)}_{x, \bi_{+}}(D) \, \theta_D\big( B(y, \lambda 2^{-N-l}) \big) \\
			\mathrel{\leq}& \eta^4 \sum_{\substack{D \in \cR \\ \supp \theta_D \cap B(y, \lambda 2^{-N-l}) \neq \emptyset}} \mu^{\cD_{m+1, \bi_{+}, C_0 n}(x)}_{x, \bi_{+}}(D)
			\ \leq \ \eta^4 \, \omega'\big( B(y, 2^{c_d} \lambda 2^{-N}) \big) .\nonumber
		\end{align}

		Let 
		\aryst
		\Omega_k \ := \ \Big\{ \, y\in  \supp \zeta \ : \ \omega''\big( B(y, 2^{c_d} \lambda 2^{-N}) \big) \ \leq \ \eta^{1/2} \, \omega'\big( B(y, 2^{c_d} \lambda 2^{-N}) \big) \Big\} .
		\earyst
		The restriction to $\supp \zeta$ guarantees that $\zeta(B(y, 2^{c_d} \lambda 2^{-N})) > 0$ for $y \in \Omega_k$.
		
		We claim that 
		\ary\label{equ:omega''}
		\zeta(\Omega_k^c)\leq \eta^3
		\eary (See \cref{lem:ball-compare} for the proof.)
		Take an arbitrary $y \in \Omega_k$. Then 
		\aryst
		&& \zeta( B(y,\lambda 2^{-N-l})) = \omega'( B(y,\lambda 2^{-N-l}))  + \omega''( B(y,\lambda 2^{-N-l}))  \\
		&\leq&   \eta^4   \omega'( B(y,2^{c_d} \lambda 2^{-N}) )  + \omega''(B(y,2^{c_d} \lambda 2^{-N})) \leq     (  \eta^4  + \eta^{1/2} ) \omega'( B(y,2^{c_d} \lambda 2^{-N}) ) \\
		&<& \zeta(B(y,2^{c_d} \lambda 2^{-N})) / 2 ,
		\earyst
		where the last inequality uses $\omega' \leq \zeta$, $\eta^4 + \eta^{1/2} < 1/2$ (as $\eta \leq 1/16$) and $\zeta(B(y, 2^{c_d} \lambda 2^{-N})) > 0$ (as $y \in \supp \zeta$).
		
		Set $\tilde\zeta := (\Dil_{T(\tilde{x}, n)})_* \zeta$. Since $\Dil_{T} B(y, r) = B(Ty, Tr)$ and $T(\tilde{x}, n) \, \lambda 2^{-N} = 2^{-k}$, the last display reads, for every $k \in \mathcal{S}$ and every $y \in \Omega_k$,
		\ary \label{eq dilated ball spreading}
		\tilde\zeta \big( B(T(\tilde{x}, n) y, \ 2^{-k-l}) \big) \ < \ \tilde\zeta \big( B(T(\tilde{x}, n) y, \ 2^{c_d - k}) \big) / 2 ,
		\eary
		that is, spreading between the radii $2^{-(k - c_d)}$ and $2^{-(k+l)}$, which differ by the factor $2^{l + c_d}$.

		Put $\Omega_k := \emptyset$ for $k \in \{0, \ldots, n-1\} \setminus \mathcal{S}$, and let $g(y) := \frac{1}{n} \big| \{ 0 \leq k \leq n-1 : y \notin \Omega_k \} \big|$. Since $\zeta(\Omega_k^c) \leq \eta^3$ for $k \in \mathcal{S}$ by \eqref{equ:omega''}, and $\zeta(\Omega_k^c) = 1$ for the fewer than $\eta^3 n$ indices $k \notin \mathcal{S}$,
		\aryst
		\int g \, d\zeta \ = \ \frac{1}{n} \sum_{k = 0}^{n-1} \zeta(\Omega_k^c) \ < \ \eta^3 + \eta^3 \ = \ 2 \eta^3 .
		\earyst
		Set $\Omega := \{ y \in \R : g(y) < \eta \}$. By Markov's inequality,  we have $\zeta(\Omega) > 1 - \eta$ as $\eta<1/2$, and for every $y \in \Omega$,
		
		\aryst
		\big| \{ 0 \leq k \leq n-1 :  y \in \Omega_k \} \big| > (1 - \eta) n.
		\earyst
		Hence, the dilated set $\Dil_{T(\tilde{x}, n)} \Omega$ has $\tilde\zeta$-measure $\tilde\zeta(\Dil_{T(\tilde{x}, n)} \Omega) = \zeta(\Omega) > 1 - \eta$, and each of its points $\tilde{y} = T(\tilde{x}, n) y$ satisfies the ball inequality \eqref{eq dilated ball spreading} for more than $(1 - \eta) n$ scales $0 \leq k \leq n-1$.

		\section{Slab-Filling property} \label{sec SF}

		\subsection{Setting and main result}
		Recall that given  
		a partial flag $\bE = (E_i)_{i = 1}^{q}$ of $\R^d$,  a tuple $\br = (r_i)_{i = 1}^{q} \in (0, \infty]^{q}$, and $x \in \R^d$, the slab  $S(x, \bE, \br)$ is defined in Subsection \ref{subsec grassmannianflagsslabs}.  		
		
		Let $X = S(x, \bE, \br)$, we denote $\lambda X = S(x, \bE, \lambda \br)$ for every  $\lambda > 0$.
		
		We say that  a subset $Y \subset \R^d$ is {\it $C$-commensurable} with $X = S(x, \bE, \br )$ for a constant $C \geq 1$ if we have $C^{-1} X \subset Y \subset C X$. In this case, we denote
		\aryst
		Y \overset{C} \sim X.
		\earyst
		Given a partial flag  $\bE = (E_i)_{i = 1}^{q}$ of $\R^d$ , a tuple $\br = (r_i)_{i = 1}^{q} \in \R_+^{q}$ with $r_{m+1} = \cdots = r_q = \delta$ for some $1 \leq m \leq q-1$,   we have 
		\aryst
		S(x, \bE, \br) \overset{d} \sim S(x, (E_{1}, \cdots, E_{m}, \R^d), (r_1, \cdots, r_{m}, \delta) ).
		\earyst

		Given a $\mu$-typical $x \in \Lambda$, an integer $n > 0$, and $\bi_{+} \in \Sigma_{+}$, $m \in \{1, \cdots, d\}$, $\epsilon>0$, and $\rho, \delta \in (  e^{n (\chi_{m}+\epsilon)},e^{n  (\chi_{m-1}-\epsilon) } )$ with $\rho > \delta$,  we denote
		\aryst
		X(x, m, \rho, \delta ) &:=& S( x, ( U^{m - 1}(x), \R^{d} ), ( \rho, \delta ) ), \\
		Y(\bi_{+}, n, m, \rho, \delta ) &:=& S( 0, (  V_{d+1-j}(\bi_+) )_{j = 1}^{d},  (  e^{- n \chi_d} \delta, \cdots, e^{- n \chi_m} \delta , e^{- n \chi_{m-1}} \rho, \cdots, e^{- n \chi_{1}} \rho  )    ).
		\earyst
		
		Recall that $\alpha\in (0,1)$ is the H\"{o}lder exponent given in \Cref{prop Holdermap} and the definition of $(\beta,\rho,\eta,\eps)$-regular measure from \cref{def regularmeasure}. The main result of this section is the following.
		
		\begin{prop}[Slab-Filling property] \label{prop slabfilling}
			Assume that we have $D_{m} = m-1$.
			Let $n \geq 1$ be an integer, and let $\rho, \delta, \eps', \eps > 0$, $\eta,  \eta_0   \in (0, 1)$. 
			Assume that $n \gg_{\cF,   \eps, \eps', \eta,  \eta_0} 1$ and $e^{n( \chi_{m-1} - \eps )} > \rho > \delta > e^{n ( \chi_{m} + \eps )}$ and $\delta^{1 - \alpha/10} >  \rho >   \delta^{1 - \eps'}$. 
			Then	there is a subset $\Lambda''   \subset \Lambda$ such that $\mu(  \Lambda''  ) > 1 -  \eta_0$ and for every $x \in \Lambda''$, for 
			\aryst
			X = X(x, m, \rho, \delta )  =    S( x, ( U^{m - 1}(x), \R^{d} ),  ( \rho, \delta  ) ),
			\earyst 
			the measure $(\Dil_{1/\rho})_{*} ( \mu^{X} )$ is $( m-1 , \delta/\rho, \eta, \eps   )$-regular.
		\end{prop}

		The remaining lemmas are the preparation for the proof of \cref{prop slabfilling}. \Cref{lem slabcovering} gives a covering by slabs with bounded multiplicity, \cref{lem 2XX} gives the doubling property $\mu( 3 X(x, m, \rho, \delta) ) \leq K_0 \mu( X(x, m, \rho, \delta) )$; and \cref{lem slabmapping} with its \cref{cor slabmapping} compares the image of a slab under $f_{\bi_{+}}^n$ with the slab at the image point.
		
		\subsection{Covering and mapping between slabs}
		
		We begin with a rather general lemma.

		\begin{lemma}[Covering by slabs with bounded multiplicity] \label{lem slabcovering}
			Let $m\in \{1,\cdots,d\}$ and let $\Lambda_1 \subset \Lambda$ be a subset on which $x \mapsto U^{m-1}(x)$ is $\alpha$-H\"older continuous.
			Let $1\gg_{\mathcal{F},\Lambda_1}\rho \geq \delta > 0$ satisfy $\delta^{1 - \alpha/10} > \rho$. Then given any subset $B \subset \Lambda_1$, the family $\{ X(x, m, \rho, \delta) \}_{x \in B}$ admits a finite subcovering $\{ X(x_l, m, \rho, \delta) \}_{l = 1}^{L}$ of $B$ such that the tripled family $\{ 3 X(x_l, m, \rho, \delta) \}_{l = 1}^{L}$ has multiplicity at most $C_d$, where $C_d > 0$ depends only on $d$.
		\end{lemma}
		
		\begin{proof}
			Write $K(x) = X(x, m, \rho, \delta) - x = B^{U^{m-1}(x)}(\rho) + B^{U^{m-1}(x)^{\perp}}(\delta)$.
			
			\smallskip
			First, we claim that if $x, x' \in \Lambda_1$ satisfy $3 X(x, m, \rho, \delta) \cap 3 X(x', m, \rho, \delta) \neq \emptyset$, then
			\ary
			\label{eqn:K inclusion}
			K(x') \subset 2 K(x) \quad \mbox{and} \quad K(x) \subset 2 K(x').
			\eary
			Indeed, since $\diam(3 X(x, m, \rho, \delta)) \leq 3 \sqrt{2} \rho$, we have $|x - x'| \leq 6 \sqrt{2} \rho$, and hence, by the H\"older continuity on $\Lambda_1$,  the Grassmannian distance $\theta := \bd( U^{m-1}(x), U^{m-1}(x') )$ satisfies $\theta \leq C_{\Lambda_1} (6 \sqrt{2} \rho)^{\alpha}$, where $C_{\Lambda_1}>0$ is the H\"{o}lder constant of $x\mapsto U^{m-1}(x)$ on $\Lambda_1$. For $v = a + b \in K(x')$ with $a \in B^{U^{m-1}(x')}(\rho)$ and $b \in B^{U^{m-1}(x')^{\perp}}(\delta)$, writing $U = U^{m-1}(x)$, we have
			\aryst
			&&| P^{U^{\perp}} v | \leq |P^{U^{\perp}} a| + |P^{U^{\perp}} b|\leq |P^{U^{\perp}} a|+\bd(U,U')|b| \leq \rho + \theta \delta,\\
			&&| P^{U} v | \leq |P^{U} a| + |P^{U} b|\leq \bd(U',U)|a|+|P^Ub| \leq \theta \rho + \delta.
			\earyst
			Therefore $v \in ( 1 + \theta \rho / \delta ) K(x)$. By the assumption $\delta^{1 - \alpha/10} > \rho$, we have $\rho^{1 + \alpha} \leq \rho^{\frac{1}{1 - \alpha/10}} < \delta$. Hence, for $\rho\ll_{\cF, \Lambda_1} 1$, we have $\theta \rho / \delta \leq C_{\Lambda_1} (6 \sqrt{2})^{\alpha} \rho^{1 + \alpha} / \delta \leq 1$. This proves $K(x')\subset 2K(x)$, and the reverse inclusion follows by symmetry.
			
			Next, we select the subcovering greedily: pick $x_1\in B$ and, inductively, as long as $B\setminus \bigcup_{l'\leq l} X(x_{l'},m,\rho,\delta)$ is non-empty, pick $x_{l+1}$ in this set.
			The retained centers satisfy $x_{l'} \notin X(x_l)$ for $l' > l$; hence, they are pairwise $\delta$-separated (recall that $X(x) \supset x+B^{\R^d}(\delta)$). Since $\Lambda$ is bounded, the procedure stops after finitely many steps, say $L$, at which point $B\subset \bigcup_{l=1}^L X(x_l,m,\rho,\delta)$. We write $X(x_l)=X(x_l,m,\rho,\delta)$ from now on.
			
			Let $z\in \mathbb{R}^d$ and let $\cI :=\{l\in \{1,\cdots,L\}: z\in 3X(x_l)\}$. We show that
			\ary
			|\cI|\leq 28^d.
			\eary
			Suppose $\cI\neq \emptyset$ and fix $l_0\in \cI$. Then \eqref{eqn:K inclusion} gives that, for every $l \in \cI$, we have $\frac{1}{2} K(x_{l_0}) \subset K(x_l) \subset 2 K(x_{l_0})$.  For each pair $l \neq l'$ in $I$, if $l'$ was retained after $l$, then $x_{l'} - x_l \notin K(x_l)$, and since $K(x_l) \supset \frac{1}{2} K(x_{l_0})$ and $K(x_{l_0})$ is symmetric, we obtain $x_l - x_{l'} \notin \frac{1}{2} K(x_{l_0})$ for every pair $l \neq l'$. Hence the sets $x_l + \frac{1}{4} K(x_{l_0})$, $l \in \cI$, are pairwise disjoint. We also have $$x_l+\frac{1}{4}K(x_{l_0}) \subset z + 3 K(x_l)+\frac{1}{4}K(x_{l_0}) \subset z + 6 K(x_{l_0})+\frac{1}{4}K(x_{l_0})\subset z+7K(x_{l_0}),$$ and volume counting gives
			\aryst
			| \cI | \leq \frac{ \mathrm{vol} ( 7 K(x_{l_0}) ) }{ \mathrm{vol} ( \frac{1}{4} K(x_{l_0}) ) } = 28^{d}.
			\earyst
		\end{proof}
		\begin{rema} \label{rem slabcovering}
			Note that the proof of \cref{lem slabcovering} shows that under the same assumptions, any two slabs $X(x, m, \rho, \delta)$ and $X(x', m, \rho, \delta)$ with $x, x' \in \Lambda_1$ whose triples intersect are $2$-commensurable.
		\end{rema}

		\begin{lemma} \label{lem 2XX}
			Let $m \in \{1, \cdots, d\}$, let $n \geq 1$ be an integer and let  $\rho, \delta, \eps, \eta'' > 0$. 
			Assume that $n \gg_{\cF,   \eps,   \eta'' } 1$, $e^{n( \chi_{m-1} - \eps )} > \rho \geq \delta > e^{n ( \chi_{m} + \eps )}$ and $\delta^{1 - \alpha/10} >  \rho$. 
			Then there exist $K_0= K_0( d, \eta'') > 1$ and $\Lambda_0 \subset \Lambda$ with $\mu(\Lambda_0) > 1 - \eta''$ such that for every $x \in \Lambda_0$
			\ary
			\label{eqn:measure 3K}
			\mu( 3 X(x, m, \rho, \delta) ) \leq K_0\mu( X(x, m, \rho, \delta) ).
			\eary
		\end{lemma}
		
		\begin{proof}

			Let us fix some $\Lambda_1 \subset \Lambda$ 
			given by \cref{prop Holdermap} with $\mu(\Lambda_1) > 1 - \eta'' / 2$ such that  $x \mapsto U^{m-1}(x)$ is  $C^{\alpha}$ restricted to $\Lambda_1$.

			Let $K_0=2C_d/\eta^{''}$ with $C_d>0$ the constant given in \Cref{lem slabcovering}. We define
			\aryst
			\Lambda_2 = \{ x \in \Lambda_1 : 	\mu( 3 X(x, m, \rho, \delta) ) > K_0\mu( X(x, m, \rho, \delta) ) \}.
			\earyst
			By \cref{lem slabcovering} applied with $B = \Lambda_2 \subset \Lambda_1$, we may extract from the family $\{ X(x, m, \rho, \delta) \}_{x \in \Lambda_2}$ a finite subcovering $\{ X(x_l, m, \rho, \delta )\}_{l = 1}^{L}$ of $\Lambda_2$ such that the tripled family $\{ 3 X(x_l, m, \rho, \delta )\}_{l = 1}^{L}$ has multiplicity bounded by $C_d$.   
			Then we have 
			\aryst
			\mu(\Lambda_2)  &\leq&  \sum_{l = 1}^{L} \mu( X(x_l, m, \rho, \delta ) ) \\
			&\leq& K_0^{-1}    \sum_{l = 1}^{L} \mu( 3 X(x_l, m, \rho, \delta ) ) \leq K_0^{-1} C_d\leq \eta''/2.
			\earyst
			Hence for $\Lambda_0 := \Lambda_1 \setminus \Lambda_2$, we have  $	\mu(\Lambda_0) > 1 - \eta''$, and it satisfies \eqref{eqn:measure 3K} by the definition of $\Lambda_2$.
		\end{proof}

		\begin{lemma} \label{lem slabmapping}
			Let $m\in \{1,\cdots,d\}$. Given parameters $\rho, \delta, \eps, \eta'' > 0$ and an integer $n > 0$. 
			Assume that $n \gg_{\cF, \eps, \eta''} 1$ and $e^{n( \chi_{m-1} - \eps )} > \rho \geq \delta > e^{n ( \chi_{m} + \eps )}$. 
			Then there is a subset $\Gamma'' \subset \Lambda \times \Sigma_{+}$ with $(\mu \otimes \nu_{+})(\Gamma'') > 1 - \eta''$ such that for every $(y, \bi_{+}) \in \Gamma''$,
			we have
			\aryst
			f_{\bi_{+}}^n(y + Y(\bi_{+}, n, m, \rho, \delta )) \overset{e^{n\eps}} \sim X(f_{\bi_{+}}^n(y),  m, \rho, \delta ).
			\earyst
		\end{lemma}

		\begin{proof}
			Let $g := \min_{1 \leq j \leq d-1} ( \chi_{j} - \chi_{j+1} ) > 0$ (the spectrum is simple) and $\eps_{0} := 5 (d+1) g$, which depends only on $\cF$. It suffices to prove the lemma for $\eps \leq \eps_{0}$. Let $\kappa =   \eps / (10 (d+1))$.
			
			We define $\Gamma''$ to be the collection of $(y, \bi_{+}) \in \Lambda \times \Sigma_{+}$ with the following properties:
			\enmt
			\item for every $1 \leq j \leq d$, we have 
			\ary \label{eq Gamma''prop1}
			e^{n (\chi_{j} - \kappa)} \leq \| D_yf_{\bi_{+}}^n |_{E_j(y, \bi_{+})  } \| \leq e^{n (\chi_{j} + \kappa)},
			\eary
			\item for every $1 \leq j \leq d - 1$, we have
			\ary 
			\label{eq Gamma''prop2}
			{\sin \measuredangle(V_{j+1}(\bi_{+}),  E_j(y, \bi_{+}) ) \geq e^{- n \kappa},} \\
			\label{eq Gamma''prop3}
			{\sin \measuredangle( V_{j+1}(\sigma^n(\bi_{+})), E_j (f_{\bi_{+}}^n(y), \sigma^{n}\bi_{+})   ) \geq e^{- n \kappa}.}
			\eary
			\eenmt
			By \Cref{thm Oseledets}, we see that by letting $n \gg_{\cF, \eps, \eta''} 1$, we have  $(\mu \otimes \nu_{+})(\Gamma'') > 1 - \eta''$: indeed, \eqref{eq Gamma''prop1} holds on a large set by Oseledets' theorem applied to the Oseledets direction $E_{j}(y, \bi_{+}) $ together with Egorov's theorem; for \eqref{eq Gamma''prop3}, we also use that $(y, \bi_{+}) \mapsto (f_{\bi_{+}}^n(y), \sigma^{n}\bi_{+})$ preserves $\mu \otimes \nu_{+}$.
			
			{It remains to  verify 	 that
				\ary \label{eq slabimage}
				f_{\bi_{+}}^n(y + Y(\bi_{+}, n, m, \rho, \delta )) \overset{{e^{(2d+3) n\kappa}}} \sim X(f_{\bi_{+}}^n(y),  m, \rho, \delta )
				\eary
				{for every $(y, \bi_{+}) \in \Gamma''$; this gives the lemma since $(2d+3) \kappa = \frac{2d+3}{10(d+1)} \, \eps \leq \eps$.} 
				
				Write $M = D_{y} f_{\bi_{+}}^n${, $z := f_{\bi_{+}}^n(y)$, $V_{j} := V_{j}(\bi_{+})$, $E_{j} := E_{j}(y, \bi_{+}) = V_{j} \cap U^{j}(y)$, and at the image pair $(z, \sigma^{n}(\bi_{+}))$, $V_{j}' := V_{j}(\sigma^{n}(\bi_{+}))$, $E_{j}' := E_{j}(z, \sigma^{n}(\bi_{+}))$, and $U' := U^{m-1}(z)$}.{ Set $\rho_{j} := \rho$ for $j \leq m-1$, $\rho_{j} := \delta$ for $j \geq m$, $r_{j} := e^{-n\chi_{j}} \rho_{j}$. { By the definition of $g$ and the hypothesis on $e^{n(\chi_{m-1}-\epsilon)}>\rho\geq \delta>e^{n(\chi_m+\epsilon)}$, we obtain $r_{k-1}\leq e^{-2n\eps}r_{k}$ for every $2\leq k\leq d$ --- indeed $r_{k-1}\leq e^{-ng}r_{k}$ for $k\neq m$, while at $k=m$ the inequality is exactly $\rho\leq e^{n(\chi_{m-1}-\chi_m-2\eps)}\delta$ ---} and for $e^{-2n\eps} \leq 1/2$,
					\ary \label{eq step1 radii}
					\sum_{j \leq k} r_{j} \ \leq \ 2 \, r_{k} \qquad (1 \leq k \leq d).
					\eary
					{Note that $r_{j} = e^{-n\chi_{j}} \rho_{j}$ increases with $j$ by \eqref{eq step1 radii}, whereas $\rho_{j}$ equals $\rho$ for $j \leq m-1$ and $\delta$ for $j \geq m$, with $\rho \geq \delta$. The long direction of the box therefore passes from the high indices at the source to the low indices at the image.}
					
					Consider the boxes
					\aryst
					\begin{array}{c}
						Y := Y(\bi_{+}, n, m, \rho, \delta), \qquad \widetilde Y := \sum_{j = 1}^{d} B^{E_{j}}(r_{j}), \qquad \widetilde Y' := \sum_{j = 1}^{d} B^{E_{j}'}(\rho_{j}), \\[6pt]
						X_{0} := X(z, m, \rho, \delta) - z = B^{U'}(\rho) + B^{U'^{\perp}}(\delta) .
					\end{array}
					\earyst

					\begin{center}
						\begin{tikzpicture}[scale=0.95,>=latex]
							\def\sx{0.423}\def\sy{0.906}
							\def\nx{-0.906}\def\ny{0.423}
							\begin{scope}
								\def\L{2.30}\def\h{0.55}
								\fill[gray!12] (-\L,-\h) rectangle (\L,\h);
								\draw[thick]   (-\L,-\h) rectangle (\L,\h);
								\draw[thick,dashed]
								({\L+\h*\sx},{\h*\sy}) -- ({\L-\h*\sx},{-\h*\sy})
								-- ({-\L-\h*\sx},{-\h*\sy}) -- ({-\L+\h*\sx},{\h*\sy}) -- cycle;
								\draw[->,thick] (0,0) -- (\L+1.05,0) node[below right]{\small $f_{2}=E_{2}$};
								\draw[->,thick] (0,0) -- (0,1.55) node[above]{\small $f_{1}$};
								\draw[->,thick] (0,0) -- ({1.30*\sx},{1.30*\sy}) node[right]{\small $E_{1}$};
								\node at (-1.70,1.15) {\small $Y$};
								\node at (1.95,1.15)  {\small $\widetilde Y$};
							\end{scope}
							\begin{scope}[xshift=9.6cm]
								\def\L{1.70}\def\h{0.52}
								\fill[gray!12]
								({\L*\sx+\h*\nx},{\L*\sy+\h*\ny}) -- ({\L*\sx-\h*\nx},{\L*\sy-\h*\ny})
								-- ({-\L*\sx-\h*\nx},{-\L*\sy-\h*\ny}) -- ({-\L*\sx+\h*\nx},{-\L*\sy+\h*\ny}) -- cycle;
								\draw[thick]
								({\L*\sx+\h*\nx},{\L*\sy+\h*\ny}) -- ({\L*\sx-\h*\nx},{\L*\sy-\h*\ny})
								-- ({-\L*\sx-\h*\nx},{-\L*\sy-\h*\ny}) -- ({-\L*\sx+\h*\nx},{-\L*\sy+\h*\ny}) -- cycle;
								\draw[thick,dashed]
								({\L*\sx+\h},{\L*\sy}) -- ({\L*\sx-\h},{\L*\sy})
								-- ({-\L*\sx-\h},{-\L*\sy}) -- ({-\L*\sx+\h},{-\L*\sy}) -- cycle;
								\draw[->,thick] (0,0) -- (1.5,0) node[below right]{\small $E_{2}'=f_{2}'$};
								\draw[->,thick] (0,0) -- ({2.05*\sx},{2.05*\sy}) node[above right]{\small $E_{1}'=U'$};
								\draw[->,thick] (0,0) -- ({1.85*\nx},{1.85*\ny}) node[below]{\small $U'^{\perp}$};
								\node at (-1.90,1.05) {\small $X_{0}$};
								\node at (1.85,0.95)  {\small $\widetilde Y'$};
							\end{scope}
							\draw[->,very thick] (3.85,2.05) -- (5.65,2.05) node[midway,below]{\small $M$};
						\end{tikzpicture}
						{\small The case $d=1,m=1$: The $E_{2}$-direction is drawn horizontal in both panels, so the box, which is long along $E_{2}=f_{2}$ at the source, is long along $E_{1}'=U'$ at the image. 
						}
					\end{center}
					\smallskip

					With $\Theta := e^{n\kappa} (1 + e^{n\kappa})^{d-1}$, we shall prove
					\ary \label{eq boxes}
					\begin{array}{c}
						\mbox{(i)} \quad { \widetilde Y \ \subset \ 2 \, Y \ \subset \ 4 e^{n\kappa} \, \widetilde Y} , \\[4pt]
						\mbox{(ii)} \quad { e^{-n\kappa} \, \widetilde Y' \ \subset \ M \widetilde Y \ \subset \ e^{n\kappa} \, \widetilde Y'} , \\[4pt]
						\mbox{(iii)} \quad { \widetilde Y' \ \subset \ d \, X_{0} \ \subset \ 2 d \Theta \, \widetilde Y'} .
					\end{array}
					\eary
					Assuming \eqref{eq boxes}, since $d \leq 2^{d} \leq 2 \Theta$, 
					\aryst
					M Y \ \subset \ 2 e^{n\kappa} \, M \widetilde Y \ \subset \ 2 e^{2n\kappa} \, \widetilde Y' \ \subset \ 4 e^{2n\kappa} \Theta \, X_{0}
					\quad \mbox{and} \quad
					X_{0} \ \subset \ 2 \Theta \, \widetilde Y' \ \subset \ 2 e^{n\kappa} \Theta \, M \widetilde Y \ \subset \ 4 e^{n\kappa} \Theta \, M Y.
					\earyst
					This implies $$f_{\bi_{+}}^n(y + Y) = z + M Y \overset{C_{n}}{\sim} X(z, m, \rho, \delta)$$ with $C_{n} := 4 e^{2n\kappa} \Theta = 4 e^{3n\kappa} (1 + e^{n\kappa})^{d-1} \leq 2^{d+1} e^{(d+2) n\kappa} \leq e^{(2d+3) n\kappa}$ as soon as $e^{n\kappa} \geq 2$. This is \eqref{eq slabimage}.}
				
				{The proof of \eqref{eq boxes} is based on estimating the coordinates under two different bases. Let $f_{j} \in V_{j} \cap V_{j+1}^{\perp}$ be unit vectors ($V_{d+1} := \{0\}$). Then $V_{j} = \mathrm{span}(f_{j}, \cdots, f_{d})$, and 
					\aryst
					Y \ = \ \Big\{ \sum_{k = 1}^{d} s_{k} f_{k} \ : \ |s_{k}| < r_{k} \Big\} .
					\earyst
					Let $e_{j} \in E_{j}$ be unit vectors and $a_{kj} := \langle e_{j}, f_{k} \rangle$. As $e_{j} \in V_{j}$, $a_{kj} = 0$ for $k < j$. Moreover, $|a_{kj}| \leq 1$, and since $V_{j} = \R f_{j} \oplus V_{j+1}$, $$|a_{jj}| = \sin \measuredangle (e_{j}, V_{j+1}) = \sin \measuredangle (V_{j+1}, E_{j}) \geq e^{-n\kappa}$$ for $j \leq d-1$ by \eqref{eq Gamma''prop2}, while $|a_{dd}| = 1$ because $E_{d} = V_{d} = \R f_{d}$. Hence, $(a_{kj})$ is lower triangular with a non-zero diagonal; $(e_{j})$ is a basis of $\R^{d}$. We have $$\widetilde Y = \{ \sum_{j} t_{j} e_{j} : |t_{j}| < r_{j} \},$$ and the coordinates of $v = \sum_{j} t_{j} e_{j} = \sum_{k} s_{k} f_{k}$ satisfy $s_{k} = \sum_{j \leq k} a_{kj} t_{j}$, and solving from the top, we obtain for $1\leq k\leq d$,
					\begin{equation} \label{eq step1 recursion}
						|s_{k}| \leq \ \sum_{j \leq k} |t_{j}|
						\qquad \mbox{and} \qquad
						|t_{k}| \ \leq \ \frac{1}{|a_{kk}|} \Big( |s_{k}| + \sum_{j < k} |a_{kj}| \, |t_{j}| \Big)
						\ \leq \ e^{n\kappa} \Big( |s_{k}| + \sum_{j < k} |t_{j}| \Big).
					\end{equation}
					Since $|s_{k}| \leq \| v \|$, the second inequality gives by induction on $k$
					\ary \label{eq step1 normbound}
					|t_{k}| \ \leq \ e^{n\kappa} ( 1 + e^{n\kappa} )^{k-1} \, \| v \| \ \leq \ \Theta \, \| v \| \qquad (1 \leq k \leq d)
					\eary
					(if this holds for all $j < k$, then $\sum_{j < k} |t_{j}| \leq ( (1 + e^{n\kappa})^{k-1} - 1 ) \| v \|$). All of the above holds verbatim at the image pair $(z,\sigma^n(\bi_+))$, with ($f_{j}', e_{j}', a_{kj}'$) and \eqref{eq Gamma''prop3} in place of \eqref{eq Gamma''prop2}.}
				
				{We prove (i) in \eqref{eq boxes}. If $v \in \widetilde Y$, then \eqref{eq step1 recursion} and \eqref{eq step1 radii} give $|s_{k}| < \sum_{j \leq k} r_{j} \leq 2 r_{k}$, which proves $v \in 2 Y$. Conversely, if $v \in Y$, we show by induction on $k$ that $|t_{k}| < 2 e^{n\kappa} r_{k}$: granting this for $j < k$, \eqref{eq step1 radii} and $r_{k-1} \leq e^{-ng'} r_{k}$ give $\sum_{j < k} |t_{j}| \leq 4 e^{n\kappa} r_{k-1} \leq 4 e^{-n(g' - \kappa)} r_{k} \leq r_{k}$ for $n \gg_{\cF,  \eps} 1$. Then \eqref{eq step1 recursion} gives $|t_{k}| < e^{n\kappa} (r_{k} + r_{k})$, proving $Y\subset 2^{n\kappa}\widetilde{Y}$.}
				
				\smallskip
				{The middle inclusions (ii) in \eqref{eq boxes} follow from the identity $M B^{E_{j}}(r_{j}) = B^{E_{j}'}( \| M|_{E_{j}} \| \, r_{j} )$, and $e^{-n\kappa} \rho_{j} = e^{n(\chi_{j} - \kappa)} r_{j} \leq \| M|_{E_{j}} \| \, r_{j} \leq e^{n(\chi_{j} + \kappa)} r_{j} = e^{n\kappa} \rho_{j}$ by \eqref{eq Gamma''prop1}, for each $j\in \{1,\cdots,d\}$.}
				
				\smallskip
				{Finally, we compare $\widetilde Y'$ with $X_{0}$ and prove the (iii) in \eqref{eq boxes}. Note that $X_{0} = \{ v \in \R^{d} : \| P^{U'^{\perp}} v \| < \rho , \ \| P^{U'} v \| < \delta \}$, and 
					\ary \label{eq step3 U}
					U' \ = \ E_{1}' \oplus \cdots \oplus E_{m-1}'.
					\eary
					Let $v = \sum_{j} t_{j} e_{j}' \in \widetilde Y'$ and split $v = a + b$ with $a := \sum_{j \leq m-1} t_{j} e_{j}' \in U'$ and $b := \sum_{j \geq m} t_{j} e_{j}'$. Then $\| a \| < (m-1) \rho$ and $\| b \| < (d-m+1) \delta$, so $\| P^{U'^{\perp}} v \| \leq \| a \| + \| b \| < d \rho$ and $\| P^{U'} v \| = \| P^{U'} b \| < d \delta$, which proves \ $v \in d X_{0}$. 
					
					Conversely, let $v \in X_{0}$ and split $v = p + q$ with $p := P^{U'^{\perp}} v \in U'$, $\| p \| < \rho$, and $q := P^{U'} v$, $\| q \| < \delta$. Writing $t_{j}(\cdot)$ for the coordinates in the basis $(e_{j}')$, we have $t_{j}(p) = 0$ for $j \geq m$ by \eqref{eq step3 U}, and $|t_{j}(p)| < \Theta \rho$ and $|t_{j}(q)| < \Theta \delta$ for all $j$ by applying \eqref{eq step1 normbound} to $p,q$ at the image pair $(z,\sigma^n(\bi_+))$. Hence $|t_{j}(v)| < 2 \Theta \rho$ for $j \leq m-1$ and $|t_{j}(v)| = |t_{j}(q)| < \Theta \delta$ for $j \geq m$, which proves \ $v \in 2 \Theta \widetilde Y'$.}}
		\end{proof}

		Taking $\rho = \delta$ in \cref{lem slabmapping}, we obtain the following. 
		
		\begin{cor} \label{cor slabmapping}
			Let $m \in \{1, \cdots, d\}$, let $\delta, \eps, \eta'' > 0$ and let $n > 0$ be an integer.
			Assume that $n \gg_{\cF,  \eps, \eta''} 1$ and $e^{n(\chi_{m-1} - \eps)} > \delta > e^{n(\chi_{m} + \eps)}$.
			Then there is a subset $\Gamma'' \subset \Lambda \times \Sigma_{+}$ with $(\mu \otimes \nu_{+})(\Gamma'') > 1 - \eta''$ such that for every $(y, \bi_{+}) \in \Gamma''$,
			\aryst
			f_{\bi_{+}}^n \big( y + Y(\bi_{+}, n, m, \delta, \delta) \big) \ \overset{\sqrt{2} \, e^{n\eps}}{\sim} \ B^{\R^{d}} \big( f_{\bi_{+}}^n(y), \delta \big) .
			\earyst
		\end{cor}

		\subsection{A covering lower bound for positive-mass sets}
		
		\begin{lemma} \label{lem conditionalentropy}
			Let $m \in \{1, \cdots, d\}$ with $D_{m} = m-1$, let $n \geq 1$ be an integer,  and let $\eps, \eps', \eta'', \eps'', \eta > 0$. Assume that $\eps'' \ll_{\cF,   \eps, \eps', \eta} 1$,  $n \gg_{\cF,  \eps, \eps', \eta'', \eta } 1$.
			Then  for every $\rho > \delta > 0$ satisfying $e^{n( \chi_{m-1} - \eps )} > \rho > \delta > e^{n ( \chi_{m} + \eps )}$ and $\rho >  \delta^{1 - \eps'}$, there is a subset $\Gamma' \subset \Lambda \times \Sigma_{+}$ with $(\mu \otimes \nu_{+})(\Gamma') > 1 - \eta''$ such that for every $(y, \bi_{+}) \in \Gamma'$, the following is true. Set $Y' = Y(\bi_{+}, n, m, \rho, \delta )$ and $Y'' = Y(\bi_{+}, n, m, \delta, \delta )$. Then  for every Borel subset $A$ with $\mu^{y + e^{- \eps'' n} Y'}(A) > \eta$, the set  $A$ requires at least $(\rho / \delta)^{m - 1 - \eps}$ many translations of $Y''$ to cover. \footnote{It is trivial to see that $Y'$ is covered by $C (\rho / \delta)^{m - 1}$ many translations of $Y''$.}  
		\end{lemma}

		\begin{proof} 	
			Denote 
			\aryst
			W'(\bi_{+}) &=& S( 0, (  V_{d+1-j}(\bi_+) )_{j = 1}^{d},  (  e^{- n \chi_d} \delta, \cdots, e^{- n \chi_{m}} \delta, \delta, \cdots, \delta  )  ).
			\earyst
			We first claim that
			\ary \label{eq W'intube}
			x + W'(\bi_{+}) \subset x + V_m(\bi_{+}) + B^{V_m(\bi_{+})^{\perp}}( d \delta ), \qquad x \in \R^d.
			\eary
			Indeed, each layer of $W'(\bi_{+})$ with exponent $\chi_k$, $k \geq m$, lies in $V_k(\bi_{+}) \subset V_m(\bi_{+})$, while each layer with exponent $\chi_k$, $k \leq m - 1$, is a ball of radius $\delta$, whose projection to $V_m(\bi_{+})^{\perp}$ has length at most $\delta$; since there are at most $m - 1 \leq d$ such layers, the claim follows. 
			By \cref{thm Exactdim}, for $(\mu\otimes \nu_+)$-a.e. $(x,\bi_+)\in \Lambda\times \Sigma_+$, 
			we have
			\aryst
			\lim_{\delta \to 0}  \frac{\log	\mu( x + V_m(\bi_{+}) + B^{V_m(\bi_{+})^{\perp}}( d \delta ) ) }{\log (d \delta) } = D_{m} = m - 1.
			\earyst
			Hence, by \eqref{eq W'intube}, for $(\mu\otimes \nu_+)$-a.e. $(x,\bi_+)\in \Lambda\times \Sigma_+$,
			\aryst
			\liminf_{\delta \to 0} \frac{ \log \mu( x +  W'(\bi_{+}) ) }{ \log \delta } = D_m = m - 1.
			\earyst

			Let $C_B=C_B(d)>0$ be the constant given in \Cref{lem besicovich}. Set 
			\ary
			\label{eqn:epsilon star}
			\eps_*:=\frac{\eta \eta''}{8C_B}
			\eary
			which depends only on $d,\eta$ and $\eta''$.
			By letting $n \gg_{\cF,  \eps_*} 1$, there is a subset $\Gamma_* \subset \Lambda \times \Sigma_{+}$ with $(\mu \otimes \nu_{+})(\Gamma_*) > 1 - \eps_*$ such that  for every $(y, \bi_{+}) \in \Gamma_*$, we have 
			\aryst
			\mu( y + W'(\bi_{+}) ) \leq \delta^{m - 1 - \eps/4}.
			\earyst
			For each $\bi_{+} \in \Sigma_{+}$, we denote
			\aryst
			\Gamma_* (\bi_{+}) = \{ y\in \Lambda : (y, \bi_{+})\in\Gamma_* \}.
			\earyst
			Given $\bi_{+} \in \Sigma_{+}$ and $y \in \Lambda$, by Vitali's Covering Lemma, we see that the set $(y + Y'') \cap \Gamma_{*}(\bi_{+})$ can be covered by at most $O_d(e^{- n \sum_{i = 1}^{m-1}\chi_i})$-many translates of  $W'(\bi_{+})$ centered at points in $\Gamma_{*}(\bi_{+})$. Thus we have
			\ary \label{eq upperboundinabox}
			\mu( (y + Y'' ) \cap \Gamma_{*}(\bi_{+}) ) \leq C_d e^{- n \sum_{i = 1}^{m-1}\chi_i}  \cdot \delta^{m - 1 - \eps/4}.
			\eary
			
			Now, we let $ \Gamma'$ be the set of $(y,  \bi_{+})$ such that  
			\ary 
			\label{eq lowerboundmeasure}
			\mu( y + e^{-\eps''n} Y' ) > (\mathrm{vol}( P^{ V_{m}(\bi_{+})} ( e^{-\eps''n} Y' ) ))^{1 + \eps}
			\eary
			and
			\ary \label{eq lowerbounddensity}
			\mu^{y + e^{-\eps''n} Y'}( \Gamma_{*}(\bi_{+}) ) > 1 - \eta / 2.
			\eary
			\textbf{We claim that} for $\eps_* \ll_{\eta, \eta''} 1$ and $n \gg_{\cF,  \eps_*, \eps, \eta, \eta''} 1$,
			\ary
			\label{eqn:Gamma ''}
			(\mu\otimes \nu_+)(\Gamma') > 1 - \eta''.
			\eary
			We first show that the condition \eqref{eq lowerbounddensity} holds on a set of measure $> 1 - \eta''/2$ by \cref{lem besicovich}. For a $\nu_+$-typical $\bi_+$, set $Z:=e^{-\eps''n}Y'$. Apply \Cref{lem besicovich} with $\omega=\mu$, $\Gamma=\Gamma_*(\bi_+)$, $W=Z$, and $\eta/4$ in place of $\eta$:
			the set
			\aryst
			A(\bi_+):=\{y\in \R^d: \mu((y+Z)\cap \Gamma_*(\bi_+))\geq (1-\eta/4)\mu(y+Z)\}
			\earyst
			satisfies $\mu(\R^d\setminus A(\bi_+))\leq \frac{4C_B}{\eta}\mu(\Lambda\setminus \Gamma(\bi_+))$. Hence, 
			\aryst
			&&(\mu\otimes \nu_+)\big(\{(y,\bi_+)\in \Lambda\times \Sigma_+: \eqref{eq lowerbounddensity}\,\,\text{fails}\}\big)\leq \frac{4C_B}{\eta}\int \mu(\Lambda\setminus \Gamma_*(\bi_+))d\nu_+(\bi_+)\\
			&=& \frac{4C_B}{\eta} (\mu\otimes \nu_+)((\Lambda\times \Sigma_+)\setminus \Gamma_*)<\frac{4C_B\epsilon_*}{\eta}=\eta''/2,
			\earyst
			where we use the definition of $\epsilon_*$ given in \eqref{eqn:epsilon star}.
			
			Next, we  show that \eqref{eq lowerboundmeasure} fails on a set of $(\mu\otimes \nu_+)$-measure less than $\eta''/2$. Assume $\eps''\leq \epsilon$. Then $r_j:=e^{-\eps''n}e^{-n\chi_j}\rho<1$ for $1\leq j\leq m-1$, and $r_j:=e^{-\eps'' n} e^{-n\chi_j}\delta\geq 1$ for $m\leq j\leq d$. Using the hypotheses on $\rho$ and $\delta$, we obtain
			\begin{equation}
				\label{eqn:vol PZ}
				\mathrm{vol}(P^{V_m(\bi_+)}(Z))=\prod_{j=1}^{m-1}2r_j=(2e^{-n\eps''})^{m-1} \rho^{m-1} e^{-n\sum_{j=1}^{m-1}\chi_j}
				\leq(2e^{-n\epsilon})^{m-1}\leq e^{-n\epsilon/2}.
			\end{equation}
			Let
			\ary
			B(\bi_+):=\{y\in \Lambda: \mu(y+Z)\leq (\mathrm{vol}(P^{V_m(\bi_+)}(Z)))^{1+\epsilon}\}.
			\eary
			We cover $B(\bi_+)$ by translates of $Z$ centered in $B(\bi_+)$. More precisely, let $\{y_i\}_{i\in J}$ be a maximal family of points in $B(\bi_+)$ such that $y_i+\frac{1}{2}Z$, $i\in J$, are pairwise disjoint. Then we claim
			\ary
			\label{eqn:J card}
			|J|\leq N:=4^d/\mathrm{vol}(P^{V_m(\bi_+)}(Z)).
			\eary
			
			We deduce from the maximality of $\{y_i\}_{i\in J}$ that for any $y\in B(\bi_+)\setminus \{y_i\ :  i\in J\}$ and any $i\in J$, $(y+\frac{1}{2}Z)\cap (y_i+\frac{1}{2}Z)\neq \emptyset$, which gives $y-y_i\in \frac{1}{2}Z-\frac{1}{2}Z=Z$. Hence, $B(\bi_+)\subset \bigcup_{i\in J}(y_i+Z)$. Using \eqref{eqn:vol PZ} and \eqref{eqn:J card}, we obtain
			\aryst
			\mu(B(\bi_+))\leq \sum_{i\in J}\mu(y_i+Z)\leq |J|(\mathrm{vol}(P^{V_m(\bi_+)}(Z)))^{1+\epsilon}\leq 4^{d}e^{-n\epsilon^2/2}.
			\earyst
			Consequently, for $n\gg _{\epsilon,\eta''} 1$, integrating the above inequality in $\bi_+$ gives
			\aryst
			(\mu\otimes \nu_+)\big(\{(y,\bi_+)\in \Lambda\times \Sigma_+: \eqref{eq lowerboundmeasure} \,\,\text{fails}\}\big)=\int \mu(B(\bi_+)) d\nu_+(\bi_+)\leq 4^{d} e^{-n\epsilon^2/2}<\frac{\eta''}{2},
			\earyst
			\textbf{completing the proof of the claim \eqref{eqn:Gamma ''}.}

			Take an arbitrary $(y, \bi_{+}) \in \Gamma'$ and let $A$ be a Borel subset with $\omega(A) > \eta$, where we abbreviate
			\aryst
			\omega \ := \ \mu^{y + e^{-\eps''n} Y'} .
			\earyst
			By \eqref{eq lowerbounddensity}, $\omega(\Gamma_{*}(\bi_{+})) > 1 - \eta/2$. Hence the measures
			\aryst
			\omega' := \omega( \, \cdot \, \cap \Gamma_{*}(\bi_{+}) ) , \qquad \omega'' := \omega( \, \cdot \, \setminus \Gamma_{*}(\bi_{+}) )
			\earyst
			satisfy $\omega = \omega' + \omega''$ and $|\omega''| \leq \eta/2$. Moreover, for every $z \in \Gamma_{*}(\bi_{+})$, \eqref{eq upperboundinabox} and \eqref{eq lowerboundmeasure} give
			\aryst
			\omega'\big( z + Y'' \big) \ = \ \frac{ \mu\big( (z + Y'')\cap \Gamma_{*}(\bi_{+}) \big) }{ \mu( y + e^{-\eps''n} Y' ) } \ \leq \ \frac{ C_d \, e^{- n \sum_{j = 1}^{m-1}\chi_j} \, \delta^{m - 1 - \eps/4} }{ \big( \mathrm{vol}( P^{V_m(\bi_{+})} ( e^{-\eps''n} Y' ) ) \big)^{1 + \eps} } \ =: \ t .
			\earyst
			\Cref{lem counting}, applied to $\omega$, to the family $\cK$ of translates of $Y''$ centered in $\Gamma_{*}(\bi_{+})$, and to $A$, shows that $A$ requires at least \[(\omega(A) - |\omega''|)/t \geq \eta/(2t)\] members of $\cK$ to cover. Substituting the value of $\mathrm{vol}( P^{V_m(\bi_{+})} ( e^{-\eps''n} Y' ) )$ given in \eqref{eqn:vol PZ}, the remaining factors are absorbed into the exponent of $\rho/\delta$ at the cost of an $O(\eps)$ loss, by $\eps'' \ll_{\cF,\eps} 1$, $\rho > \delta^{1-\eps'}$ and the choice of $n$. Since $\eps$ is arbitrary, we may equivalently fix it smaller from the outset. Hence $A$ requires at least $(\rho/\delta)^{m - 1 - O(\eps)}$ translations of $Y''$, which concludes the proof.
		\end{proof}

		\subsection{The Slab-Filling property: proof of Proposition \ref{prop slabfilling}}

		Let  $n \gg_{\cF,   \eps, \eps', \eta,  \eta_0} 1$ be a large integer to be determined in due course.  Let $\eta''\ll_{\eta_0,\eta} 1$ be a small parameter determined by \eqref{eqn:eta'' cond 1} and \eqref{eq muXAwLambdaX}. Let $\epsilon''\ll_{\mathcal{F},\epsilon, \epsilon',\eta} 1$	be a small parameter that satisfies the requirement of  \cref{lem conditionalentropy}, and that $d \eps'' \leq (|\chi_m| - \eps) \eps \eps'$ (used for the final absorption).

		Let $\Lambda_0\subset \Lambda$, $K>1$ be given by \cref{lem 2XX} with $\eta''$ and $\Lambda_0$ also satisfy \cref{prop Holdermap}.
		
		Let $\Gamma' \subset \Lambda \times \Sigma_{+}$ be given by  \cref{lem conditionalentropy} for $\eta''$, $\frac{\eta }{4\cdot 3^d K}$ in place of $\eta$ for some $C_d > 1$ depending only on $d$, and $\epsilon'''$.
		Let $\Gamma'' \subset  \Lambda \times \Sigma_{+}$ be given by   \cref{lem slabmapping} and \cref{cor slabmapping}  for $\eta''$, and for $\eps''$ in place of $\eps$.
		
		Since $(\mu \otimes \nu_{+})(\Gamma' \cap \Gamma'') > 1 - 2 \eta''$, we see that there exist $\bi_{+} \in \Sigma_{+}$ and a subset $\Lambda_1 \subset \Lambda$ with $\mu(\Lambda_1 )  > 1 - 3 \eta''$ such that for every $x \in \Lambda_1$, we have $ (x_{-n}, x[n] \bi_{+}) := \sigma^{-n}(x, \bi_{+})  \in \Gamma' \cap \Gamma''$. Hence,  
		\begin{properties}
			\item\label{P:comparision slabs} by \cref{lem slabmapping},
			\ary \label{eq fYX}
			f_{x[n]}(x_{-n} + Y(x[n] \bi_{+}, n, m, \rho, \delta ) ) \overset{e^{n\eps''}} \sim X( x, m, \rho, \delta );
			\eary
			\item\label{P:card lower} by \cref{lem conditionalentropy}, every Borel subset $A$ with $\mu^{x_{-n} + e^{- \eps'' n} Y'}(A) > \eta^2/(40K)$ is covered by at least $(\rho / \delta)^{m-1-\eps}$-many translations of $Y(x[n] \bi_{+}, n, m, \delta, \delta )$.
		\end{properties}

		We set
		\aryst
		\widetilde\Lambda = \Lambda_0 \cap \Lambda_1.
		\earyst
		By our choices of $\Lambda_0$ and $\Lambda_1$, we have 
		\aryst
		\mu(\widetilde\Lambda ) > 1 - 4 \eta''.
		\earyst
		We let $\Lambda''$ be the set of $x \in \widetilde\Lambda$ such that 
		\aryst
		\mu^{X(x, m, \rho, \delta)}(\widetilde\Lambda ) > 1 - \sqrt{\eta''}.
		\earyst
		Due to \cref{lem slabcovering}, one can extract from $\{ X(x, m, \rho, \delta) \}_{x \in  B}$ with any $B \subset \Lambda_0$, a subcovering of $B$ with multiplicity $O_d(1)$.
		Then from covering argument, it is direct to see that 
		\ary
		\label{eqn:eta'' cond 1}
		\mu(\Lambda'') > 1 -  C_d \sqrt{\eta''}.
		\eary
		By letting $\eta'' \ll_d \eta_0$, we deduce that $\mu(\Lambda'') > 1  - \eta_0$.

		Fix an arbitrary $x \in \Lambda''$, and an arbitrary Borel subset $A$ with $\mu^{X(x, m, \rho, \delta)}(A) > \eta$.  For simplicity, we denote  $X = X(x, m, \rho, \delta)$.
		By letting $\eta'' < \eta^2 / 4$, we have  
		\begin{align} \label{eq muXAwLambdaX}
			\mu(X \cap A \cap \widetilde\Lambda) > \mu(X \cap A) - \mu(X \setminus \widetilde\Lambda) > (\eta - \sqrt{\eta''} )  \mu(X) > \eta \mu(X) / 2.
		\end{align}
		By $x \in \Lambda_0$, we deduce 
		\ary \label{eq X2X0}
		\mu(3 X) \leq K \mu( X)\leq 2\eta^{-1}K\mu(X \cap A \cap \widetilde\Lambda).
		\eary
		
		Since $\mu$ is a $\nu_0$-stationary measure  we have
		\aryst
		\mu  = \int_{I^n} (f_{\bj})_*  \mu  d\nu_0^{\otimes n}(\bj).
		\earyst
		By \eqref{eq X2X0} and Markov's inequality, we deduce that there is $\bj \in I^n$ such that 
		\ary \label{eq 2XXAwLambda}
		\mu( (f_{\bj})^{-1}  (3 X) ) \leq 2 \eta^{-1} K  \mu  ( (f_{\bj})^{-1} (X \cap A \cap \widetilde\Lambda) ).
		\eary
		
		Now, let us denote
		\aryst
		Y' &=&  Y(\bj \bi_{+}, n, m, \rho, \delta ), \\
		Y'' &=&  Y(\bj \bi_{+}, n, m, \delta, \delta ).
		\earyst
		We fix a covering $\{ y_l  + e^{ - n \eps''} Y' \}_{l = 1}^{L}$ of $(f_{\bj})^{-1} (X  \cap A \cap \widetilde\Lambda ) \cap \Lambda$  with $\{ y_l \}_{l = 1}^{L} \subset (f_{\bj})^{-1} (X \cap A \cap \widetilde\Lambda) \cap \Lambda$, and with multiplicity at most $3^d$. 
		For each $1 \leq l \leq L$, we have $y_l \in (f_{\bj})^{-1}(\widetilde\Lambda)  \cap \Lambda \subset (f_{\bj})^{-1}(\Lambda_1) \cap \Lambda$,  and hence $(y_l, \bj\bi_{+}) \in \Gamma' \cap \Gamma''$.
		Then by \eqref{eq fYX}
		\aryst
		\bigcup_{l = 1}^{L} ( y_l +  e^{ - n \eps''} Y'  ) \subset ( f_{\bj})^{-1}(  3X ).
		\earyst
		Then by \eqref{eq 2XXAwLambda}, we have 
		\begin{align}
			\sum_{l = 1}^{L} \mu(  y_l +  e^{ - n \eps'' } Y'  )  &\leq 3^d  \mu( (f_{\bj})^{-1}  (3 X) )   \leq  2 \cdot 3^d\eta^{-1} K  \mu( ( f_{\bj})^{-1}( X \cap  A \cap \widetilde\Lambda)  )   ) \\
			&\leq 2\cdot 3^d \eta^{-1} K    \sum_{l = 1}^{L}  \mu(  ( y_l +  e^{ - n \eps''} Y'  )  \cap ( f_{\bj})^{-1}( A \cap \widetilde\Lambda)  ) ).
		\end{align}
		We deduce that there is $1 \leq l \leq L$ such that 
		\aryst
		\mu^{y_l + e^{ - n \eps''}  Y'}( (f_{\bj})^{-1}( A \cap \widetilde\Lambda) ) > \frac{\eta}{4\cdot 3^d K}.
		\earyst
		Then by \ref{P:card lower} applied at $(y_l, \bj\bi_{+}) \in \Gamma'$ (the condition is met by the above inequality), the set $(f_{\bj})^{-1}( A \cap \widetilde\Lambda)$ requires at least $(\rho / \delta)^{m - 1 - \eps}$ many translations of $Y''$ to cover. We claim that   $N_{\delta}(A) \geq  (\rho / \delta)^{m - 1 - 2\eps}$. Indeed, let $\{ B(z_k, \delta) \}_{k}$ be a minimal covering of  $A \cap \widetilde\Lambda \cap f_{\bj}(\Lambda)$ centered in this set. We have $y_k' := (f_{\bj})^{-1}(z_k) \in (f_{\bj})^{-1}(\widetilde\Lambda) \cap \Lambda \subset (f_{\bj})^{-1}(\Lambda_1) \cap \Lambda$, hence $(y_k', \bj\bi_{+}) \in \Gamma''$   by the property of $\Lambda_1$; by \cref{cor slabmapping} applied at $(y_k', \bj\bi_{+})$, we have $(f_{\bj})^{-1}( B(z_k', \delta) ) \subset y_k' + e^{n\eps''} Y''$, and each $e^{n\eps''} Y''$ is covered by at most $C_d e^{d n \eps''}$ translations of $Y''$. Therefore $(f_{\bj})^{-1}( A \cap \widetilde\Lambda)$ is covered by at most $C_d e^{d n \eps''} N_{\delta}(A)$ translations of $Y''$, and hence
		\aryst
		N_{\delta}(A) \geq C_d^{-1} e^{- d n \eps''} (\rho/\delta)^{m-1-\eps} \geq (\rho/\delta)^{m-1-2\eps},
		\earyst
		where the last inequality uses $\rho/\delta > \delta^{-\eps'} \geq e^{(|\chi_m| - \eps) n \eps'}$ (from the assumption $\rho > \delta^{1 - \eps'}$) and $d \eps'' \leq (|\chi_m| - \eps) \eps \eps'$; this we add to the requirement on $\eps''$ in \cref{lem conditionalentropy}. Since $\eps$ is a free parameter, replacing $\eps$ by $2\eps$ gives the claimed regularity. The proof is completed.

		\section{Dimension Increment: Proof of Proposition \ref{prop diminc}} \label{sec proofofprop diminc}
		
		In this section, we prove \cref{prop diminc}. Our proof is inspired by that of \cite[Theorem 4.11]{wuProjectionTheoremsCountably2025}.
		
		\noindent\textit{Overview of the proof.} We assume to the contrary that $D_{d_0 + 1} > D_{m+1}$ and denote as before by $m_0$ the largest number such that $D_{m+1} = D_{m_0 +1}$.
		We will reach a contradiction by counting at a fine scale $0 < r_n \ll 1$.  
		
		More precisely, on one hand, the exact dimensionality of the projected measures $\nu_{\bj_{+}} = P^{V_{m+1}(\bj_{+})}_{*} \mu$ 
		yields the upper bound  $$N_{r_n}(P^{V_{m+1}(\bj_{+})} Y) < 2^{(D_{m+1} + \eps_0) n'}$$ for a large part $Y$ of a well-chosen box and for $\nu_{+}$-almost every $\bj_{+}$ (see \eqref{eq projYupper}). On the other hand, two ingredients give a stronger lower bound. The slab-filling property (\cref{prop slabfilling}) provides $(m-1)$-dimensional regularity inside the small boxes; the spreading of the projected fiberwise measures (\cref{prop fiberwisedim}) provides the input of \cref{thm Wu}/\cref{lem:ball-sumset} along the line $J_m (x, \bj_{+})^{\perp}$. Combining them, Wu's entropy increment yields the stronger lower bound $$N_{r_n}(P^{V_{m+1}(\bj_{+})} Y) > 2^{(D_{m+1} + 2 \eps_0) n'}$$ (see \eqref{eq coveringofYlowerbound}), a contradiction.
		
		\smallskip
		
		\noindent\textit{The input from Wu.}
		A crucial ingredient in our proof is  \cite[Theorem 4.7]{wuProjectionTheoremsCountably2025}. We state a slightly generalized version here.%
		
		\begin{thm}[entropy increase theorem] \label{thm Wu}
			Given $0 < \gamma,{ c_*} < 1$ and $l \in \Z_{+}$, there exists $\delta=\delta(\gamma,l) > 0$ such that the following holds for all   $n \gg_{\gamma, l, { c_*}} 1$. Let $A \subset [0, 1] \cap 2^{-n}\Z$, $\eta \in \cP([0, 1] \cap 2^{-n} \Z)$. Suppose that $\eta(A) \geq { c_*}$ and for each $x \in A$,
			\aryst
			| \{ 1 \leq k \leq n :  \eta(\cD_k(x)) \geq 2 \eta( \cD_{k+l}(x) )   \} | \geq (1 - \gamma / 2) n.
			\earyst
			Let $B \subset [0, 1] \cap 2^{-n} \Z$, and for each $a \in A$, let $B_a \subset B$. We have the following two cases.
			
			\begin{enumerate}
				\item 
				Suppose that $|B| \leq 2^{n (1 - \gamma)}$ and $| B_a | \geq | B |^{1 - \delta}$ for all $a \in A$. Then we have
				\aryst
				\Big  | \bigcup_{a \in A} (a + B_a)  \Big | \geq | B |^{1 + \delta}.
				\earyst
				\item Suppose that $|B|<2^{\delta n}$ and $B_a$ are non-empty for all $a\in A$. Then we have
				\aryst
				\Big  | \bigcup_{a \in A} (a + B_a)  \Big | \geq   2^{\delta n}.
				\earyst
			\end{enumerate}
		\end{thm}
		\begin{rema}
			The first case is \cite[Theorem 4.7]{wuProjectionTheoremsCountably2025} for ${ c_*}=1/2$. For general $0<{ c_*}<1$, this can be deduced by the same proof, noting that $1/2$ is only used in the last displayed equation in the proof of \cite[Theorem 4.7]{wuProjectionTheoremsCountably2025}; in particular, $\delta$ does not depend on ${ c_*}$, and ${ c_*}$ only enters through the required largeness of $n$. The second case handles $D_{m+1}=m-1$.
			
		\end{rema}
		\begin{proof}[Proof of the second case]
			By the spreading property, for any $x\in A$, we have $\eta(\{x\})\leq 2^{-(1-\gamma/2)n/l+2}$. 
			Hence, we have $|A|\geq { c_*} 2^{(1-\gamma/2)n/l-2}$. For each $a\in A$, take $b_a\in B_a\subset B$. Then the number of $a$ such that $a+b_a=x$ for some fixed $x$ is less than $|B| \leq 2^{\delta n}$. Hence, as long as $\delta\in (0,\frac{1-\gamma/2}{4l})$ and $n>\frac{2l}{1-\gamma/2}(2-\log_2 c_*)$, we obtain
			\[\Big  | \bigcup_{a \in A} (a + B_a)  \Big | \geq |\cup_{a\in A}(a+b_a)|\geq |A|/|B|\geq 2^{\delta n}.\] 
		\end{proof}
		The actual version we need is a continuous version with a covering number instead of a counting number; see \cref{lem:ball-sumset}.
		
		\smallskip

		\smallskip
		\noindent\textit{Scales.} Three scales govern the argument:
		\begin{equation}\label{eq scales}
			2^{- C_0 n} \ \gg \ \rho_n := 2^{- (C_0+1) n} \ \gg \ r_n := \rho_n 2^{- n'} = 2^{- (C_0+1+c) n}, \qquad \mbox{where } n' := c n.
		\end{equation}  
		{{The exponent $c := \frac{1}{2}\frac{\chi_m - \chi_{m+1}}{\chi_m - \chi_{m_0+1}} \in (0, 1)$ depends only on the Lyapunov exponents; it is chosen so that in changing projection directions, \eqref{eq B'jandBj} holds.}}
		We may and do assume that $c n \in \N$, which causes no loss of generality. 
		
		The coarse scale $2^{-C_0 n}$ is the long side of the master box $Z(\bi_{+}, n)$ \eqref{eqn:Z master box} and the side length of the atoms $\cD_{m+1, \bi_{+}, C_0 n}$; this is where the fiberwise structure of $\mu$ lives. The intermediate scale $\rho_n$ is the size of the small boxes $Z_0(\bi_{+}, n)$ \eqref{eqn:Z0 smaller box} into which the argument is localized, and after the rescaling $\Dil_{2^{(C_0+1)n}}$, it becomes $1$. The fine scale $r_n$ is the working scale: after the rescaling, it becomes the dyadic scale $2^{-n'}$ at which all covering numbers are evaluated and \cref{thm Wu}/\cref{lem:ball-sumset} is applied.

		\smallskip
		
		\noindent\textit{Parameters.}
		The parameters are chosen in the following order.
		First, choose a sufficiently large integer $C_0=C_0(\cF)$ such that
		\[
		(C_0+1)(\chi_{m-1}-\chi_m)>c|\chi_{m-1}|,
		\qquad
		(C_0+1)\alpha>c(10-\alpha).
		\]
		Next, choose
		\[
		c_0:=\frac{1}{4C_d},
		\qquad
		{ c_2:=\frac{1}{100\,C_d^{3}}},
		\qquad
		{ c_3:=\frac{c_2^{2}}{10\,C_d}=\frac{1}{10^{5}C_d^{7}}}.
		\]
		{Here and below $C_d\ge1$ denotes a constant depending only on $d$, chosen larger than the multiplicity constant appearing in \cref{lem besicovich}.}
		Set $\beta:=D_{m+1}-(m-1)$ and choose
		\[
		\eta_0\ll_{\cF,d}
		\min\bigl\{(c_0c_2)^4,\,1-\beta\bigr\}.
		\]
		With these parameters fixed, let
		\[
		l=l(\cF,\eta_0),
		\qquad
		\delta=\delta(2l,2\eta_0)>0.
		\]
		Choose $\eps_0>0$ sufficiently small that
		\begin{align}
			\eps_0\ll_{\cF,d,\eta_0}
			\min\Biggl\{
			& \delta,\ 2\eta_0^{1/4},\
			\frac{c}{C_0+1+c},
			\frac{(C_0+1)(\chi_{m-1}-\chi_m)-c|\chi_{m-1}|}
			{2(C_0+1)+c},
			\\
			& { 1-\beta},\ {\eta_0},\ \frac{c_2}{2},\ { c_3},\ {\delta\max\{\beta,\,\frac{\beta}{5-3\delta},\,\frac{1}{10}\}}
			\Biggr\}.
		\end{align}
		Finally, set
		\[
		\eps_1:=
		\frac{c}{100(C_0+1)(d+1)}\,\eps_0,
		\qquad
		\eta_1:=\frac{\eps_1\eta_0}{2C_d},
		\]
		choose $r_0=r_0(\eps_1,\eta_1)$, and take
		\[
		n\gg_{\cF,\eps_0,\eta_0,c_0}1.
		\]
		
		\smallskip

		\noindent\textit{The set $S$.}     
		To simplify notations, given a subset $\Gamma \subset \Lambda \times \Sigma_{+} $ with positive $(\mu \otimes \nu_{+})$-measure, for each $\bi_{+} \in \Sigma_{+}$, we denote 
		\ary
		\label{eqn:fiber set}
		\Gamma(\bi_{+}):= \{ x \in \Lambda : (x, \bi_{+}) \in \Gamma \}.
		\eary
		For $\bi_{+} \in \Sigma_{+}$, we denote 
		\aryst
		\nu_{\bi_{+}}:= (P^{V_{m+1}(\bi_{+})})_{*} \mu.
		\earyst
		Recall that $\nu_{\bi_{+}}$ is exact dimensional with $\dim \nu_{\bi_{+}} = D_{m+1}$ for $\nu_{+}$-almost every $\bi_{+}$. By Egorov's theorem, there exist $r_0 = r_0(\eps_1, \eta_1) > 0$ and a Borel subset $S \subset \Lambda \times \Sigma_{+}$ with $(\mu \otimes \nu_{+})(S) > 1 - \eta_1$ such that for every $(x, \bi_{+}) \in S$ and every $r \leq r_0$, we have
		\ary \label{eq Sregularity}
		\nu_{\bi_{+}} \big ( B( P^{V_{m+1}(\bi_{+})} x, r ) \big ) \in   [ r^{D_{m+1} + \eps_1}, r^{D_{m+1} - \eps_1} ].
		\eary

		\smallskip
		
		\subsection{Covering lemmas}\label{subsec covering lemmas}
		
		We collect here the counting and linear algebra lemmas used throughout the proof of \cref{prop diminc}: \Cref{lem:mea-counting} converts pointwise regularity of a measure into covering bounds, and \cref{lem numbercoverofprojection} compares the covering number of a set with that of a projection and with those of the fibers. Another covering lemma used in this section, \cref{lem:comparison}, is postponed to the appendix.
		
		\begin{lemma}\label{lem:mea-counting}
			Let $\nu$ be a Borel probability measure on $\R^{d}$, $\delta_*, \eps > 0$ and $r_* > 0$. Let $Y$ be a subset of points $y$ satisfying $\nu(B(y, r)) \in   [r^{\delta_*  + \eps}, r^{\delta_* -\eps}]$ for all $r \leq r_*$. Then there is a constant $C' = C'(\delta_*) > 0$ such that for every $x\in \R^d$ and every $r_2 < r_1 \leq r_* / 4$, we have 
			\aryst
			N_{r_2}(Y\cap B(x,r_1))\leq C' \, r_1^{\delta_*-\eps} r_2^{-(\delta_*+\eps)}.
			\earyst  	
		\end{lemma}
		\begin{proof}
			We may assume that $Y \cap B(x, r_1) \neq \emptyset$ and take $y_0 \in Y \cap B(x, r_1)$. Let $\{ y_i \}_{i = 1}^{N}$ be a maximal $r_2$-separated subset of $Y \cap B(x, r_1)$, so that $N_{r_2}(Y \cap B(x, r_1)) \leq N$. The balls $B(y_i, r_2/2)$ are disjoint and contained in $B(x, 2 r_1) \subset B(y_0, 4 r_1)$. Hence
			\aryst
			N \cdot (r_2/2)^{\delta_*+ \eps} \leq \sum_{i = 1}^{N} \nu( B(y_i, r_2/2) ) \leq \nu( B(y_0, 4 r_1) ) \leq (4 r_1)^{\delta_*- \eps},
			\earyst
			and the claim follows with $C' = 2^{\delta_*+ \eps} 4^{\delta_*- \eps}$.
		\end{proof}  
		
		The following basic lemma allows one to add or quotient directions under a projection.
		\begin{lemma}\label{lem numbercoverofprojection}
			There exists $c_d>0$ such that the following holds. Let $P^U$ be an orthogonal projection of Euclidean space $V=R^d$ and $T$ a bounded subset of $V$ with $S\subset T$. Then, for any $r>0$, we have
			\ary\label{equ:number-cover-of-projection} N_r(P^US)\leq c_d\frac{N_r(T)}{\min_{y\in S}N_r(T\cap (y+(P^U)^{-1}(B^{U^\perp}(r))))}. \eary
		\end{lemma}
		\begin{proof}
			Set
			\aryst
			M := \min_{y\in S}N_r \big (T\cap (y+(P^U)^{-1}(B^{U^\perp}(r))) \big ),
			\earyst
			and let $\{ B(x_i, r) \}_{i = 1}^{N}$ be an optimal covering of $T$, such that $N = N_r(T)$. Let $\{ z_k \}_{k = 1}^{K}$ be a maximal $r$-separated subset of $P^U S$; by maximality, we have $P^U S \subset \bigcup_{k = 1}^{K} B^{U^\perp}(z_k, r)$, hence
			\aryst
			N_r(P^U S) \leq K.
			\earyst
			Choose $s_k \in S$ with $P^U s_k = z_k$ and set $F_k := s_k + (P^U)^{-1}(B^{U^\perp}(r)) = (P^U)^{-1}(B^{U^\perp}(z_k, r))$. For each $k$, the balls $B(x_i, r)$ that meet $T \cap F_k$ cover $T \cap F_k$, so there are at least $N_r(T \cap F_k) \geq M$ of them. On the other hand, if $B(x_i, r)$ meets $F_k$, then $z_k \in B^{U^\perp}(P^U x_i, 2r)$; since the points $z_k$ are $r$-separated, the balls $B^{U^\perp}(z_k, r/2)$ are disjoint and contained in $B^{U^\perp}(P^U x_i, 5r/2)$, hence for each $i$ there are at most $5^{d}$ such indices $k$. Double counting the pairs $(i, k)$ such that $B(x_i, r)$ meets $T \cap F_k$ gives
			\aryst
			K M \leq 5^{d} N,
			\earyst
			and the claim follows with $c_d = 5^{d}$.
		\end{proof} 
		
		The next lemma compares projections along nearby hyperplanes; it is used in Subsection \ref{subsec concludingtheproof}.
		\begin{lemma}\label{lem projhyperplanes}
			For any two hyperplanes $H, H' \subset \R^d$, we have
			\aryst
			\| P^{H'} P_H \| \leq 2 \bd(H, H').
			\earyst
		\end{lemma}
		\begin{proof}
			Let $u, u'$ be unit vectors in $H^{\perp}$ and $(H')^{\perp}$. For every unit vector $v$ we have $P_H v = v - \langle v, u \rangle u$, and hence
			\aryst
			P^{H'} P_H v = \langle P_H v, u' \rangle u' = \langle v, u' - \langle u', u \rangle u \rangle \, u',
			\earyst
			which gives $\| P^{H'} P_H \| \leq \| u' - \langle u', u \rangle u \| = \sin \measuredangle(u, u')$. For hyperplanes, this equals $\bd(H, H')$: indeed, by the definition of the Grassmannian distance, $\bd(H, H') = \max_{V^1 \subset H} \bd(V^1, H')$, and $\bd(V^1, H') = | \langle v, u' \rangle |$ for a unit vector $v \in V^1$, whose maximum over $v \in H$ is $\| P_H u' \| = \sin \measuredangle(u, u')$.
		\end{proof}
		
		\subsection{Preparing a random set}  \label{subsec preparingarandomset}
		
		\
		
		\smallskip

		\noindent{\textbf{Step 1}(H\"older continuity).}
		Let $\alpha \in (0, 1)$ be given by \cref{prop Holdermap}. 
		We fix some $\Lambda_1 \subset \Lambda$ given by \cref{prop Holdermap} with $\mu(\Lambda_1) > 1 - \eta_0$ such that  $x \mapsto U^k(x)$ is  $C^{\alpha}$ restricted to $\Lambda_1$ for every $k\in \{1,\cdots,d\}$.

		\smallskip
		\smallskip

		\noindent{\textbf{Step 2} (Slab-Filling property).}
		We apply \cref{prop slabfilling} with $(\rho, \delta, \eps', \eps, \eta, \eta_0) = (\rho_n, r_n, \eps_0, \eps_0, \eps_0, \eta_0)$, with its time parameter replaced by an integer $\tilde n$ chosen as follows.
		
		One constraint on $\eps_0$ is that $(C_0 + 1)(\chi_{m-1} - \chi_m - 2\eps_0) > c(|\chi_{m-1}| + \eps_0)$, which gives
		\aryst
		\alpha_0 \ := \ \frac{(C_0+1+c) \log 2}{|\chi_m| - \eps_0} \ < \ \alpha_1 \ := \ \frac{(C_0+1) \log 2}{|\chi_{m-1}| + \eps_0} .
		\earyst
		Choose an integer $\tilde n$ with $\alpha_0 n < \tilde n < \alpha_1 n$; such an integer exists for $n \gg_{\cF, C_0, \eps_0} 1$. Then, since $\chi_m < \chi_{m-1} < 0$,
		\aryst
		e^{\tilde n(\chi_m + \eps_0)} \ < \ e^{\alpha_0 n (\chi_m + \eps_0)} \ = \ 2^{-(C_0+1+c)n} \ = \ r_n, \\
		\qquad
		e^{\tilde n(\chi_{m-1} - \eps_0)} \ > \ e^{\alpha_1 n (\chi_{m-1} - \eps_0)} \ = \ 2^{-(C_0+1)n} \ = \ \rho_n ,
		\earyst
		so the two scale hypotheses of \cref{prop slabfilling} hold, and its largeness requirement is met since $\tilde n \asymp n$. The remaining two constraints in the parameter list, $(C_0+1)\alpha > c(10-\alpha)$ and $\eps_0(C_0+1+c) < c$, say respectively that $r_n^{1-\alpha/10} > \rho_n$ and $r_n^{1-\eps_0} < \rho_n$.
		
		Applying the proposition and by letting $n \gg_{\cF, \eps_0} 1$,
		there is a Borel subset $\Lambda_2(n) \subset \Lambda$  such that $\mu(  \Lambda_2(n)  ) > 1 - \eta_0$ and for every $x \in \Lambda_2(n)$, for subset
		\aryst
		W(x, n) :=    S( 0, ( U^{m - 1}(x), \R^{d} ),  (  \rho_n, r_n   ) )
		\earyst 
		the measure $(\Dil_{2^{(C_0+1)n}})_{*}   \mu^{x+W(x, n)}  $ is $( m-1 ,  2^{-n'}, \eps_0, \eps_0   )$-regular.
		
		\smallskip
		\smallskip
		
		\noindent{\textbf{Step 3} (Lower regularity for boxes for $\bi_{+}$).}
		For each $\bi_{+} \in \Sigma_+$,  we denote 
		\begin{align}
			\label{eqn:Z master box}
			Z( \bi_{+}, n) &:= S(0,  (V_{ m_0 +1}(\bi_{+}), \R^{d} ), (2^{ - C_0 n },  \rho_n) ), \\
			\label{eqn:Z0 smaller box}
			Z_0( \bi_{+}, n) &:=  S(0,  (V_{m+1}(\bi_{+}), \R^{d} ), (  \rho_n,   \rho_n ) ).
		\end{align}
		By \cref{lem regularmeasureeverywhere} applied with $(C, n)$ replaced by $\big ( \frac{C_0 + 1}{c},   n' \big )$ (note that $2^{- \frac{C_0 + 1}{c} n'} = \rho_n$, so $Z_0(\bi_{+}, n)$ is exactly the box appearing in \cref{lem regularmeasureeverywhere}), for $n \gg_{\cF,  C_0, \eps_0,  \eta_1 } 1$,  there is a subset $\Gamma'_3(n) \subset \Lambda \times \Sigma_{+}$ with $(\mu \otimes  \nu_{+})(\Gamma'_3(n)) > 1 - \eta_1$  such that for each $(x, \bi_{+}) \in \Gamma'_3(n)$,    the following is true:
		the measure 
		\aryst
		( \Dil_{2^{(C_0+1)n}} )_{*}\Big (    P^{V_{m+1}(\bi_{+})}_{*}  \mu^{x + Z_0(\bi_{+}, n )} \Big )
		\earyst
		is $(D_{m+1}, 2^{-n'}, \eps_0, \eps_0 )$-lower regular.
		
		Let $\Gamma_3(n)$ be the subset of those $(x, \bi_{+}) \in \Gamma'_3(n)$ such that \aryst \mu^{x + Z(\bi_{+}, n)} \big ( \Gamma'_3(n)(\bi_{+}) \big ) > 1 - \eps_1, \earyst where $\Gamma_3'(n)(\bi_+)$ is the fiber subset defined as in \eqref{eqn:fiber set} .   By \cref{lem besicovich} and the definition of $\eta_1$, we know $(\mu \otimes \nu_{+})(\Gamma_3(n)) > 1 - \eta_0$. 
		
		\noindent{\textbf{Step 4} (Upper regularity for boxes).}
		Let $S\subset \Lambda\times \Sigma_+$ be the subset with $(\mu\otimes \nu_+)(S)>1-\eta_1$ which satisfies \eqref{eq Sregularity}. We define 
		\begin{equation} \label{eq Gamma34}
			\Gamma_4(n) := \big \{ (x, \bi_{+}, \bj_{+}) \in \Lambda \times \Sigma_{+} \times \Sigma_{+} : \mu^{x + Z(\bi_{+}, n)}( S(\bi_{+}) \cap S(\bj_{+}) ) > 1 - 2\eps_0 \big \}.
		\end{equation}
		We claim that 
		\ary \label{eq gamma34lower}
		(\mu \otimes \zeta_n)(\Gamma_4(n)) > 1 - \eta_0,
		\eary
		where $\zeta_n$ is the self-coupling of $\nu_+$ defined in \eqref{eqn:self coupling}.
		
		Indeed, fix $(\bi_{+}, \bj_{+})$ and apply Besicovitch's covering lemma (\cref{lem besicovich}) with $W=Z(\bi_+,n)$ and $A=\Gamma_4(n)(\bi_+,\bj_+)$ and then integrate over $(\bi_+,\bj_+)$ %
		\aryst
		(\mu \otimes \zeta_n)( \Gamma_4(n)^{c} ) \leq 2 C_d \eps_0^{-1} \eta_1 < \eta_0
		\earyst
		by the definition of $\eta_1$ and $\eps_1 \leq \eps_0$.

		\smallskip

		\smallskip

		\noindent{\textbf{Step 5} (Angle control and spreading property).}
		By \cref{prop fiberwisedim} (applied with $\eps_1$ in place of $\eps$ there),
		there exists an integer $l = l(\cF,  c\eta_0) > 0$ and a Borel subset $  \Gamma_5(n)  \subset \Lambda \times \Sigma_{+} \times \Sigma_{+}$ with $ (\mu \otimes \zeta_n)   ( \Gamma_5(n) ) > 1 - c\eta_0$  
		such that    the following is true:
		\enmt
		\item for every $(x, \bi_{+}, \bj_{+}) \in   \Gamma_5(n)$, we have 
		\begin{align}
			\label{equ:angle-m0m}
			&d(V_{m_0+1}(\bi_{+}), V_{m+1}(\bj_{+})) \leq  2^{- n +   \eps_1 n}\\
			\label{equ:angle-mm}
			&d(V_{m+1}(\bi_{+}), V_{m+1}(\bj_{+})) \leq  2^{- 2c n +   \eps_1 n}\\
			\label{eqn:UV lower bound}
			&\measuredangle( V_{m+1}(\bi_{+}), U^{m-1}(x) ), \,\,\measuredangle( V_{m+1}(\bj_{+}), U^{m-1}(x) ) \geq 2^{-   \eps_1 n},
		\end{align}
		\item 
		for every $(x, \bi_{+}, \bj_{+}) \in   \Gamma_5(n)$, the probability measure
		$$(\Dil_{2^{(C_0+1)n}})_*(P^{V_{m+1}(\bj_{+}) + U^{m-1}(x)})_* ( \mu_{x, \bi_{+}} )^{\cD_{m+1, \bi_{+}, C_0 n}(x) }$$ is $(n, l, c\eta_0)$-ball spreading and the diameter of the support is less than $2^{\eps_1 n}$. From the definition of ball spreading, we actually get that the above measure is $(n',l,\eta_0)$-ball spreading.
		\eenmt

		\smallskip
		
		To finish our construction,  we  define $\widetilde\Gamma$ by
		\aryst
		\widetilde\Gamma:=  \{ (x, \bi_{+}, \bj_{+}) \in   \Gamma_4(n)\cap \Gamma_5(n) : x \in \Lambda_1\cap \Lambda_2(n), (x,\bi_+)\in\Gamma_3(n)\}.
		\earyst
		It is clear that we have
		\ary \label{eq wildegammalower}
		(\mu \otimes \zeta_n)(\widetilde\Gamma) \geq 1 - 5\eta_0.
		\eary

		\subsection{Finding a good point} \label{subsec findingagoodpoint}

		For each $\bi_{+}, \bj_{+} \in \Sigma_{+}$, we define
		\aryst
		\Lambda_{\bi_{+}, \bj_{+}}:=  \{ x \in \Lambda : (x, \bi_{+}, \bj_{+}) \in \widetilde\Gamma \}.
		\earyst
		Let  $\widehat\Sigma$ be the set of  $(\bi_{+}, \bj_{+})$ such that 
		\ary \label{eq lambdaijlower}
		\mu( \Lambda_{\bi_{+}, \bj_{+}}  ) > 1 -  ( \eta' )^2,
		\eary
		where $$\eta':=\eta_0^{1/4}.$$
		By \eqref{eq wildegammalower}, we have $\zeta_n( \widehat \Sigma ) > 1 - 5  ( \eta' )^2$.
		
		Fix some $(\bi_{+}, \bj_{+}) \in \widehat\Sigma$ for the rest of the section.  
		
		\smallskip
		
		\noindent{\textbf{Stage I}.}
		{{We find a point $x$ in $\Lambda_{\bi_+,\bj_+}$ with nice properties described in \cref{lem J0isalmostfullinJ} and \cref{lem Yjisstructured}.}}

		\begin{lemma}[Good base point] \label{lem goodbasepoint}
			Let $c_0 = c_0(d) \in (0, 1)$ be small enough. Then there exists $x \in \Lambda_{\bi_{+}, \bj_{+}}$ such that
			\enmt
			\item \label{itm basedoubling} {\rm (Doubling)}
			\ary \label{eq localdensitybetweennearbyscales0} \label{eq localdensitybetweennearbyscales}
			\mu( x + Z(\bi_{+}, n) ) > %
			c_0\mu( x + 8 Z(\bi_{+}, n)) ;
			\eary
			\item \label{itm basedensity} {\rm (Density)}
			\ary \label{eq localdensitybetweennearbyscales1} 
			\mu(  \Lambda_{\bi_{+}, \bj_{+}}   \cap (x + Z(\bi_{+}, n))  ) > (1 - \eta' ) \mu( x + Z(\bi_{+}, n) ). 
			\eary
			\eenmt
		\end{lemma}
		\begin{proof}
			Let $\Lambda^{(1)}_{\bi_{+}, \bj_{+}} \subset \Lambda_{\bi_{+}, \bj_{+}}$ be the subset of $x$ satisfying \eqref{eq localdensitybetweennearbyscales0}. By \eqref{eq lambdaijlower} and \cref{lem:comparison} (applied with $W = Z(\bi_{+}, n)$ and $W' = 8 Z(\bi_{+}, n)$), we see that $$ \mu( \Lambda^{(1)}_{\bi_{+}, \bj_{+}} ) > 1 - C_d c_0 - ( \eta' )^2 .$$
			Next, let $\Lambda^{(2)}_{\bi_{+}, \bj_{+}} \subset \Lambda_{\bi_{+}, \bj_{+}}$ be the subset of $x$ satisfying \eqref{eq localdensitybetweennearbyscales1}. By \eqref{eq lambdaijlower} and Besicovitch's covering lemma (\cref{lem besicovich}), we have 
			\aryst
			\mu( \Lambda^{(2)}_{\bi_{+}, \bj_{+}} ) > 1 -  C_d \eta'.
			\earyst
			By letting $c_0, \eta_0 \ll_d 1$, we see that   $  \Lambda^{(1)}_{\bi_{+}, \bj_{+}} \cap  \Lambda^{(2)}_{\bi_{+}, \bj_{+}}$ is not empty since
			\aryst
			\mu(   \Lambda^{(1)}_{\bi_{+}, \bj_{+}} \cap  \Lambda^{(2)}_{\bi_{+}, \bj_{+}} ) > 1 - C_d c_0 - ( \eta' )^2 -  C_d \eta'  > 0.
			\earyst
			Any point in this intersection satisfies the conclusions.
		\end{proof}  
		
		From now on, we fix a point $x$ given by   \cref{lem goodbasepoint}.
		
		\smallskip 
		\smallskip
		
		\noindent{\textbf{Stage II}.}  
		The goal of this stage is to remove from $x + Z(\bi_{+}, n)$ the points failing the lower regularity of Step 3 or the uniform exact dimensionality of \eqref{eq Sregularity}. The surviving set $Y$ has nearly full measure, and \cref{lem:mea-counting} then gives the upper bound \eqref{eq projYupper}.
		
		We set
		\aryst
		Y^{(3)} := (x + Z(\bi_{+}, n)) \cap \Gamma'_3(n)(\bi_{+}).
		\earyst
		By Step 3, we have
		\ary \label{eq Y33density}
		\mu^{x + Z(\bi_{+}, n)}(Y^{(3)})>1-\eps_1   \geq 1 - \eta'.
		\eary

		We set
		\aryst 
		Y^{(4)} := (x + Z(\bi_{+}, n)) \cap S(\bi_{+})\cap S(\bj_{+}).
		\earyst
		By $(x, \bi_{+}, \bj_{+}) \in   \Gamma_4(n) $, we have 
		\ary \label{eq Y34density}
		\ \mu^{x + Z(\bi_{+}, n)}(Y^{(4)}) > 1 - 2\eps_0   \geq 1 - \eta'.
		\eary 
		Due to \eqref{equ:angle-m0m}, both $P^{V_{m+1}(\bi_{+})} ( Z(\bi_{+}, n) )$ and $P^{V_{m+1}(\bj_{+})} ( Z(\bi_{+}, n) )$ have diameter $ \leq 4\rho_n 2^{\eps_1 n}$; note also that $ \rho_n 2^{\epsilon_1 n} \leq r_0 / 4$ for $n$ large.  Applying \cref{lem:mea-counting} with $\nu = \nu_{\bi_{+}}$, $Y = P^{V_{m+1}(\bi_{+})}(Y^{(4)})$, $\delta_*= D_{m+1}$, $\eps = \eps_1$, $r_1 = 4\rho_n 2^{\eps_1 n}$ and $r_2 = r_n$, we obtain
		\ary  \label{eq coveringofY4bi} 
		N_{r_n}( P^{V_{m+1}(\bi_{+})}(Y^{(4)}) ) \leq C' r_1^{D_{m+1} - \eps_1} r_n^{- (D_{m+1} + \eps_1)} <  2^{(D_{m+1} + \eps_0) n'}, 
		\eary
		where in the second inequality we used the definitions of $\eps_1$ and $n \gg 1$ to absorb the error terms. The same argument with $\bj_{+}$ in place of $\bi_{+}$ gives
		\ary \label{eq coveringofY4}
		N_{r_n}( P^{V_{m+1}(\bj_{+})}(Y^{(4)}) ) <  2^{(D_{m+1} + \eps_0)   n'}.
		\eary

		We set
		\ary \label{eq defY}
		Y :=     Y^{(3)} \cap Y^{(4)} \cap \Lambda_{\bi_{+}, \bj_{+}}.
		\eary
		In particular, we have $Y \subset Y^{(3)} \subset   x + Z(\bi_{+}, n )$. Thus, we deduce from \eqref{eq Y33density}, \eqref{eq Y34density} and \eqref{eq localdensitybetweennearbyscales1}  
		\begin{align}
			\mu^{x + Z(\bi_{+}, n)}(Y)  \geq&  \mu^{x + Z(\bi_{+}, n)}( \Lambda_{\bi_{+}, \bj_{+}} )   - 	\mu^{x + Z(\bi_{+}, n)}((Y^{(3)})^c) - 	\mu^{x + Z(\bi_{+}, n)}((Y^{(4)})^c)   \nonumber  \\
			\geq& 1 - 3 \eta'.  \label{eq lowerboundforY}
		\end{align}
		By $Y \subset Y^{(4)}$ and \eqref{eq coveringofY4}, it is clear that 
		\ary \label{eq projYupper}
		N_{ r_n}( P^{V_{m+1}(\bj_{+})}(Y) ) <  2^{(D_{m+1} + \eps_0) n'}.
		\eary

		\smallskip 
		\smallskip
		
		\noindent{\textbf{Stage III}.}  
		The goal of this stage is to partition $Y$ into pieces $Y_j, \,j \in \cJ$ in the $\rho_n$-boxes $y_j + Z_0(\bi_{+}, n)$ and to single out the good pieces $Y_j (j  \in \cJ_0)$. The two lemmas below summarize the structure needed for the lower bound in Subsection \ref{subsec concludingtheproof}.
		
		Let us fix a finite subset of points $\{  y_j \}_{j \in \cJ} \subset Y$ so that 
		$\{ y_j + Z_0(\bi_{+}, n) \}_{j \in \cJ}$ forms a covering of $Y$ with multiplicity $O_d(1)$.
		
		We define
		\aryst
		Y_j:= Y \cap (y_j + Z_0(\bi_{+}, n) ),  \quad j \in \cJ.
		\earyst
		Let $c_2 \ll_d c_0^2$ be a small parameter that satisfies \cref{lem J0isalmostfullinJ}. We say an index $j\in \cJ$ is {\it $Y$-dense} if $\mu^{y_j + Z_0(\bi_{+}, n)} ( Y_j ) > c_2$, and {\it doubling} if $\mu( y_j + Z_0(\bi_{+}, n) ) > c_2 \mu( y_j +  3 Z_0(\bi_{+}, n) )$. 
		The doubling condition is used to transfer mass from the enlarged box $y_j +  3 Z_0(\bi_{+}, n)$ back to $y_j + Z_0(\bi_{+}, n)$ in the proofs of \cref{lem J0isalmostfullinJ,lem Yjisstructured}. We set
		\begin{equation}\label{def;j0}
			\cJ_0:= \Big\{ j \in \cJ \ \Big| \ \mu^{y_j + Z_0(\bi_{+}, n)} (  Y_j ) >  c_2 \ \mbox{ and } \   \mu( y_j + Z_0(\bi_{+}, n) ) > c_2 \mu( y_j +  3 Z_0(\bi_{+}, n) )   \Big\}.
		\end{equation}

		The following lemmas summarize the main properties we need from $x$. Recall the notation we introduced in \eqref{eq J def}:
		\begin{equation}\label{eq Hmxbi+}
			J_m(x, \bi_{+}  )=  V_{m+1}(\bi_{+}) + U^{m-1}(x)\qquad \text{and}\qquad J_m(x, \bj_{+}  )=  V_{m+1}(\bj_{+}) + U^{m-1}(x).
		\end{equation}
		
		\cref{lem J0isalmostfullinJ} provides the input of Wu's theorem: a ball-spreading measure on the line $J_m(x, \bj_{+})^{\perp}$ charging the $r_n$-tubes around the projected centers of the good pieces with mass $> c_0/2$. \Cref{lem Yjisstructured} states that each good piece is structured: it has big fibers (item \eqref{itm Gj1}), and its images on the line $J_m(x, \bi_{+})^{\perp}$ are large but not excessively large (item \eqref{itm Gj2}).
		
		\begin{lemma} \label{lem J0isalmostfullinJ} 
			For every $c_2 \ll_d c_0^2$,   there is a probability measure $\zeta$ on $	J_m(x, \bj_{+}  )^{\perp}$
			such that    $(\Dil_{2^{(C_0+1)n}})_*  \zeta $ is   $( n', l+2 , 2\eta_0 )$-ball spreading with support of diameter less than $e^{\eps_0 n'}$, and 
			\begin{equation}
				\label{eqn:mass lower bound}
				\zeta( \bigcup_{j \in \cJ_0} \big (  B^{J_m(x, \bj_{+}  )^{\perp}}( r_n)  +    P^{ J_m(x, \bj_{+}  ) }   P_{J_m(x, \bi_{+}  )}   (y_j)  \big ) ) > c_0(d)/2.
			\end{equation}
		\end{lemma}
		
		\begin{lemma} \label{lem Yjisstructured}
			There is a collection of finite subsets $( G_j \subset Y_j  )_{j \in \cJ_0}$ such that:
			\enmt
			\item \label{itm Gj1} 	for every $j \in \cJ_0$,  for every $y \in G_j$, we have 
			\aryst
			N_{r_n}( Y_j \cap (y + 2W(x, n)) ) > 2^{(m-1 - \eps_0) n'},
			\earyst
			\item  \label{itm Gj2} for every $j \in \cJ_0$, we set
			\begin{equation}
				B_j :=  P^{ 	J_m(x, \bi_{+}  )}(G_j).
			\end{equation}
			Then  
			\begin{align}
				&N_{r_n}(B_j) > 2^{(D_{m+1} - m + 1 - 2 \eps_0)  n'}, \quad j \in \cJ_0,\\
				&N_{ r_n}( \cup_{j \in \cJ_0} B_j ) < 2^{(D_{m+1} - m + 1 + 3\eps_0)n'}.
			\end{align}
			\eenmt
		\end{lemma}

		\bigskip

		\subsection{The comparison at scale $r_n$} \label{subsec concludingtheproof}

		The proof of \cref{prop diminc} now reduces to a single comparison at scale $r_n$. Recall the upper bound given in \eqref{eq projYupper}:
		\aryst
		N_{ r_n}( P^{V_{m+1}(\bj_{+})}(Y) ) <  2^{(D_{m+1} + \eps_0)  n'}.
		\earyst
		We will show the opposite lower bound:
		\ary \label{eq coveringofYlowerbound}
		N_{ r_n}  \big (P^{V_{m+1}(\bj_+)} Y) > 2^{(D_{m+1} + 2 \eps_0) n'}.
		\eary
		This concludes the proof, since the lower bound contradicts the upper bound.

		We denote 
		\aryst
		W' :=P^{V_{m+1}(\bj_{+})} ( W(x, n) ).
		\earyst
		Let $( G_j )_{j \in \cJ_0}$ be the subsets given by \cref{lem Yjisstructured}. Given $j \in J_0$, we denote as in \cref{lem Yjisstructured}
		\aryst
		B_j:= P^{J_m(x, \bi_{+}  ) }(G_j).
		\earyst
		We also denote
		\aryst
		B'_j := P^{J_m(x, \bj_{+}  ) }(G_j), \qquad 
		E'_j :=  P^{V_{m+1}(\bj_{+})} (  G_j  ), \qquad 
		Y' _j := P^{V_{m+1}(\bj_{+})} (Y_j).
		\earyst

		\smallskip
		
		\noindent{(\textbf{Structure of $E'_j$}).}
		We have
		\aryst
		P^{J_m(x, \bj_{+}  ) } = P^{U'} P^{V_{m+1}(\bj_{+})}
		\earyst
		where
		\begin{equation}
			\label{eqn:U projection}
			U':= P^{V_{m+1}(\bj_{+})}( U^{m-1}(x) ) \in \Gr( V_{m+1}(\bj_{+})^{\perp}, m -1).
		\end{equation}
		Thus we have  
		\aryst
		B'_j =  P^{U'} ( E'_j ). 
		\earyst

		For any $e'\in E'_j$, by definition, there is $y \in G_j$ such that  $e' = P^{V_{m+1}(\bj_{+})} (y)$.
		By \cref{lem Yjisstructured}\eqref{itm Gj1},  
		\[ N_{ r_n} ((y  + 2 W(x, n) ) \cap Y_j)\geq 2^{(m-1 - \eps_0) n'}.\]
		By  Step 5(1)  in Subsection \ref{subsec preparingarandomset}, that is $\measuredangle(V_{m+1}(\bj_+),U^{m-1}(x))>2^{- \eps_1 n}$, we have 
		\[\diam((P^{V_{m+1}(\bj_{+})})^{-1}(B^{V_{m+1}(\bj_{+})^\perp}( r_n))\cap W(x,n))= O(2^{ \eps_1 n} r_n).  \] 
		Thus, we have
		\begin{equation}\label{eq coveringnumberoffiber}
			\begin{split}
				& N_ { r_n}((e' + 2 W' ) \cap  Y'_j )\geq N_ { r_n}( P^{V_{m+1}(\bj_{+})}( (y  + 2 W(x, n) ) \cap Y_j ))\\
				&  \geq 2^{ - d \eps_1 n - O(1) }  N_ { r_n }( (y  + 2 W(x, n) ) \cap Y_j ) \geq 2^{(m-1 -  2 \eps_0) n'}.	
			\end{split}
		\end{equation}
		Here, the first inequality follows from the inclusion $(e' + 2 W' ) \cap  Y'_j \supset P^{V_{m+1}(\bj_{+})}( (y  + 2 W(x, n) ) \cap Y_j )$. The second one follows from the fiber diameter bound displayed above: an $r_n$-covering of the projection pulls back to a covering of $(y + 2W(x,n)) \cap Y_j$ by $N_{r_n}$ of the projection many sets of diameter $O(2^{\eps_1 n} r_n)$, and each such set is covered by $(C_d 2^{\eps_1 n})^d$ balls of radius $r_n$. The last one uses \cref{lem Yjisstructured}\eqref{itm Gj1} and that $d \eps_1 n + O(1) \leq \eps_0 n'$ by the choice of $\eps_1$ and $n \gg 1$.

		\smallskip
		
		\noindent{(\textbf{Compare $B_j$ and $B'_j$}).} By direct computation, we obtain 
		\aryst
		P^{ J_m(x, \bj_{+}  ) }  =  Q     P^{ J_m(x, \bi_{+}  ) }   +   Q'
		\earyst
		where
		\begin{equation}
			\label{eqn:Q and Q'}
			Q  =  P^{  J_m(x, \bj_{+}  ) }  |_{ J_m(x, \bi_{+} )^{\perp}}, \qquad
			Q' =  P^{J_m(x, \bj_{+}  ) }P_{J_m(x, \bi_{+}  )}.
		\end{equation}
		The angle constraints among $V_{m+1}(\bj_{+})$, $V_{m+1}(\bi_{+})$ and $U^{m-1}(x)$ given in \eqref{equ:angle-mm} and \eqref{eqn:UV lower bound} give
		\begin{equation}
			\label{eqn:distance H planes}
			\bd( J_m(x, \bi_{+} ), J_m(x, \bj_{+} ) ) \lesssim   2^{- 2c n + 2\eps_1 n}\leq 2^{- c n-\eps_0n}.
		\end{equation}
		Hence, the map $Q$ is a linear isomorphism between $J_m(x, \bi_{+} )^{\perp}$ and $J_m(x, \bj_{+} )^{\perp}$ such that 
		\ary
		\label{eqn:Q norm}
		\| Q \|, \,\| Q^{-1} \| \leq 2.
		\eary 
		
		By \cref{lem projhyperplanes}, we have $\| P^{H'}P_H\|\leq 2\bd(H,H')$ for two hyperplanes $H, H'$ in $\R^d$, and we apply this result with $H= J_m(x, \bi_{+})$ and $H' = J_m(x, \bj_{+})$). Combining this with the estimates $\diam(Z_0(\bi_{+}, n)) \lesssim  \rho_n$ and \eqref{eqn:distance H planes}, 
		we obtain   
		\begin{equation}
			\label{eqn:Q' diam}
			\diam(   Q'(Z_0( \bi_{+}, n )) )< r_n
		\end{equation}
		for sufficiently large $n$.
		Since $G_j \subset Y_j \subset y_j + Z_0(\bi_{+}, n)$, we have  
		\ary \label{eq B'jandBj}
		d_{H}( B'_j, Q B_j + Q' (y_j) ) =  O(  r_n).
		\eary
		Here $d_H$ denotes the Hausdorff distance. 
		
		\smallskip
		
		\noindent{(\textbf{Dimension increment}).}
		We denote
		\aryst
		A := Q'( \{ y_j \}_{j \in J_0} ).
		\earyst
		By \cref{lem J0isalmostfullinJ}, there is a probability  measure $\zeta$ on $J_m(x, \bj_{+} )^{\perp}$ such that the measure $(\Dil_{2^{(C_0+1)n}})_*\zeta$ is $( n', l+2 , 2\eta_0 )$-ball spreading, and 
		\aryst
		\zeta(A + B^{J_{m}(x, \bj_{+})^{\perp}}( r_n)) >  c_0/2.
		\earyst 
		
		Recall that we set $\beta=D_{m+1}-(m-1)$. By \cref{lem Yjisstructured} and \eqref{eqn:Q norm}, if { $\beta>0$} we have
		\[ N_{r_n}(QB_j)>N_{r_n}(\cup_j QB_j)^{1-\delta}.  \]
		Case 1  of \cref{lem:ball-sumset} therefore applies for every $\beta>0$. If $\beta=0$, we are in case 2  of \cref{lem:ball-sumset}: item \eqref{itm Gj2} of \cref{lem Yjisstructured} and \eqref{eqn:Q norm} give $N_{r_n}(\cup QB_j)\leq { 2^{3\eps_0n'}}\leq 2^{\delta n'}$, where we use the constraint $3\eps_0<\delta$.
		
		Now we can apply \cref{lem:ball-sumset} (with $2\eta_0$ in place of $\gamma$, $c_0/2$ in place of the threshold $ c_*$, $n'$ in place of $n$, and $\eps_0=\eps$; the constant $\delta = \delta(2 l, 2\eta_0) > 0$ is given by this lemma), and we obtain 
		\aryst
		N_{r_n}(\bigcup_{j \in \cJ_0}  ( Q B_j + Q' (y_j)  ))\geq \max\{2^{(1  + \delta)(\beta -  2\eps_0)  n'},2^{\delta n'}\}> 2^{(\beta  + 5 \eps_0)  n'}
		\earyst
		provided that $\eps_0$ meets the constraint in the parameter list, $\eps_0\ll\delta\max\{\beta,1/10\}$.
		The conditions of \cref{lem:ball-sumset} are met as well: $\gamma<1/100$ and $\gamma\le c_*^2/400$ follow from the choice of $\eta_0$, and the last one, $0<\eps\le\gamma/200$, reads $\eps_0\le\eta_0/100$, likewise included in the choice of $\eps_0$.

		By \eqref{eq B'jandBj}, we see that 
		\aryst
		N_{ r_n} ( \bigcup_{j \in \cJ_0}  B'_j ) > 2^{(D_{m+1} - m + 1 + 5  \eps_0)  n' - C}.
		\earyst
		By \eqref{eq coveringnumberoffiber} and \eqref{equ:number-cover-of-projection} with $S=\cup_{j\in \cJ_0} E_j'$, $U=U'$ given in \eqref{eqn:U projection}, $T=P^{V_{m+1}(\bj_+)} Y,\ r= 2r_n $ and $(P^U)^{-1}B^{U^\perp}(r)\supset  2W'$, we see that 
		\aryst
		N_{ r_n} (P^{V_{m+1}(\bj_+)} Y)  > \frac{1}{c_d} N_{ 2r_n} (\cup_{j\in \cJ_0}B_j') \min_{e'\in \cup_{j\in \cJ_0} E_j'}N_{ 2r_n} ((e'+2W')\cap Y_j')\geq  2^{(D_{m+1} + 2 \eps_0)  n'}.
		\earyst
		Here the left-hand side uses $N_{r_n} \geq N_{2r_n}$. Converting the covering numbers between the scales $r_n$ and $2r_n$ (both for $\cup_j B_j'$ and for the fibers) costs a factor $2^{O(d)}$, which is absorbed into the $+5\eps_0 n'$ slack in the exponent.
		This implies \eqref{eq coveringofYlowerbound}.

		\smallskip
		
		To complete the proof, we still need to prove \cref{lem J0isalmostfullinJ} and \cref{lem Yjisstructured}. These will be the subjects of the rest of this section.

		\bigskip
		
		\subsection{A spreading measure charging the good pieces: proof of Lemma \ref{lem J0isalmostfullinJ}}

		We  set
		\aryst
		R:=   \bigcup_{j \in \cJ_0} (  y_j + Z_0(\bi_{+}, n) ).
		\earyst
		We first prove that
		\begin{equation}
			\label{eqn:Rc measure estimate}
			\mu^{x + Z(\bi_{+}, n)}( R^c )  \leq  \mu^{x+Z(\bi_+,n)}(Y^c)+2C_d c_0^{-1} c_2.
		\end{equation}
		All complements below are taken inside $x + Z(\bi_{+}, n)$. Since $R^{c} \subset Y^{c} \cup \bigcup_{j \in \cJ \setminus \cJ_{0}} Y_{j}$, we have
		\aryst
		\mu( R^{c} ) \leq \mu( Y^{c} ) + \sum_{j \in \cJ \setminus \cJ_{0}} \mu( Y_{j} ).
		\earyst
		For $j \in \cJ \setminus \cJ_{0}$, one of the two conditions in \eqref{def;j0} fails. Accordingly, we split $\cJ \setminus \cJ_{0} = \cJ_{1} \cup \cJ_{2}$, where 
		\aryst
		&&\cJ_{1} := \{ j \in \cJ \setminus \cJ_{0}: \mu( Y_{j} ) \leq c_{2} \mu( y_{j} + Z_{0}(\bi_{+}, n) ) \}, \quad
		\\
		&&\cJ_{2} := \{ j \in \cJ \setminus \cJ_{0}:\mu( y_{j} + Z_{0}(\bi_{+}, n) ) \leq c_{2} \mu( y_{j} +  3 Z_{0}(\bi_{+}, n) ) \}.
		\earyst
		
		We estimate on $\cJ_{1}$. Since the family $\{ y_{j} + Z_{0}(\bi_{+}, n) \}_{j \in \cJ}$ has multiplicity $O_d(1)$, and $y_{j} + Z_{0}(\bi_{+}, n) \subset x +  3 Z(\bi_{+}, n)$ for every $j \in \cJ$ (because $y_{j} \in Y \subset x + Z(\bi_{+}, n)$ and clearly $Z_{0}(\bi_{+}, n) \subset \sqrt{2}\, Z(\bi_{+}, n)$), we have
		\aryst
		\sum_{j \in \cJ_{1}} \mu( Y_{j} ) \leq c_{2} \sum_{j \in \cJ_{1}} \mu( y_{j} + Z_{0}(\bi_{+}, n) ) \leq C_{d} c_{2} \mu( x +  3 Z(\bi_{+}, n) ) \leq C_{d} c_{0}^{-1} c_{2} \mu( x + Z(\bi_{+}, n) ),
		\earyst
		where the last inequality is \eqref{eq localdensitybetweennearbyscales0}.
		
		We estimate on $J_{2}$. We first claim that the enlarged family $\{ y_{j} +  3 Z_{0}(\bi_{+}, n) \}_{j \in \cJ}$ still has multiplicity $O_d(1)$. %
		Since $y_{j} +  3 Z_{0}(\bi_{+}, n) \subset x +   6 Z(\bi_{+}, n)$ (again by   $Z_{0}(\bi_{+}, n) \subset \sqrt{2}\, Z(\bi_{+}, n)$), we obtain
		\aryst
		\sum_{j \in \cJ_{2}} \mu( Y_{j} ) \leq \sum_{j \in \cJ_{2}} \mu( y_{j} + Z_{0}(\bi_{+}, n) ) \leq c_{2} \sum_{j \in \cJ_{2}} \mu( y_{j} +  3 Z_{0}(\bi_{+}, n) ) \\
		\leq C_{d} c_{2} \mu( x +  6 Z(\bi_{+}, n) ) \leq C_{d} c_{0}^{-1} c_{2} \mu( x + Z(\bi_{+}, n) ),
		\earyst
		where the last inequality is \eqref{eq localdensitybetweennearbyscales0}. 
		
		Combining the two estimates and dividing by $\mu( x + Z(\bi_{+}, n) )$, we obtain \eqref{eqn:Rc measure estimate}.

		By \eqref{eq lowerboundforY}, and by $c_2 \ll_d c_0^2$   and $\eta' \ll_{d} c_0 c_2$, we have  	 
		\aryst
		\mu^{x + Z(\bi_{+}, n)}( ( R \cap Y )^c ) \leq     2 \eta'  + 2C_d c_0^{-1} c_2 \ll_d c_0.
		\earyst
		Due to \eqref{eq localdensitybetweennearbyscales0}, we further obtain
		\begin{equation}
			\mu((x+Z(\bi_+,n))\cap R\cap Y)\geq (1-c_0)\mu(x+Z(\bi_+,n))> \frac{c_0}{2}\mu(x+2Z(\bi_+,n))
		\end{equation}
		Applying Disintegration along $V_{m_0+1}(\bi_+)$, there exists $z$ such that 
		\begin{equation}\label{eqn:lower bound box}
			\mu_z^{V_{m_0+1}(\bi_+)}((x+Z(\bi_+,n))\cap R\cap Y)> \frac{c_0}{2}\mu_z^{V_{m_0+1}(\bi_+)}(x+2Z(\bi_+,n)) 
		\end{equation} 
		By the definition of $Z(\bi_+,n)$, in each fiber $z+V_{m_0+1}(\bi)$, 
		\aryst (x+Z(\bi_+,n))\cap (z+V_{m_0+1}(\bi))&\subset& \bigcup _{y\in (x+Z(\bi_+,n))\cap (z+V_{m_0+1}(\bi))}\cD_{m_0+1, \bi_{+}, C_0 n}(y)\\
		&\subset& (x+2Z(\bi_+,n))\cap (z+V_{m_0+1}(\bi)) \earyst
		Applying Rokhlin's Disintegration Theorem to $(z+V_{m_0+1}(\bi_+))\cap (x + 2Z(\bi_{+}, n))$ along the partition $\cD_{m_0+1, \bi_{+}, C_0 n}$, together with the above inclusion relation and \eqref{eqn:lower bound box} give that there is $y\in (x+Z(\bi_+,n))\cap(z+V_{m_0+1}(\bi_+))$ such that
		\[ \mu_y^{V_{m_0+1}(\bi_+)}(\cD_{m_0+1, \bi_{+}, C_0 n}(y)\cap (R\cap Y))> \frac{c_0}{2} \mu_y^{V_{m_0+1}(\bi_+)}(\cD_{m_0+1, \bi_{+}, C_0 n}(y)).\]
		We notice that for $m_0\geq m$, we can apply \cref{lem smallersupport} to obtain that for a.e. $y\in x+Z(\bi_+,n)$, we have $\mu_{y, \bi_{+}}^{\cD_{m+1, \bi_{+}, C_0 n}(y)}=\mu_{y, \bi_{+}}^{\cD_{m_0+1, \bi_{+}, C_0 n}(y)}$. The above equation translates to
		\begin{equation}
			\label{equ large fiber}
			\mu_{y, \bi_{+}}^{  \cD_{m+1, \bi_{+}, C_0 n}(y)   } (R \cap Y) > c_0/2.
		\end{equation}
		The conditional measure is supported on the atom, so the atom intersects $R \cap Y$. Choosing the base point $y$ inside this intersection (the conditional measure only depends on the atom), we may assume $y \in (x + 2Z(\bi_{+}, n)) \cap (R \cap Y)$ (only $y \in Y$ and $d(x, y) \lesssim 2^{-C_0 n}$ are used below).

		We set
		\aryst
		\widetilde\zeta :=  \mu_{y, \bi_{+}}^{   \cD_{m+1, \bi_{+}, C_0 n}(y )  }.
		\earyst
		As $y \in Y \subset \Lambda_{\bi_{+}, \bj_{+}}$, by  Step 5(2) in Subsection \ref{subsec preparingarandomset}, the measure 
		$$ (\Dil_{2^{(C_0+1)n}})_*(P^{J_m(y,\bj_+)})_* \widetilde\zeta $$
		is $( n', l, \eta_0)$-ball spreading. 
		Since both $x, y \in \Lambda_1$, we have 
		\aryst
		d( J_m(x, \bj_{+}  ), J_m(y, \bj_{+}  )  )  \leq
		d( U^{m-1}(x), U^{m-1}(y) ) \leq C_{\eta_0}2^{- C_0 \alpha n} \leq 2^{- 3n},
		\earyst
		where $C_{\eta_0}$ is a constant from Step 1 and $n$ is taken large enough depending on $\eta_0$ to kill this constant.
		By \cref{lem spreading nearby}, the measure 
		\aryst
		(\Dil_{2^{(C_0+1)n}})_*(P^{J_m(x,\bj_+)})_* \widetilde\zeta
		\earyst
		is also $( n', l+2, 2\eta_0)$-ball spreading. 
		
		Recall that $Q'=P^{J_m(x,\bj_+)}P_{J_m(x,\bi_+)}$. To finish the proof, we let
		\begin{equation}
			\zeta:=(Q')_* \widetilde\zeta. 
		\end{equation}
		Since $\widetilde{\zeta}$ is supported on $y+V_{m+1}(\bi_+)\subset y+J_m(x,\bi_+)$, for $\widetilde{\zeta}$-a.e. $z$, we have $P^{J_m(x,\bi_+)}z=P^{J_m(x,\bi_+)}(y)$, which gives $Q'z=P^{J_m(x,\bj_+)}z-P^{J_m(x,\bj_+)}P^{J_m(x,\bi_+)}y$. Thus $\zeta$ is a translation of $(P^{J_m(x,\bj_+)})_*\widetilde{\zeta}$, so its ball-spreading property and support diameter are preserved.
		
		We still need to verify the mass condition \eqref{eqn:mass lower bound} for $\zeta$. \eqref{eqn:Q' diam} gives
		\begin{equation}
			Q'(y_j+Z_0(\bi_+,n))\subset Q' y_j+B^{J_m(x,\bj_+)^{\perp}}(r_n).
		\end{equation}
		It follows that
		\begin{equation}
			\zeta\left(\bigcup_{j\in \cJ_0}\big(Q'y_j+B^{J_m(x,\bj_+)^{\perp}}(r_n)\big)\right)\geq \widetilde{\zeta}\left(\bigcup_{j\in \cJ_0}\big(y_j+Z_0(\bi_+,n)\big)\right)>c_0/2.
		\end{equation}
		\smallskip
		
		\subsection{Structure of the good pieces: proof of Lemma \ref{lem Yjisstructured}}
		
		For any $y \in Y  \subset \Lambda_1 \cap (x + Z(\bi_{+}, n ))$, we deduce that 
		\aryst
		d( U^{m-1}(x), U^{m-1}(y) ) \leq C_{\eta_0} 2^{ - \alpha C_0 n} \leq 2^{- n} / d,
		\earyst
		and hence
		\ary \label{eq WxWy}
		\frac{1}{2}  W(x, n) \subset  W(y, n) \subset 2W(x, n).
		\eary
		By Besicovitch's Covering Lemma, we can find a finite subset $\hat{G}_j \subset Y_j$ such that $\{  y + W(y, n) \}_{y \in \hat{G}_j}$ forms a covering of $Y_j$ with multiplicity $O_d(1)$.

		Let $c_3 \ll_d c_2^2$ be a small parameter. 
		We let $G_j$ be the set of $y \in \hat{G}_j$ such that 
		\aryst
		\mu^{y + W(y, n)}(  Y_j   ) > c_3.
		\earyst
		By letting $\eps_0 < c_3$ and by $G_j \subset Y_j \subset Y \subset \Lambda_2(n)$, we deduce by Step 2 in Subsection \ref{subsec preparingarandomset}  that for $y\in G_j$
		\aryst
		N_{ r_n}( Y_j \cap (y + W(y, n)) ) > 2^{(m-1 - \eps_0) n'},
		\earyst 
		Item \eqref{itm Gj1} then follows from \eqref{eq WxWy}.

		Denote 
		\aryst
		\widetilde{Y}_j = \bigcup_{y \in  G_j} (y + W(y, n)) \cap Y_j.
		\earyst
		Then  by the definitions of $\hat{G}_j$, $G_j$ and \eqref{def;j0}, we have 
		\aryst
		&&\mu^{y_j + Z_0(\bi_{+}, n)} (\widetilde{Y}_j  ) \\
		&\geq&  \mu^{y_j + Z_0(\bi_{+}, n)} (Y_j) - \mu^{y_j + Z_0(\bi_{+}, n)} ( \bigcup_{y \in \hat{G}_j \setminus G_j} (y + W(y, n)) \cap Y_j ) \\
		&>& c_2    -   \sum_{y \in \hat{G}_j \setminus G_j }
		\frac{\mu( y + W(y, n) )}{\mu(y_j + Z_0(\bi_{+}, n))}
		\mu^{ y + W(y, n) }(  Y_j  ) \\
		&>&  c_2 - C_d c_3 \frac{ \mu(y_j +   3Z_0(\bi_{+}, n)) }{ \mu(y_j + Z_0(\bi_{+}, n)) } \\
		&>&  c_2 - C_d  c_3 / c_2 > c_2/2.
		\earyst
		Note that in the third inequality above, we have used \eqref{def;j0} and that 
		\aryst
		\bigcup_{y \in \hat{G}_j}  y + W(y, n) \subset y_j +   3 Z_0(\bi_{+}, n).
		\earyst
		Indeed, since orthogonal projection is 1-Lipschitz and both layers of $W(y,n)$ have radius $\le \rho_n$, we have $W(y, n) \subset (1 + r_n/\rho_n) Z_0(\bi_{+}, n) = (1 + 2^{-n'}) Z_0(\bi_{+}, n)$, hence $y + W(y, n) \subset y_j + (2 + 2^{-n'}) Z_0(\bi_{+}, n) \subset y_j + 3 Z_0(\bi_{+}, n)$ for every $y \in \hat G_j \subset y_j + Z_0(\bi_{+}, n)$.
		By  Step 3 in Subsection \ref{subsec preparingarandomset} and by letting $\eps_0 < c_2/2$, we have 
		\ary  \label{eq coveringPVY'jlower}
		N_{ r_n}( P^{V_{m+1}(\bi_{+})} (\widetilde{Y}_j ) ) > 2^{(D_{m+1}  - \eps_0)  n'}.
		\eary

		Note that  we have 
		\ary  \label{eq P=PP}
		P^{ 	J_m(x, \bi_{+}  )} =  P^{V_{m+1}(\bi_{+}) + U^{m-1}(x)} = P^{U} P^{V_{m+1}(\bi_{+})}
		\eary 
		where
		\aryst
		U = P^{V_{m+1}(\bi_{+})}( U^{m-1}(x) ) \in \Gr( V_{m+1}(\bi_{+})^{\perp}, m -1).
		\earyst 
		Then by \eqref{eq coveringPVY'jlower} and \eqref{eq P=PP}, we deduce 
		\aryst
		N_{ r_n}(   P^{ 	J_m(x, \bi_{+}  )}  (\widetilde{Y}_j ) )  &>&  2^{- (m-1 +\eps_0/2)  n'} N_{ r_n}( P^{V_{m+1}(\bi_{+})} (\widetilde{Y}_j ) )  \\
		&>&  2^{(D_{m+1} - m + 1  - 3 \eps_0/2)  n'}.
		\earyst 
		Note that in the first inequality above we have used the following trivial upper bound on the size of the fiber (compare with item  \eqref{itm Gj1}):
		\aryst 
		\sup_{y}	N_{ r_n}( Y_j \cap (y + W(x, n)) ) < 2^{(m-1 + \eps_0/2) n'},	\earyst
		and 
		\[P^{V_{m+1}(\bi_+)}\widetilde{Y}_j \cap (P^U)^{-1}B_{J_m(x,\bi_+)^\perp}(y,r_n/2)\subset P^{V_{m+1}(\bi_+)}(y+W(x,n)). \]
		Due to \eqref{eq WxWy},  for $y\in G_j$, we have 
		\aryst P^{ 	J_m(x, \bi_{+}  )}  (y+W(y,n))\subset P^{ 	J_m(x, \bi_{+}  )}  (y)+B_{J_m(x, \bi_{+}  )^\perp}(2r_n)
		\earyst
		Hence by definition of $\widetilde{Y}_j$
		\aryst
		N_{r_n}(B_j)=N_{r_n}(P^{J_m(x, \bi_{+}  )}G_j)\geq N_{2r_n}(P^{J_m(x, \bi_{+}  )}(\widetilde{Y}_j)) \geq 2^{-d} N_{ r_n}(   P^{ 	J_m(x, \bi_{+}  )}  (\widetilde{Y}_j ) ).
		\earyst
		This gives the first inequality in \eqref{itm Gj2}. 
		
		It remains to prove the second inequality in \eqref{itm Gj2}.
		Note that we have 
		\aryst
		\bigcup_{j \in \cJ_0} G_j \subset  \bigcup_{j \in \cJ_0} Y_j \subset Y.
		\earyst
		Thus we have
		\aryst
		\bigcup_{j \in \cJ_0} B_j \subset   P^{ 	J_m(x, \bi_{+}  )}  ( \bigcup_{j \in \cJ_0} G_j)  \subset P^{ 	J_m (x, \bi_{+}  )}( Y ).   
		\earyst
		Recall that
		\aryst
		N_{r_n}( P^{V_{m+1}(\bi_{+})}( Y )   ) <  2^{(D_{m+1} + \eps_0)  n'}.
		\earyst
		By \eqref{eq P=PP} and item \eqref{itm Gj1}, we can deduce
		\aryst
		N_{ r_n}(  \bigcup_{j \in \cJ_0} B_j   ) \leq 		N_{ r_n}(  P^{U} P^{V_{m+1}(\bi_{+})}( \cup_{j\in \cJ_0} G_j )   ) <  2^{(D_{m+1} - m + 1 +   3\eps_0)  n'}.
		\earyst
		This is because:  by applying \cref{lem numbercoverofprojection} to $S=P^{V_{m+1}(\bi_{+})}( \cup_{j\in \cJ_0} G_j ), T= P^{V_{m+1}(\bi_{+})}( Y)   , r=2r_n$ and noticing that $(P^U)^{-1}(B^{U^\perp}(r))\supset P^{V_{m+1}(\bi_+)}2W(x,n) $, we have
		\aryst
		N_{2r_n}(  P^{U} P^{V_{m+1}(\bi_{+})}( \cup_{j\in \cJ_0} G_j )   ) \leq \frac{N_{2r_n}(  P^{V_{m+1}(\bi_{+})}( Y)   ) }{\min_{y\in G_j,j\in \cJ_0}N_{2r_n}(  P^{V_{m+1}(\bi_{+})}(Y\cap ( y+2W(x,n))   )) }
		\earyst
		By \eqref{eq coveringnumberoffiber} applied with $\bi_{+}$ in place of $\bj_{+}$ (the angle condition in Step 5(1) holds for both), for $y \in G_j$ the denominator satisfies
		\aryst
		&N_{r_n}(  P^{V_{m+1}(\bi_{+})}(Y\cap ( y+2W(x,n)) ) ) \geq N_{r_n}(  P^{V_{m+1}(\bi_{+})}(Y_j\cap ( y+2W(x,n)) ) ) \\
		&\geq 2^{-d\eps_1 n - O(1)} N_{r_n}( Y_j \cap (y + 2W(x,n)) ) \geq 2^{(m-1-\eps_0)n' - d\eps_1 n - O(1)}.
		\earyst
		This completes the proof.

		\appendix

		\section{More on the Oseledets splitting}
		We are collecting more results on Oseledets splitting.

		\begin{cor} \label{cor Oseledets restriction}
			Assume that the Lyapunov spectrum $\chi_1 > \cdots > \chi_d$ is simple.
			Then for $\nu_{+}$-a.e.\ $\bi_{+}$ and every $k \in \{1, \cdots, d\}$,
			\ary \label{eq restriction norm}
			&&\lim_{t \to \infty} \frac{1}{t} \log \big\| Df^{t}_{\bi_{+}} |_{V_k(\bi_{+})} \big\| \ = \ \chi_k\\
			&&\lim_{t \to \infty} \frac{1}{t} \log \mathrm{J}_{V_k(\bi_{+})} \big( Df^{t}_{\bi_{+}} \big) \ = \ \chi_k + \chi_{k+1} + \cdots + \chi_d,\nonumber
			\eary
			where $\| Df^{t}_{\bi_{+}} |_{V_k(\bi_{+})} \| = \sup \{ \| Df^{t}_{\bi_{+}} v \| / \| v \| : v \in V_k(\bi_{+}) \setminus \{0\} \}$.
			In particular, for every $\eps > 0$ and all $t \geq t_0(\bi_{+}, \eps)$,
			\ary \label{eq restriction sv}
			&&\big\| Df^{t}_{\bi_{+}} |_{V_k(\bi_{+})} \big\| \ \leq \ e^{\eps t} \, \alpha_k \big( Df^{t}_{\bi_{+}} \big),\\
			&&\mathrm{J}_{V_k(\bi_{+})} \big( Df^{t}_{\bi_{+}} \big) \ \leq \ e^{\eps t} \, \alpha_k \big( Df^{t}_{\bi_{+}} \big) \cdots \alpha_d \big( Df^{t}_{\bi_{+}} \big).\nonumber
			\eary
		\end{cor}
		
		Recall \cref{thm Oseledets}, we write $\lambda_1 \geq \cdots \geq \lambda_d$ for the Lyapunov exponents repeated according to multiplicity, so that $\lambda_m = \chi_j$ whenever
		$d_1 + \cdots + d_{j-1} < m \leq d_1 + \cdots + d_j$.
		
		\begin{cor}\label{cor Oseledets restriction sv}
			Let $\cF$, $\nu$,
			$\chi_j, d_j, V_j, U^j, E_j$ be as in \Cref{thm Oseledets}. Then for $\nu$-almost every $\bi \in \Sigma$, the following holds.
			\enmt
			
			\item For $j \in \{1, \dots, k\}$, set $r_{j} = \dim V_{j}(\bi_{+}) = d_{j} + \cdots + d_{k}$
			and $s_{j} = d_{1} + \cdots + d_{j-1}$, so that $s_{j} + r_{j} = d$. Then for every
			$p \in \{1, \dots, r_{j}\}$,
			\aryst
			\lim_{n \to \infty} \frac{1}{n}
			\log \alpha_{p} \big( M_{\bi_{+}}^{n} |_{V_{j}(\bi_{+})} \big) = \lambda_{s_{j} + p} .
			\earyst
			
			\item For every $j \in \{1, \dots, k\}$ and every $\varepsilon > 0$, there is
			$C = C(\bi, \varepsilon) \geq 1$ such that
			\aryst
			C^{-1} e^{n ( \chi_{j} - \varepsilon )} \norm{v}
			\ \leq \ \norm{ M_{\bi_{+}}^{n} v }
			\ \leq \ C e^{n ( \chi_{j} + \varepsilon )} \norm{v} ,
			\qquad v \in E_{j}(\bi), \ n \geq 0 .
			\earyst
			\eenmt
		\end{cor}
		
		The following lemma is used in Step 1 of the proof of \Cref{prop fiberwisedim}. We use the notation of Subsection \ref{subsec setting1}: the integers $m$ and $m_0$, the time $T_n = \lfloor n \log 2 / (\chi_m - \chi_{m_0+1}) \rfloor$ and the number $c$ with $2c = (\chi_m - \chi_{m+1}) / (\chi_m - \chi_{m_0+1})$. The lemma follows directly from \cref{cor Oseledets restriction} and \cref{cor Oseledets restriction sv}.
		
		\begin{lemma} \label{lem Oseledets statistics}
			Let $\eps > 0$ and $\eta > 0$. There is a Borel set $A = A(\eps, \eta) \subset \Sigma_{+}$ with $\nu_{+}(A) > 1 - \eta$ such that for every integer $n \gg_{\cF,  \eps, \eta} 1$ and every $\bi_{+} \in A$, the map $M := Df^{T_n}_{\bi_{+}}$ satisfies
			\ary
			\label{eq stat gap1}
			\alpha_{m}(M) &>& 2^{(1 - \frac{\eps}{3}) n} \, \alpha_{m_0+1}(M), \\
			\label{eq stat gap2}
			\alpha_{m}(M) &>& 2^{(2c - \frac{\eps}{3}) n} \, \alpha_{m+1}(M), \\
			\label{eq stat jac}
			\mathrm{J}_{V_{m+1}(\bi_{+})}(M) &\leq& 2^{\frac{\eps n}{3}} \, \alpha_{m+1}(M) \cdots \alpha_{d}(M), \\
			\label{eq stat norm1}
			\sup_{v \in V_{m_0+1}(\bi_{+}) \setminus \{0\}} \frac{\| M v \|}{\| v \|} &\leq& 2^{\frac{\eps n}{3}} \, \alpha_{m_0+1}(M), \\
			\label{eq stat norm2}
			\sup_{v \in V_{m+1}(\bi_{+}) \setminus \{0\}} \frac{\| M v \|}{\| v \|} &\leq& 2^{\frac{\eps n}{3}} \, \alpha_{m+1}(M).
			\eary
		\end{lemma}

		\section{Linear algebra: deferred proofs} \label{app linear algebra}
		\label{app proofs lemmas}
		
		Recall that for a subspace $V \subset \R^d$ of dimension $k$ and $M \in \GL(d,\R)$, we write
		\[ \mathrm{J}_V(M) := \frac{\| (\wedge^{k} M) \, v \|}{\|v\|}, \qquad 0 \neq v \in \wedge^{k} V. \]
		
		\begin{lemma} \label{lem lowerboundforpojection}
			Let $J \in \Gr(d, d-1)$ and $M \in \GL(d, \R)$. Then, for every $v \in \R^d$,
			\ary \label{eq proj identity}
			\| P^{MJ} M v \| = \frac{|\det M|}{\mathrm{J}_J(M)} \, \| P^{J} v \| .
			\eary
			Moreover, let $\rho > 0$ and $L \in \Gr(d, 1)$. If 
			\begin{equation}
				\label{eqn:angle bound}
				\sin \measuredangle(L, J), \, \sin \measuredangle(ML, MJ) \geq \rho, 
			\end{equation}
			then for every nonzero $v \in \R^d$ and $w \in L$,
			\ary \label{eq proj ratio}
			\frac{\| P^{MJ} M v \|}{\| M v \|} \ \in \ [\rho, 1/\rho] \cdot \frac{\| P^{J} v \| \, \| M w \|}{\| M v \| \, \| w \|} .
			\eary
		\end{lemma}
		
		\begin{proof}[Proof of \cref{lem lowerboundforpojection}]
			
			By definition of projection, we have
			\[\|P^Jv\|=\frac{\|v\wedge \wedge^{d-1}J \|  }{\|\wedge^{d-1}J \|},\  \|P^{MJ}Mv\|=\frac{\|Mv\wedge \wedge^{d-1}MJ \|  }{\|\wedge^{d-1}MJ \|}= \frac{|\det M|\|v\wedge \wedge^{d-1}J \|  }{\|\wedge^{d-1}MJ \|} \]
			On the other hand, from the angle, we have
			\[ \sin\measuredangle(ML,MJ)=\frac{\|Mw\wedge\wedge^{d-1} MJ\|}{\|Mw\| \|\wedge^{d-1} MJ\|}=\frac{|\det M|\|w\|\|\wedge^{d-1}J\|}{\|Mw\| \|\wedge^{d-1}MJ\|}\sin\measuredangle(L,J) \]
			
			Hence, we have
			\[ \frac{\|P^{MJ}Mv\|}{\|P^Jv \|}=\frac{\sin\measuredangle(ML,MJ)}{\sin\measuredangle(L,J) }\frac{\|M w\|}{\|w\|}. \]
		\end{proof}
		\begin{lemma} \label{lem distance-control}
			Let $m,m'\in \{1,\cdots,d-1\}$. Let $M \in \GL(d, \R)$, $L \in \Gr(d, 1)$, $V \in \Gr(d, d-m)$, and ${ C_2}, { C_3} \geq 1$.
			Suppose that
			\aryst
			\| M w \| \ \leq \ { C_2} \, \alpha_{m'+1}(M) \, \| w \| \quad \mbox{for } w \in L ,
			\quad \mbox{and} \quad
			\mathrm{J}_V(M) \ \leq \ { C_3} \, \alpha_{m+1}(M) \cdots \alpha_d(M) .
			\earyst
			Then
			\[ d(V, L) \ \leq \ { C_2}{ C_3} \, \frac{\alpha_{m'+1}(M)}{\alpha_m(M)} . \]
		\end{lemma}
		
		\begin{proof}
			We may assume that $\delta:=d(V,L)>0$. Then $V' := V \oplus L \in \Gr(d, d-m+1)$.
			Let $w\in L$ be a unit vector. Fix an orthonormal basis $e_1, \cdots, e_{d-m}$ of $V$.
			We have $\mathrm{vol}_{d-m+1}(e_1, \cdots, e_{d-m}, w) = \mathrm{dist}(w, V) = \delta$.
			This gives
			\aryst
			\mathrm{vol}_{d-m+1}(M e_1, \cdots, M e_{d-m}, M w)
			=\mathrm{J}_{V'}(M) \, \delta \ \geq \ \alpha_{m}(M) \alpha_{m+1}(M) \cdots \alpha_{d}(M) \, \delta .
			\earyst
			On the other hand, $M e_1, \cdots, M e_{d-m}$ span $MV$ and $\mathrm{vol}_{d-m}(M e_1, \cdots, M e_{d-m}) = \mathrm{J}_{V}(M)$, so 
			\aryst
			\mathrm{vol}_{d-m+1}(M e_1, \cdots, M e_{d-m}, M w)
			&=& \mathrm{J}_{V}(M) \cdot \mathrm{dist}(M w, MV) \ \leq \ \mathrm{J}_{V}(M) \, \| M w \| \\
			&\leq& { C_3} \, \alpha_{m+1}(M) \cdots \alpha_{d}(M) \cdot { C_2} \, \alpha_{m'+1}(M) .
			\earyst
			Comparing the two displays and dividing by $\alpha_{m}(M) \cdots \alpha_{d}(M) > 0$ yields $\delta \leq {C_2}{ C_3} \alpha_{m'+1}(M) / \alpha_m(M)$.
		\end{proof}

		The following identity is used to define $\lambda_t$ in \eqref{eq scaling}, to prove \eqref{eq tube identity} and \eqref{eq lambda increments}, and in the proof of \cref{lem fullness}.
		
		\begin{lemma} \label{lem base times height}
			Let $M \in \GL(d, \R)$. Let $V \subset \R^d$ be a subspace and $W$ a hyperplane of $V$, and let $u$ be a unit vector of $V \cap W^{\perp}$. Then
			\aryst
			\mathrm{J}_{V}(M) \ = \ \mathrm{J}_{W}(M) \, d(Mu, MW),
			\earyst
			and, for every $v \in V$,
			\aryst
			d(Mv, MW) \ = \ \frac{ \mathrm{J}_{V}(M) }{ \mathrm{J}_{W}(M) } \, d(v, W) .
			\earyst
		\end{lemma}
		
		\begin{proof}
			Put $V' := MV$ and $W' := MW$. 
			Computing the Jacobian of $M|_{V}$ in an orthonormal basis of $W$ completed by $u$ gives $\mathrm{J}_{V}(M) = \mathrm{J}_{W}(M) \, d(Mu, W')$, which is the first identity. 
			
			Since $W$ is a hyperplane of $V$, $d(v, W) = |\phi(v)|$ for $v \in V$, where $\phi(v) := \langle v, u \rangle$. Similarly, $d(v', W') = |\phi'(v')|$ for $v' \in V'$, with $\phi'(v') := \langle v', u' \rangle$ and $u'$ a unit vector of $V' \cap (W')^{\perp}$. The linear functional $\phi' \circ M$ on $V$ vanishes on $W$. Hence, $$\phi' \circ M = c \, \phi\quad \text{with} \quad|c| = |\phi'(M u)| = d(Mu, W').$$ Therefore, for every $v\in V$,  $$d(Mv, W') = |\phi'(Mv)| = |c| \, |\phi(v)| = \big( \mathrm{J}_{V}(M) / \mathrm{J}_{W}(M) \big) \, d(v, W).$$ 
		\end{proof}

		\subsection{Proofs of Lemmas \ref{lem slab transfer} and \ref{lem kernel pullback}}
		\begin{proof}[Proof of \cref{lem slab transfer}]
			Let $N$ be the level of the atom $D$, so that $D \in \cD_{m+1, \bi_{+}, N}$, and let $A = A(x, \bi_{+})$, $J = J(x, \bj_{+})$, $K$ and $\lambda$ be as in \eqref{eq A def}-\eqref{eqn:lambda norm}. Recall that $\mu_{x, \bi_{+}}$ is carried by $A$, as $(x, \bi_{+}) \in \Xi$.
			Write $V := V_{m_0+2}(\bi_{+})$.
			We shall use the elementary inequality
			\begin{equation}
				\label{eqn:dist K V}
				| d(w, K) - d(w, V) | \leq \| P^{K} - P^{V} \| \, |w|\quad \text{for}\quad w \in V_{m_0+1}(\bi_{+}).
			\end{equation}
			
			Put $r_1 := 2^{-N-l}$. 
			For $y\in A$, let $\widehat\Tube_{r}(y, V) = \widehat\Tube_{r}(y, \bi_{+})$ and $\widehat\Tube_{r}(y, K)$ be the tubes around $y + V$ and $y + K$ in $A$, given in \eqref{eq slab def}.
			Suppose $\widehat\Tube_{r_1}(y, K) \cap D \neq \emptyset$. For any $y',z \in \widehat\Tube_{r_1}(y, K) \cap D$, we have
			\aryst
			&&d(z - y', V) \ \leq \ d(z - y', K) + \| P^{K} - P^{V} \| \, |z - y'| \\
			&<& \ 2 r_1 + \| P^{K} - P^{V} \| \diam (D) \ \leq \ 3 r_1,
			\earyst
			where we use  \eqref{eqn:dist K V} in the first inequality and the upper bound $\diam (D)\leq \sqrt{d}2^{-N}$ in the last one.
			Thus $\widehat\Tube_{r_1}(y, K) \cap D \subset \widehat\Tube_{3 r_1}(y', V)$. Combining it with the $\eta$-good hypothesis, we obtain
			\ary
			\label{eqn:measure tube}
			\mu^{D}_{x, \bi_{+}} \big( \widehat\Tube_{r_1}(y, K) \big) \ < \ 4 \eta \qquad \mbox{for every } y \in A .
			\eary
			The restriction of $P^{J}$ to $V_{m_0+1}(\bi_{+})$ is a linear map of rank one with kernel $K$. So $| P^{J} w | = \lambda \, d(w, K)$ for $w \in V_{m_0+1}(\bi_{+})$ by  definition of $\lambda$.
			This gives, for $y \in A$, 
			\ary
			\label{eqn:tube ball}
			\{z\in A: P^{J} z \in B(P^{J} y, \lambda r_1)\}=\widehat\Tube_{r_1}(y, K).
			\eary
			\eqref{eqn:measure tube} and \eqref{eqn:tube ball} imply $$\sup_{p} (P^{J})_* \mu^{D}_{x, \bi_{+}} \big( B(p, \lambda 2^{-N-l}) \big) < 4 \eta,$$ which is \eqref{eq good sequence} with $4 \eta$ in place of $\eta$. Hence, $\bj_{+}$ is $4 \eta$-good for $(D, l)$ with respect to $(x, \bi_{+})$.
		\end{proof}
		
		\begin{proof}[Proof of \cref{lem kernel pullback}]
			We abbreviate
			\aryst
			&&E' := E_{m_0+1}(x', \bi'_{+}), \quad V' := V_{m_0+2}(\bi'_{+}), \quad K' := K(x', \bi'_{+}, \bj'_{+}) =  H' \cap V_{m_0+1}(\bi'_{+}),\\
			&&E_0 := E_{m_0+1}(x, \bi_{+}), \quad V_0 := V_{m_0+2}(\bi_{+}).
			\earyst
			We have $V_{m_0+1}(\bi'_{+}) = E' \oplus V'$ and $V_{m_0+1}(\bi_{+}) = E_0 \oplus V_0$.
			
			We first prove that $K'$ is a graph over $V'$ with a slope of at most $e^{\eps n}$.
			By \eqref{equ:hum-1}, $E' \not\subset H'$. As $E' \subset V_{m_0+1}(\bi'_{+})$, this shows $V_{m_0+1}(\bi'_{+}) \not\subset H'$, so $K'$ is a hyperplane of $V_{m_0+1}(\bi'_{+})$ with $K' \cap E' = \{0\}$.
			Hence the projection of $K'$ to $V'$ along $E'$ is injective between spaces of the same dimension, thus bijective, and there is a linear map $\phi : V' \to E'$ such that
			\aryst
			K' = \{ v + \phi(v) \ : \ v \in V' \} .
			\earyst
			Let $v \in V'$ be a unit vector and put $w := v + \phi(v) \in K'$.
			Then we have $d(\phi(v), K') = d(v, K') \leq |v| = 1$.
			On the other hand, 
			\aryst
			d(\phi(v), K') \ \geq \ |\phi(v)| \, \sin \measuredangle(E', K') \ \geq \ |\phi(v)| \, \sin \measuredangle(E', H') \ > \ |\phi(v)| \, e^{-\eps n} .
			\earyst
			Therefore $\| \phi \| \leq e^{\eps n}$.
			
			Next, we pull back to the time $0$.
			By \eqref{eq M equivariance}, $M V_0 = V'$, $M E_0 = E'$ and $M K = K'$. Hence,
			we get
			\aryst
			K = M^{-1} K' = \{ u + \psi(u) \ : \ u \in V_0 \} ,
			\qquad
			\psi := \big( M |_{E_0} \big)^{-1} \circ \phi \circ M |_{V_0} : V_0 \to E_0 .
			\earyst
			By hypothesis (a), $\| M |_{V_0} \| \leq e^{t (\chi_{m_0+2} + \eps)}$ and, $E_0$ being a line, $$\| (M |_{E_0})^{-1} \| = \big( \min_{v' \in E_0, |v'| = 1} | M v' | \big)^{-1} \leq e^{- t (\chi_{m_0+1} - \eps)}.$$
			Combining with the upper bound for $\|\phi\|$,
			\aryst
			\| \psi \| \ \leq \ \| (M |_{E_0})^{-1} \| \, \| \phi \| \, \| M |_{V_0} \| \ \leq \ e^{- t (\chi_{m_0+1} - \eps)} \, e^{\eps n} \, e^{t (\chi_{m_0+2} + \eps)}.
			\earyst
			
			For a unit vector $u \in V_0$, the vector $u + \psi(u)$ lies in $K$, so $d(u, K) \leq | \psi(u) | \leq \| \psi \|$.
			Taking the maximum over $u$ gives $d(V_0, K) \leq \| \psi \|$, which is \eqref{eq kernel close}.
		\end{proof}
		
		\subsection{Fullness of the pulled-back cubes}
		We use the notation from the proof of \cref{lem goodonaverage}: the linear parts $L_{x,t}$ of \eqref{eqn:L equivariant}, the functions $\lambda_t$ of \eqref{eq scaling},  the stopping times $\tau_k$ of \eqref{eq tau bounds}, the set $\Gamma_0\subset \Lambda\times \Sigma_+$ with its conditions (G1)–(G3),  the sets $T_{n,k}(x, \bi_{+})$, the notion of a full set, and the level $N = C_0 n + k$.
		
		\begin{lemma} \label{lem fullness}
			Let $0 \leq k < n$. If $(x, \bi_{+}) \in \Gamma_0$, $\sigma^{- \tau_k(x, \bi_{+})}(x, \bi_{+}) \in \Gamma_0$ and $T_{n, k}( x, \bi_{+} )$ is not full, then $$d \big( x, \partial \cD_{ \bi_{+}, C_0n+k}(x) \big) < 2^{-l/4 - C_0 n - k}.$$
		\end{lemma}
		
		\begin{proof}
			Write $t := \tau_k(x, \bi_{+})$, $(x', \bi'_{+}) := \sigma^{-t}(x, \bi_{+}) = (f_x^{-t}(x), x[t]\bi_{+})$ and $M_t := Df_{x[t]} = Df^{t}_{\bi'_{+}} = L_{x,t}^{-1}$. Put $$V := V_{m_0+1}(\bi'_{+}),\quad W := V_{m_0+2}(\bi'_{+}),\quad E' := E_{m_0+1}(x', \bi'_{+}) \subset V,$$ so that $M_t V = V_{m_0+1}(\bi_{+})$, $M_t W = V_{m_0+2}(\bi_{+}) =: W_1$ and $M_t E' = E_{m_0+1}(x, \bi_{+}) =: E$. 
			
			Let $u$ be a unit vector of $V \cap W^{\perp}$ and $e'$ a unit vector of $E'$. By \cref{lem base times height}, applied to $M_t = L_{x,t}^{-1}${, $V$ and $W$},
			\aryst
			2^{-\lambda_t(x, \bi_{+})} = \frac{ \mathrm{J}_{V}(M_t) }{ \mathrm{J}_{W}(M_t) } = d(M_t u, M_t W) .
			\earyst
			Write $u = \alpha e' + w_0$ with $w_0 \in W$. We have $1 = \alpha \langle e', u \rangle$, and $|\langle e', u \rangle| = d(e', W) = \sin \measuredangle(E', W)$, so $|\alpha| = 1 / \sin \measuredangle(E', W)$ and
			\aryst
			d(M_t u, M_t W) = |\alpha| \, d(M_t e', W_1) = \frac{ |M_t e'| \, \sin \measuredangle(E, W_1) }{ \sin \measuredangle(E', W) } .
			\earyst
			
			Let $v \in V$ be a unit vector and write $v = \beta e' + w$ with $w \in W$. Pairing with $u$ gives $\beta = \langle v, u \rangle / \langle e', u \rangle$, so $|\beta| \leq 1 / \sin \measuredangle(E', W)$ and $|w| \leq |v| + |\beta| \leq 2 / \sin \measuredangle(E', W)$. Hence, by (G2) at $(x', \bi'_{+}) \in \Gamma_0$ with this $t = \tau_k(x, \bi_{+}) \geq l$,
			\aryst
			|M_t v| \ \leq \ |\beta| \, |M_t e'| + |w| \, \| M_t|_{W} \| \ \leq \ \frac{ |M_t e'| + 2 \| M_t|_{W} \| }{ \sin \measuredangle(E', W) } \ \leq \ \frac{ 3 \, |M_t e'| }{ \sin \measuredangle(E', W) } .
			\earyst
			These estimates give 
			\begin{equation} 
				\label{eq full norm bound}
				\| M_t|_{V} \| \ \leq \ \frac{ 3 }{ \sin \measuredangle(E, W_1) } \, 2^{-\lambda_t(x, \bi_{+})} \ < \ \frac{3}{3 \cdot 2^{-l/4}} \, 2^{-N - l/2} \ = \ 2^{-N - l/4},
			\end{equation}
			where we use (G3) at $(x, \bi_{+}) \in \Gamma_0$ together with \eqref{eq tau bounds} to obtain the second inequality.
			
			The set $f_{x[t]} \big( \cal A_{m_0+1}(x', \bi'_{+}) \cap \Lambda \big) = x + M_t \big( \cal A_{m_0+1}(x', \bi'_{+}) \cap \Lambda - x' \big)$ contains $x$, and it is contained in $x + V_{m_0+1}(\bi_{+}) \subset x + V_{m+1}(\bi_{+})$. It has a diameter of at most $\| M_t|_{V} \| \diam(\cal A_{m_0+1}(x', \bi'_{+}) \cap \Lambda) < 2^{-N-l/4}$ by \eqref{eq full norm bound}, since $\diam \Lambda \leq 1$. If $d(x, \partial \cD_{\bi_{+}, N}(x)) \geq 2^{-N-l/4}$, this set is therefore contained in $\cD_{\bi_{+}, N}(x)$, i.e.\ $T_{n,k}(x, \bi_{+})$ is full. This is a contradiction.
		\end{proof}

		\section{Auxiliary measure-theoretic lemmas} \label{app miniscule}

		\begin{lemma}\label{lem:ball-compare}
			Let $\omega', \omega''$ be positive finite measures on $[0,1]$, and let
			\[
			\zeta := \omega' + \omega''
			\]
			be a probability measure on $[0,1]$. Let $0 < \eta < 1$ and $0 < r < 1$, and
			suppose that
			\begin{equation}\label{eq:hyp}
				|\omega''| := \omega''([0,1]) \;\le\; \eta^{4}.
			\end{equation}
			Then, writing $B(x,r) := [0,1] \cap (x-r, x+r)$ for the ball of radius $r$
			around $x$, we have
			\begin{equation}\label{eq:concl}
				\zeta\Bigl\{ x : \omega''\bigl(B(x,r)\bigr) > \eta^{1/2}\, \omega'\bigl(B(x,r)\bigr) \Bigr\}
				\;\le\; 2\eta^{7/2} + \eta^{4}.
			\end{equation}
			In particular, whenever $2\eta^{1/2} + \eta \le 1$ ,
			\[
			\zeta\Bigl\{ x : \omega''\bigl(B(x,r)\bigr) > \eta^{1/2}\, \omega'\bigl(B(x,r)\bigr) \Bigr\}
			\;\le\; \eta^{3}.
			\]
		\end{lemma}
		
		\begin{proof}
			Set
			\[
			E := \Bigl\{ x : \omega''\bigl(B(x,r)\bigr) > \eta^{1/2}\, \omega'\bigl(B(x,r)\bigr) \Bigr\}.
			\]
			For every $x \in E$ we have
			\begin{equation}\label{eq:bad}
				\omega'\bigl(B(x,r)\bigr) \;\le\; \eta^{-1/2}\, \omega''\bigl(B(x,r)\bigr),
			\end{equation}
			and \eqref{eq:bad} also holds trivially when either ball mass vanishes.
			
			By the Besicovitch covering theorem on the real line, there exists a
			countable subfamily of the balls $\{B(x_i,r)\}_{i}$, $x_i \in E$, which
			covers $E$ and whose covering multiplicity is at most $2$, i.e.\
			$\sum_i \mathbf 1_{B(x_i,r)}(y) \le 2$ for every $y \in \mathbb{R}$.
			Applying \eqref{eq:bad} at each center $x_i$, we obtain
			\aryst
			\omega'(E)
			&\le& \sum_i \omega'\bigl(E \cap B(x_i,r)\bigr)
			\le \sum_i \omega'\bigl(B(x_i,r)\bigr) \\
			&\le& \eta^{-1/2} \sum_i \omega''\bigl(B(x_i,r)\bigr)
			\le 2 \eta^{-1/2}\, \omega''([0,1])
			\le 2 \eta^{7/2},
			\earyst
			where the second last inequality uses the covering multiplicity. On the
			other hand, \eqref{eq:hyp} gives $\omega''(E) \le \omega''([0,1]) \le
			\eta^4$. Adding the two estimates yields
			\[
			\zeta(E) = \omega'(E) + \omega''(E) \le 2\eta^{7/2} + \eta^{4},
			\]
			which is \eqref{eq:concl}. 
		\end{proof}

		\begin{lemma}\label{lem:comparison}
			Let $\mu$ be a finite Borel measure on $\mathbb R^d$, and let
			$W=[-s/2,s/2]^d\subset W'=[-s'/2,s'/2]^d$ with $0<s<s'$. Put $N:=\lceil 2s'/s\rceil^{d}$.
			Then
			\begin{equation}\label{eq:master}
				\int_{\mathbb R^d}\frac{\mu(x+W')}{\mu(x+W)}\,d\mu(x)\ \le\ N\,\mu(\mathbb R^d),
			\end{equation}
			and consequently, for every $c_0>0$,
			\begin{equation}\label{eq:bad mass}
				\mu\bigl(\{x:\ \mu(x+W)\le c_0\,\mu(x+W')\}\bigr)\ \le\ N\,c_0\,\mu(\mathbb R^d).
			\end{equation}
		\end{lemma}
		
		\begin{proof}
			Since $W'=-W'$, Tonelli's theorem gives
			\[
			\int\frac{\mu(x+W')}{\mu(x+W)}\,d\mu(x)
			=\iint\frac{\mathbf 1_{W'}(y-x)}{\mu(x+W)}\,d\mu(y)\,d\mu(x)
			=\int\Bigl(\int_{y+W'}\frac{d\mu(x)}{\mu(x+W)}\Bigr)d\mu(y).
			\]
			Fix $y$ and subdivide the cube $y+W'$ into $N=\lceil 2s'/s\rceil^{d}$ congruent
			half-open subcubes $Q$, each of side $s'/\lceil 2s'/s\rceil\le s/2$. If $x\in Q$
			then $\|z-x\|_\infty\le s/2$ for all $z\in Q$, i.e. $Q\subset x+W$, whence
			$\mu(x+W)\ge\mu(Q)$ and $\int_Q \mu(x+W)^{-1}d\mu(x)\le 1$. Summing over the $N$
			subcubes yields $\int_{y+W'}\mu(x+W)^{-1}d\mu(x)\le N$, and \eqref{eq:master}
			follows. Finally, on the set in \eqref{eq:bad mass} one has
			$1\le c_0\,\mu(x+W')/\mu(x+W)$, so \eqref{eq:bad mass} follows from \eqref{eq:master}.
		\end{proof}

		\subsubsection*{Random dyadic partitions}
		\label{subsec:random partition}
		
		For $n \ge 1$ recall ${\mathcal{D}}_n$ denotes the standard dyadic partition of
		$\mathbb{R}^d$ into half-open cubes of side length $2^{-n}$, i.e.\ the cubes
		$[k_1 2^{-n}, (k_1+1) 2^{-n}) \times \cdots \times [k_d 2^{-n}, (k_d+1) 2^{-n})$,
		$k \in \mathbb{Z}^d$, and write $\mathcal{D}_n + t$ for its translate by
		$t \in \mathbb{R}^d$. For $c > 0$ let
		\[
		\partial_c( {\mathcal{D}}_n + t)
		:= \Bigl\{ x \in \mathbb{R}^d : \dist\bigl(x,\; \partial( {\mathcal{D}}_n+t)\bigr) < c\, 2^{-n} \Bigr\}
		\]
		be the $c\, 2^{-n}$-neighbourhood of $\partial(\mathcal{D}_n + t)$.
		
		\begin{lemma}\label{lem:avg}
			Let $\mu$ be a probability measure on $\mathbb{R}^d$. For all integers $C \ge 1$ and $N \ge 1$ there exists a translation
			$t \in [0,\, 2^{-CN})^d$ such that for all $c\in(0,1/2]$
			\begin{equation}\label{eq:avg}
				\frac{1}{N} \sum_{n = CN}^{(C+1)N - 1}
				\mu\bigl( \partial_c( {\mathcal{D}}_n + t) \bigr)
				\;\le\; 9d\sqrt{c} .
			\end{equation}
		\end{lemma}
		
		\begin{proof}
			Since $2^{-CN}/2^{-n} = 2^{n-CN}$ is an integer for every
			$n \in \{CN, \dots, (C+1)N - 1\}$, the box $T := [0,\, 2^{-CN})^d$ is a
			fundamental domain for each grid $\mathcal{D}_n$ in the block.
			
			Endow $T$ with Lebesgue measure. By Tonelli's theorem,
			\aryst
			\frac{1}{N} \sum_{n = CN}^{(C+1)N-1}
			\int_T \mu\bigl( \partial_c( {\mathcal{D}}_n + t) \bigr)\, dt
			=
			\int_{\mathbb{R}^d}
			\frac{1}{N} \sum_{n = CN}^{(C+1)N-1}
			\bigl| \bigl\{ t \in T : x \in \partial_c( {\mathcal{D}}_n + t) \bigr\} \bigr|
			\, d\mu(x) .
			\earyst
			Fix $x \in \mathbb{R}^d$ and $n \in \{CN, \dots, (C+1)N-1\}$. The point $x$
			lies in $\partial_c( {\mathcal{D}}_n + t)$ iff for some coordinate $i \in \{1, \dots, d\}$
			and some $k \in \mathbb{Z}$,
			\[
			\bigl| x_i - t_i - k\, 2^{-n} \bigr| < c\, 2^{-n}.
			\]
			For a fixed coordinate $i$, the set of $t_i \in [0, 2^{-n})$ satisfying this
			is one or two intervals of total length at most $2 c\, 2^{-n}$. Since $T$
			contains exactly $2^{n-CN}$ full periods of the grid of scale $2^{-n}$ in
			each coordinate, the set of $t_i \in [0, 2^{-CN})$ satisfying the condition
			has measure at most
			\[
			2 c\, 2^{-n} \cdot 2^{n - CN} = 2 c\, 2^{-CN},
			\]
			uniformly for $N$ in the block. Consequently, for fixed $x$ and $n$, the set
			of translations $t \in T$ with $x \in \partial_c(\mathcal{D}_n + t)$ has
			Lebesgue measure at most
			\[
			d \cdot \bigl( 2 c\, 2^{-CN} \bigr) \cdot
			\bigl( 2^{-CN} \bigr)^{d-1}
			= 2cd \cdot 2^{-CNd}
			= 2cd \cdot |T| .
			\]
			Hence, since $\mu$ is a probability measure,
			\[
			\frac{1}{N} \sum_{n = CN}^{(C+1)N-1}
			\int_T \mu\bigl( \partial_c( {\mathcal{D}}_n + t) \bigr)\, dt
			\;\le\; 2cd\, |T| .
			\]
			Dividing by $|T|$, the average over $t \in T$ of the scale-averaged boundary
			measure is at most $2cd$; therefore the probability of $t \in T$ for which
			\eqref{eq:avg} doesn't hold with $4d\sqrt{c}$ is less than $ \sqrt{c}/2$. Then we sum over $c\in\{4^{-m},\ m\in\N \}$, the total bad probability of $t\in T$ is less than
			\[ \sum_{m\geq 1}\sqrt{4^{-m}}/4\leq 1/4. \]
			Hence, there exists $t$ such that for all $c\in\{4^{-m},\ m\in\N \}$, \eqref{eq:avg} holds with $4d\sqrt{c}$. Hence for all $c\in(0,1/2]$, we have that \eqref{eq:avg} holds.
		\end{proof}

		The following lemma says that ball spreading is stable under nearby projections.
		\begin{lemma}\label{lem spreading nearby}
			Let $\zeta$ be a Borel probability measure on $\R^d$ whose support has diameter at most $D$. Let $H, H' \subset \R^d$ be hyperplanes and denote by $\zeta^{H}, \zeta^{H'}$ the push-forwards of $\zeta$ under $P^{H}$ and $P^{H'}$, viewed as measures on $\R$ by identifying the lines $H^{\perp}, (H')^{\perp}$ with $\R$. Let $n, l \geq 1$ be integers and $\eps > 0$, and suppose that
			\ary \label{eq nearby hyp}
			2 D \, \bd(H, H') \ \leq \ 2^{- n - l - 2}.
			\eary
			If $\zeta^{H}$ is $(n, l, \eps)$-ball spreading, then $\zeta^{H'}$ is $(n, l + 2, \eps + n^{-1})$-ball spreading.
		\end{lemma}
		
		\begin{proof}
			Let $u, u'$ be unit vectors spanning $H^{\perp}$ and $(H')^{\perp}$. For every $v \in \R^d$,
			\aryst
			\big| \| P^{H} v \| - \| P^{H'} v \| \big| = \big| | \langle v, u \rangle | - | \langle v, u' \rangle | \big| \leq \| v \| \, \| u - u' \| \leq 2 \, \| v \| \, \bd(H, H').
			\earyst
			Since $\operatorname{supp} \zeta$ has diameter at most $D$, for every $z, w \in \operatorname{supp} \zeta$,
			\begin{equation} \label{eq nearby comparison}
				\big| \| P^{H} (w - z) \| - \| P^{H'} (w - z) \| \big| \ \leq \ \delta, \qquad \delta := 2 D \, \bd(H, H') \leq 2^{- n - l - 2}.
			\end{equation}
			Writing $t = P^{H} z$ and $s = P^{H'} z$ and integrating \eqref{eq nearby comparison} in $w$, we obtain
			\ary \label{eq nearby balls}
			\zeta^{H}\big( B(t, r - \delta) \big) \ \leq \ \zeta^{H'}\big( B(s, r) \big) \ \leq \ \zeta^{H}\big( B(t, r + \delta) \big), \qquad r > 0.
			\eary
			For $s \in (H')^{\perp}$, denote
			\aryst
			K'(s) = \Big\{ 1 \leq k \leq n \,\Big|\, \zeta^{H'}\big( B(s, 2^{-k}) \big) \leq 2 \, \zeta^{H'}\big( B(s, 2^{-k-l-2}) \big) \Big\},
			\earyst
			and define $K(t)$ for $t \in H^{\perp}$ similarly, with window $l$ in place of $l + 2$. If $k + 1 \notin K(t)$ and $1 \leq k \leq n - 1$, then by \eqref{eq nearby balls} and $\delta \leq 2^{- n - l - 2}$,
			\aryst
			\zeta^{H'}\big( B(s, 2^{-k}) \big) & \geq & \zeta^{H}\big( B(t, 2^{-k} - \delta) \big) \geq \zeta^{H}\big( B(t, 2^{-(k+1)}) \big)  >  2 \, \zeta^{H}\big( B(t, 2^{-(k+1)-l}) \big)\\
			& \geq & 2 \, \zeta^{H}\big( B(t, 2^{-k-l-2} + \delta) \big) \geq 2 \, \zeta^{H'}\big( B(s, 2^{-k-l-2}) \big),
			\earyst
			that is, $k \notin K'(s)$. Hence
			\ary \label{eq nearby count}
			| K'(s) | \ \leq \ | K(t) | + 1.
			\eary
			Since $\zeta^{H}$ and $\zeta^{H'}$ are the push-forwards of the same measure $\zeta$ under $z \mapsto t$ and $z \mapsto s$, respectively, \eqref{eq nearby count} gives
			\aryst
			\zeta^{H'}\big( \big\{ s : | K'(s) | \geq \eps n + 1 \big\} \big) \ \leq \ \zeta^{H}\big( \big\{ t : | K(t) | \geq \eps n \big\} \big) \ \leq \ \eps,
			\earyst
			where the last inequality is the $(n, l, \eps)$-ball spreading of $\zeta^{H}$. Thus $\zeta^{H'}\big( \big\{ s : | K'(s) | < (\eps + n^{-1}) n \big\} \big) > 1 - \eps \geq 1 - (\eps + n^{-1})$, which means exactly that $\zeta^{H'}$ is $(n, l + 2, \eps + n^{-1})$-ball spreading.
		\end{proof}

		The following lemma transfers ball spreading to dyadic spreading.
		
		\begin{lemma}\label{lem ball to dyadic}
			Let $n, l \geq 1$ be integers, let $0 < \eta < 1/10$, and let $\zeta \in \cP(\R)$. Suppose that there is a Borel set $\Omega \subset \R$ with $\zeta(\Omega) > 1 - \eta$ such that every $x \in \Omega$ satisfies
			\begin{equation}\label{eq ball spreading hyp}
				\frac{1}{n} \big| \big\{ 1 \leq k \leq n : \zeta\big( B(x, 2^{-k}) \big) > 2\, \zeta\big( B(x, 2^{-k-l}) \big) \big\} \big| > 1 - \eta .
			\end{equation}
			Set
			\aryst
			l_0 := \lceil 3 \log_2 (1/\eta) \rceil + 2, \qquad L := l  + l_0 + 1,
			\earyst
			and assume that $n \geq ( l_0 + 1)/\eta$. Then there is a Borel set $I \subset [0, 1]$ with $\mathrm{Leb}(I) > 1 - \eta$ such that for every $h \in I$, the measure $\zeta$ is $(n, L, 3\eta, h)$-dyadic spreading; that is,
			\aryst
			\zeta \Big\{ x \in \R \ \Big| \ \frac{1}{n} \big| \big\{ 1 \leq j \leq n : \zeta\big( \cD^{h}_{j}(x) \big) > 2\, \zeta\big( \cD^{h}_{j+L}(x) \big) \big\} \big| > 1 - 3\eta \Big\} > 1 - 3\eta .
			\earyst
		\end{lemma}
		
		\begin{proof}
			A scale $j \in \{1, \cdots, n\}$ is called {\em $h$-centered} for $x \in \R$ if $d(2^{j}(x - h), \Z) > 2^{-l_0}$. Since $\cD^{h}_{j}(x)$ is an interval of length $2^{-j}$ whose boundary lies at distance $2^{-j} d(2^{j}(x-h), \Z)$ from $x$, in this case $B(x, 2^{-j-l_0}) \subset \cD^{h}_{j}(x)$; and $\cD^{h}_{j+L}(x)$, being an interval of length $2^{-(j+L)}$ containing $x$, is contained in $B(x, 2^{-j-L+1})$. Hence, if the ball inequality in \eqref{eq ball spreading hyp} holds at a scale $k \in \{l_0 + 1, \cdots, n\}$ and $j := k - l_0$ is $h$-centered for $x$, then, using $j + L = k + l + 1$,
			\begin{equation}\label{eq dyadic from ball}
				\zeta\big( \cD^{h}_{j}(x) \big) \geq \zeta\big( B(x, 2^{-k}) \big) > 2\, \zeta\big( B(x, 2^{-k-l}) \big) \geq 2\, \zeta\big( \cD^{h}_{j+L}(x) \big) .
			\end{equation}
			
			The key point is that the proportion of ill-centered scales can be computed exactly and uniformly in $x$ and $j$: for every $x \in \R$ and $j \geq 1$,
			\begin{equation}\label{eq bad h measure}
				\mathrm{Leb}\big\{ h \in [0, 1) : d\big( 2^{j}(x - h), \Z \big) \leq 2^{-l_0} \big\} = 2^{1 - l_0} ,
			\end{equation}
			because $h \mapsto 2^{j}(x - h) \bmod 1$ pushes the Lebesgue measure on $[0,1)$ forward to the Lebesgue measure on $\R/\Z$. %
			Set
			\aryst
			b(x, h) := \frac{1}{n} \big| \big\{ 1 \leq j \leq n : j \mbox{ is not $h$-centered for } x \big\} \big| .
			\earyst
			By \eqref{eq bad h measure} and Fubini's theorem, $\int_0^1 \int_{\R} b(x, h) \, d\zeta(x) \, dh = 2^{1 - l_0} \leq \eta^{3}/2$ by the choice of $l_0$. Writing $E_{h} := \{ x \in \R : b(x, h) > \eta \}$, Markov's inequality first in $x$ and then in $h$ gives $\int_0^1 \zeta(E_h) \, dh \leq \eta^2 / 2$, and the Borel set
			\aryst
			I := \big\{ h \in [0, 1] : \zeta(E_{h}) \leq \eta \big\} \quad \mbox{satisfies} \quad \mathrm{Leb}(I) \geq 1 - \eta/2 .
			\earyst
			Fix $h \in I$. For $x \in \Omega \setminus E_h$ (a set of $\zeta$-measure $> 1 - 2\eta$), the scales $k \in \{l_0 + 1, \cdots, n\}$ satisfying the ball inequality \eqref{eq ball spreading hyp} with $k - l_0$ being $h$-centered for $x$ number at least $(1-\eta)n - l_0 - \eta n \geq (1 - 3\eta)n$, since $n \geq (l_0 + 1)/\eta$; by \eqref{eq dyadic from ball}, each $j = k - l_0$ is then a dyadic spreading scale for $x$. This means that $\zeta$ is $(n, L, 3\eta, h)$-dyadic spreading for every $h \in I$.
		\end{proof}

		We can now state and prove the continuous version of \cref{thm Wu} used in \cref{subsec concludingtheproof}.
		
		\begin{lemma}\label{lem:ball-sumset}
			For every $0 < \gamma < 1/100$, every integer $l \geq 1$, every $0 < c_* < 1$ with $\gamma \leq c_*^{2}/400$ and every  $0 < \eps \leq \gamma / 200$, there exist $\delta = \delta(\gamma, l) > 0$ and $n_0 = n_0(\gamma, l, c_*, \eps) \geq 1$ such that for every integer $n \geq n_0$ the following holds. Let $\zeta \in \cP(\R)$ be $(n, l, \gamma)$-ball spreading. Let $A$ be a discrete subset of an interval of length $2^{\eps n}$ with
			\begin{equation}\label{eq mass on A}
				\zeta\Big( \bigcup_{a \in A} B(a, 2^{-n}) \Big) \geq c_* .
			\end{equation}
			Let $B$ be a subset of an interval of length $2^{\eps n}$, and for every $a \in A$, let $B_{a} \subset B$ be non-empty. Then, writing $r = 2^{-n}$:
			\enmt
			\item \label{itm sumset big} if $ N_{r}(B) \leq 2^{(1 - 12\gamma) n}$ and $N_{r}(B_{a}) \geq N_{r}(B)^{1 - \delta}$ for every $a \in A$, then
			\aryst
			N_{r}\Big( \bigcup_{a \in A} (a + B_{a}) \Big) \geq N_{r}(B)^{1 + \delta} ;
			\earyst
			\item \label{itm sumset small} if $N_{r}(B) \leq 2^{\delta n}$, then
			\aryst
			N_{r}\Big( \bigcup_{a \in A} (a + B_{a}) \Big) \geq 2^{\delta n} .
			\earyst
			\eenmt
		\end{lemma}
		
		\begin{proof}
			Let $\delta_0 := \delta(48\gamma, L) > 0$ be the gain given by \cref{thm Wu} applied with $48\gamma$ in place of $\gamma$ and $L$ in place of $l$; here and below, $L = l + O_\gamma(1)$ is the window given by \cref{lem ball to dyadic}. We will prove the lemma with $\delta := \delta_0 / 4$, which depends only on $\gamma$ and $l$.
			
			\smallskip
			\noindent{\it Step 0 (reduction to unit support).} Set
			\aryst
			m := \lceil \eps n \rceil, \qquad n_1 := n + m, \qquad r_1 := 2^{- n_1} = 2^{- m} r, \qquad S(x) := 2^{- m} x,
			\earyst
			and $\zeta_1 := S_* \zeta$, $A_1 := S(A)$, $B_1 := S(B)$ and $B_{1, a_1} := S(B_{S^{-1} a_1})$ for $a_1 \in A_1$. Translating $(\zeta, A)$ together and $B$ separately beforehand, we may assume that both intervals in the statement equal $[0, 2^{\eps n}]$; then $A_1, B_1 \subset [0, 1]$. Note that $N_{r_1}(S(E)) = N_r(E)$ for every set $E$; hence
			\begin{equation} \label{eq scaled envelope}
				\begin{split}
					&N_{r_1}(B_1) \leq 2^{(1 - 12\gamma) n} = 2^{(1 - 12\gamma)(n_1 - m)} \leq 2^{(1 - 12\gamma) n_1} \quad \mbox{in case \eqref{itm sumset big}},\\
					& N_{r_1}(B_1) \leq 2^{\delta n} \leq 2^{\delta n_1} \quad \mbox{in case \eqref{itm sumset small}},	
				\end{split}
			\end{equation}
			and
			\ary \label{eq scaled charging}
			\zeta_1\Big( \bigcup_{a_1 \in A_1} B(a_1, r_1) \Big) = \zeta\Big( \bigcup_{a \in A} B(a, 2^{-n}) \Big) \geq c_*.
			\eary
			We claim that $\zeta_1$ is $(n_1, l, 2\gamma)$-ball spreading. Indeed, for $j \in \{1, \cdots, n_1\}$ and $x \in \R$,
			\aryst
			\zeta_1\big( B(Sx, 2^{-j}) \big) = \zeta\big( B(x, 2^{-(j - m)}) \big),
			\earyst
			so a bad scale $j > m$ of $\zeta_1$ at $Sx$ corresponds to a bad scale $j - m \in \{1, \cdots, n\}$ of $\zeta$ at $x$. For $x$ in the good set $\Omega$ with $\zeta(\Omega) > 1 - \gamma$ given by the ball spreading of $\zeta$, the bad scales of $\zeta_1$ at $Sx$ therefore number less than $\gamma n + m \leq 2\gamma n_1$, as required.
			
			\smallskip
			\noindent{\it Step 1 (dyadic spreading at a good shift).} By \cref{lem ball to dyadic} applied to $\zeta_1$ with $\eta = 2\gamma$ (its threshold $n_1 \geq (l_0 + 1)/(2\gamma)$ holds for $n \geq n_0(\gamma, l, \eps)$), there is a Borel set $I \subset [0, 1]$ with $\mathrm{Leb}(I) > 1 - 2\gamma$ such that $\zeta_1$ is $(n_1, L, 6\gamma, h)$-dyadic spreading for every $h \in I$. 
			
			Set {$N := n_1 - 1$ and $r_2 := 2^{-N} = 2 r_1$}, so that level-$N$ dyadic cells have length $r_2$; this is the finest level at which the averaging below sees radius-$r_1$ balls. For $x, a \in \R$, the point $x$ lies in $\cD^{h}_{N}(a)$ if and only if the grid $h + 2^{-N}\Z$ does not separate $x$ from $a$; as $h$ ranges over $[0, 1)$ the grid runs through $2^{N}$ full periods, and in each period the separating shifts form an interval of length $\min\{ |x - a|, 2^{-N} \}$. Hence, with $a = a(x)$ a nearest point of $A_1$ to $x$,
			\aryst
			\int_{0}^{1} \mathbf{1}_{\cup_{a_1\in A_1}\cD^{h}_{N}(a_1)}(x) \, dh = \big( 1 - 2^{N} |x - a| \big)_{+} \geq \frac{1}{2} \, \mathbf{1}_{\cup_{a_1\in A_1}B(a_1, 2^{-N-1})}(x).
			\earyst
			Integrating against $\zeta_1$ and using \eqref{eq scaled charging} gives
			\aryst
			\int_{0}^{1} \zeta_1 \Big( \bigcup_{a_1 \in A_1} \cD^{h}_{N}(a_1) \Big) \, dh \geq \frac{c}{2} .
			\earyst
			Since the integrand is at most $1$ and $\mathrm{Leb}([0,1] \setminus I) \leq 2\gamma \leq c/4$, there is $h \in I$ with
			\ary \label{eq good h}
			\zeta_1 \Big( \bigcup_{a_1 \in A_1} \cD^{h}_{N}(a_1) \Big) \geq \frac{c}{4} .
			\eary
			Fix such an $h$ and abbreviate $\cD_j = \cD^{h}_{j}$. %
			
			\smallskip
			\noindent{\it Step 2 (discretization and Wu's theorem).} For $D \in \cD_{N}$ let $x_D$ be its left endpoint, and set
			\aryst
			\eta := \sum_{D \in \cD_{N}} \zeta_1(D) \, \delta_{x_D}, \qquad A' := \{ x_D : D \in \cD_{N}, \ D \cap A_1 \neq \emptyset \}, \qquad \widetilde B_{x_D} := \bigcup_{a_1 \in A_1 \cap D} B_{1, a_1}.
			\earyst
			By \eqref{eq good h} we have $\eta(A') \geq c/4$; $\widetilde B_{a'}$ is a non-empty subset of $B_1$, and in case \eqref{itm sumset big} it satisfies $N_{r_1}(\widetilde B_{a'}) \geq N_{r_1}(B_1)^{1-\delta}$ (it contains $B_{1, a_1}$ for some $a_1 \in A_1 \cap D$). Let $G_1$ be the set of $(n_1, L, 6\gamma, h)$-dyadic spreading points of $\zeta_1$, so $\zeta_1(G_1) > 1 - 6\gamma$, and set $A'' := A' \cap \{ x_D : D \cap G_1 \neq \emptyset \}$; then $\eta(A'') \geq c/4 - 6\gamma \geq c/8$. If $D \in \cD_{N}$ meets $G_1$, then every $x \in G_1 \cap D$ shares its dyadic atoms with $x_D$ at all levels $j \leq N$, and the dyadic masses for $\zeta_1$ and $\eta$ agree at those levels; hence every point of $A''$ is $(N, L, 12\gamma)$-dyadic spreading with respect to $\eta$ (cf. \cite[proof of Corollary 5.4, Step 1]{LPTX}). Discretize the sets at scale $r_2=2^{-N}$: writing $E^{(r_2)} := E(r_2) \cap r_2 \Z$, set $D := B_1^{(r_2)}$ and $D_{a'} := \widetilde B_{a'}^{(r_2)}$. By \cite[Lemma 5.3]{LPTX},
			\begin{equation}
				\label{eq envelope discrete}
				|D| \leq 6 N_{r_2}(B_1), \qquad |D_{a'}| \geq N_{r_2}(\widetilde B_{a'}) \ \ (a' \in A''), \qquad N_{r_2}(E) \geq \frac{1}{6} |E^{(r_2)}|.
			\end{equation}
			
			In case \eqref{itm sumset big}, { if $N_{r_1}(B_1) \leq 2^{\delta n/2}$ then case \eqref{itm sumset small} applies and gives the conclusion, so we may assume} $N_{r_1}(B_1) > 2^{\delta n / 2} \geq 18^{2/\delta_0}$ for $n \geq n_0(\gamma, l)$ (recall $\delta = \delta_0 / 4$). Then by \eqref{eq scaled envelope} and \eqref{eq envelope discrete}, for $n \geq n_0(\gamma, l, c, \eps)$ (the one-scale shift $n_1 = N + 1$ and the $O(1)$ shifts below are absorbed since $\gamma N \gg 1$),
			\aryst
			|D| \leq 6 N_{r_2}(B_1) \leq 6 N_{r_1}(B_1) \leq 6 \cdot 2^{(1 - 12\gamma)(N+1)} \leq 2^{(1 - 48\gamma) N}, \\
			\qquad |D_{a'}| \geq N_{r_2}(\widetilde B_{a'}) \geq \frac{1}{2} N_{r_1}(B_1)^{1-\delta} \geq \frac{1}{2} \big( |D| / 6 \big)^{1-\delta} \geq |D|^{1 - \delta_0},
			\earyst
			where the last inequality uses $N_{r_1}(B_1) \geq |D|/6$ and $|D|^{\delta_0 - \delta} \geq |D|^{3 \delta_0 / 4} \geq 18^{3/2} > 12$ (as $|D| \geq N_{r_1}(B_1) \geq 18^{2/\delta_0}$).
			We may therefore apply \cref{thm Wu} (with $48\gamma,c_*/8,L,N$ in place of $\gamma,c_*,l,n$): since every point of $A''$ is $(N, L, 12\gamma)$-dyadic spreading, its hypothesis holds with $48\gamma$, and hence
			\aryst
			\Big| \bigcup_{a' \in A''} (a' + D_{a'}) \Big| \geq |D|^{1 + \delta_0}.
			\earyst
			Since $D_{a'} \subset \widetilde B_{a'}^{(r_2)}$, the set $\bigcup_{a'}(a' + D_{a'})$ is contained in $\big(\bigcup_{a'}(a' + \widetilde B_{a'})\big)^{(r_2)}$, so by \eqref{eq envelope discrete},
			\aryst
			N_{r_2}\Big( \bigcup_{a' \in A''} (a' + \widetilde B_{a'}) \Big) \geq \frac{1}{6} |D|^{1+\delta_0} \geq \frac{1}{6} N_{r_2}(B_1)^{1+\delta_0}.
			\earyst
			Finally, each $a_1 \in A_1 \cap D$ satisfies $|a_1 - x_D| \leq r_2$, so $a' + \widetilde B_{a'} \subset \big(\bigcup_{a_1 \in A_1}(a_1 + B_{1, a_1})\big)(r_2)$; applying \cite[Lemma 5.3]{LPTX} once more and using $N_{r_1}(E) \geq N_{r_2}(E)$ and $N_{r_2}(B_1) \geq \frac{1}{2} N_{r_1}(B_1)$,
			\begin{equation}
				\label{eq sumset bound}
				N_{r_1}\Big( \bigcup_{a_1 \in A_1} (a_1 + B_{1, a_1}) \Big) \geq N_{r_2}\Big( \bigcup_{a_1 \in A_1} (a_1 + B_{1, a_1}) \Big) \geq \frac{1}{72} N_{r_1}(B_1)^{1+\delta_0} \geq N_{r_1}(B_1)^{1+\delta_0/4}.
			\end{equation} 
			
			\smallskip
			\noindent{\it Step 3.} It remains to treat \eqref{itm sumset small}, where $N_{r_1}(B_1) < 2^{\delta n} \leq 2^{\delta n_1}$ by \eqref{eq scaled envelope}. Then by \eqref{eq envelope discrete}, for $n \geq n_0(\gamma, l, c, \eps)$,
			\aryst
			|D| \leq 6 \cdot 2^{\delta n_1} \leq 2^{\delta_0 N} ,
			\earyst
			and $D_{a'}$ is non-empty for every $a' \in A''$. Applying the second case of \cref{thm Wu} (with $48\gamma$, $c_*/8$ and $L$ as above, and with $N$ in place of $n$) to the digit set $A''$ and the pieces $D_{a'}$ gives
			\aryst
			\Big| \bigcup_{a' \in A''} (a' + D_{a'}) \Big| \geq 2^{\delta_0 N} ,
			\earyst
			and the same conversion as in Step 2 yields
			\ary \label{eq sumset small bound}
			N_{r_1}\Big( \bigcup_{a_1 \in A_1} (a_1 + B_{1, a_1}) \Big) \geq \frac{1}{18} \cdot 2^{\delta_0 N} \geq 2^{\delta_0 N/2}
			\eary
			for $n$ large. %
			
			\smallskip
			\noindent{\it Step 4 (scaling back).} Since $N_{r_1}(S(E)) = N_r(E)$ for every set $E$, scaling back by $S^{-1}$ transfers \eqref{eq sumset bound} and \eqref{eq sumset small bound} verbatim to the original scale, and the lemma follows (using $\delta \leq \delta_0/4$ and $N = n_1 - 1 = n + m - 1 \geq n/2$ for $n \geq n_0(\gamma)$).
		\end{proof}

		\bibliography{bibfile}

	\end{document}